\documentclass{article}
\usepackage{amsmath, amsthm, amssymb,amsfonts,mathrsfs,amssymb}
\usepackage[all]{xy}

\newtheorem{thm}{Theorem}[section]
\newtheorem{cor}[thm]{Corollary}
\newtheorem{lem}[thm]{Lemma}
\newtheorem{prop}[thm]{Proposition}
\newtheorem{conj}[thm]{Conjecture}

\theoremstyle{definition}
\newtheorem{defn}[thm]{Definition}

\theoremstyle{remark}
\newtheorem{rem}[thm]{Remark}

\title{Differential Equivariant $K$-Theory}
\author{Michael L. Ortiz}

\begin{document}

\maketitle

\begin{abstract}
Following Hopkins and Singer, we give a definition for the differential equivariant $K$-theory of a smooth manifold acted upon by a finite group. The ring structure for differential equivariant $K$-theory is developed explicitly. We also construct a pushforward map which parallels the topological pushforward in equivariant $K$-theory.  An analytic formula for the pushforward to the differential equivariant $K$-theory of a point is conjectured, and proved in the boundary case, in the case of a free action, and for ordinary differential $K$-theory in general. The latter proof is due to K. Klonoff.
\end{abstract}

\section{Introduction}

In differential geometry, one approaches the real cohomology of a smooth manifold by way of representative differential forms, thus getting a handle on global information by means of local data. In a similar way, a generalized differential cohomology theory gives a way of representing \emph{integral} generalized cohomology classes by differential geometric objects. The ordinary differential cohomology groups have appeared in the literature as the groups of Cheeger-Simons differential characters \cite{ChS} and the smooth Deligne cohomology groups \cite{De}, and the foundations for generalized differential cohomology were laid by Hopkins and Singer in \cite{HS}.

Of particular interest is differential $K$-theory. One motivation for studying differential $K$-theory arises in theoretical physics: in Type II superstring theory, a Ramond-Ramond field carries the global information of a $K$-theory class together with the locality of a field (cf. the papers \cite{FH} of Freed and Hopkins and \cite{MW} of Moore and Witten). An explicit definition for differential $K$-theory is sketched in \cite{HS}, and another sketch is given by Freed in \cite{F2}. In \cite{K}, Klonoff gives an explicit and detailed account of differential $K$-theory; after making a definition based on \cite{HS}, he develops a geometric description for the degree zero theory from a sketch in \cite{F2} that is more in the spirit of Lott's paper \cite{Lo} on $\mathbb{R}/\mathbb{Z}$ index theory. This geometric picture has also been developed by Simons and Sullivan in \cite{SS}, and an alternate definition for differential $K$-theory is given by Bunke and Schick in \cite{BS}.

In this paper, we formulate a definition of differential \emph{equivariant} $K$-theory for a finite group. Differential equivariant $K$-theory has also been defined by Szabo and Valentino in \cite{SV}; we take a different tack, but the two definitions are equivalent, as will be explained. Essential to our definition is the localization theorem of Atiyah and Segal (cf. \cite{ASe}). If $\Gamma$ is a finite group acting on a manifold $X$, then this theorem implies that $K_\Gamma^\bullet(X) \otimes \mathbb{R}$ is isomorphic to a ring $H^\bullet_\Gamma(X;\mathbb{R})$ of cohomology classes with real coefficients, the cohomology ring of a complex $(\Omega^\bullet_\Gamma(X),d_\Gamma)$ of differential forms on the $\Gamma$-fixed point sets. The image of an equivariant $K$-theory class under this isomorphism is its \emph{equivariant Chern character}. In Section 2, we specify a family of classifying spaces for the equivariant $K$-theory groups, together with differential form representatives for the universal equivariant Chern characters. This allows one to compare representatives of $K$-theory classes with representatives of de Rham cohomology classes on the level of differential forms, and thus to make the desired geometric refinement of equivariant $K$-theory.

The differential equivariant $K$-theory group $\check{K}_\Gamma^{-i}(X)$ fits into the commutative diagram
\begin{displaymath}
\xymatrix{
\check{K}^{-i}_\Gamma(X) \ar[r]^\omega \ar[d]^c & \Omega^{-i}_\Gamma(X)_\textrm{cl} \ar[d]\\
K_\Gamma^{-i}(X) \ar[r] & H^{-i}_\Gamma(X;\mathbb{R}) }
\end{displaymath}
where $\Omega^{-i}_\Gamma(X)_\textrm{cl} \subset \Omega^{-i}_\Gamma(X)$ is the subring of closed forms. If $\check{x} \in \check{K}^{-i}_\Gamma(X)$, then it is natural to refer to $c(\check{x})$ as its \emph{characteristic class} and to $\omega(\check{x})$ as its \emph{curvature}. Now, if $X$ is compact and spin, then there exists a natural pushforward map
\begin{displaymath}
\wp_\Gamma: K_\Gamma^0(X) \longrightarrow K_\Gamma^{-\textrm{dim}\,X}(\textrm{pt})
\end{displaymath}
A definition of pushforward from the ordinary differential $K$-theory of a Riemannian spin manifold $X$ to that of a point was sketched in \cite{HS} and constructed in \cite{K}, as well as in \cite{BS}, and in this paper we generalize it to the equivariant setting, thus defining
\begin{displaymath}
\wp_\Gamma: \check{K}_\Gamma^0(X) \longrightarrow \check{K}_\Gamma^{-\textrm{dim}\,X}(\textrm{pt})
\end{displaymath}
If $i$ is even, then
\begin{displaymath}
\check{K}^{-i}_\Gamma(\textrm{pt}) \cong R(\Gamma) \cong K^{-i}_\Gamma(\textrm{pt})
\end{displaymath}
where $R(\Gamma)$ is the representation ring for $\Gamma$. Thus, if $X$ is even-dimensional, then the pushforward of a class in $\check{K}^0_\Gamma(X)$ is given by the (topological) pushforward of its characteristic class. On the other hand, if $i$ is odd, then
\begin{displaymath}
\check{K}^{-i}_\Gamma(\textrm{pt}) \cong (R(\Gamma) \otimes \mathbb{R})/R(\Gamma)
\end{displaymath}
whereas $K_\Gamma^{-i}(\textrm{pt})$ is trivial. It follows that, if $X$ is odd-dimensional, then the pushforward takes values in the torus $(R(\Gamma) \otimes \mathbb{R})/R(\Gamma)$. Our pushforward map thus constitutes a \emph{geometric refinement} of the topological pushforward map. It should be able to detect secondary geometric invariants in the case that $X$ is odd-dimensional.

Now, by the Index Theorem (cf. \cite{AS1}, \cite{ASe2}), the ordinary $K$-theoretic pushforward of a class may be computed as the $\Gamma$-index of a twisted Dirac operator. In this paper, we conjecture that this theorem generalizes to differential equivariant $K$-theory in the following way. There is a natural way of associating to each class $\check{x} \in \check{K}^0_\Gamma(X)$ a complex equivariant vector bundle $V \rightarrow X$ with connection. If $X$ is odd-dimensional, then one defines the reduced eta invariant $\xi_\Gamma(D_V)$ of the twisted spinor Dirac operator $D_V$, which it is most natural to view as an element of $(R(\Gamma) \otimes \mathbb{R})/R(\Gamma)$. Our conjecture is that
\begin{displaymath}
\wp_\Gamma(\check{x}) = \xi_\Gamma(D_V)
\end{displaymath}
This has been proved in certain special cases. We have proved it in the case that $(X,V)$ is a boundary in an appropriate sense. Furthermore, Klonoff has given a reduction to the boundary case for $\Gamma$ trivial, thus proving the conjecture in ordinary differential $K$-theory, and we have used this to prove the conjecture in the case that $\Gamma$ acts freely.

Briefly, the contents of this paper are as follows. We begin in Section 2 by recalling some basic facts and fixing notation, classifying spaces, and representatives, and proceed to give the main definition. In Section 3, we present two useful alternative definitions of differential equivariant $K$-theory for compact spaces, and we develop the ring structure in Section 4. The pushforward map is constructed in Section 5, and the analytic formula for the pushforward to the differential equivariant $K$-theory of a point is conjectured in Section 6.  This conjecture is proved in the special cases referred to above.

I would like to thank my thesis adviser, Dan Freed, for his many hours of patient guidance, and I would also like to acknowledge K. Gomi and K. Klonoff for many helpful discussions.


\section{Differential Equivariant $K$-Theory}

Throughout this paper, we work with smooth manifolds, smooth maps, and smooth actions, unless otherwise stated.  Thus, the statement that $X$ is a $\Gamma$-manifold, where $\Gamma$ is some group, will be taken to mean that $X$ is a smooth manifold with a smooth $\Gamma$-action. $\Gamma$ is always taken to be a finite group.

\subsection{Preliminaries}

Recall the following localization theorem of Atiyah and Segal for equivariant $K$-theory (cf. \cite{ASe}):

\begin{thm} \label{thm:localization} Let $\Gamma$ be a finite group. Suppose $X$ is a $\Gamma$-manifold.  Then there is a natural isomorphism
\begin{displaymath}
K_\Gamma^\bullet(X) \otimes \mathbb{C} \stackrel{\sim}{\longrightarrow} \left( \bigoplus_{g \in \Gamma} K^\bullet(X^g) \otimes \mathbb{C} \right)^\Gamma
\end{displaymath}
where $X^g \subset X$ is the space fixed by the action of $g \in \Gamma$. \end{thm}

\noindent This isomorphism is defined in the following way. Each class in $K_\Gamma^0(X)$ may be represented by complex equivariant vector bundle $V \rightarrow X$ with a $\mathbb{Z}/2\mathbb{Z}$-grading preserved by the $\Gamma$-action. Fix $g \in \Gamma$ for the time being, and let $V_g \rightarrow X^g$ be the restriction of $V$ to $X^g$. The group element $g$ acts on the fibers of $V_g$, so $V_g$ decomposes into a direct sum of bundles $V_{g,\lambda} \rightarrow X^g$ of $g$-eigenspaces with $g$-eigenvalue $\lambda$.  Notice that $V_{g,\lambda}$ is an equivariant vector bundle under action by the centralizer of $g$.  The isomorphism of Theorem~\ref{thm:localization} is given by
\begin{displaymath}
\lbrack V \rbrack \mapsto \sum_{g \in \Gamma} \sum_{\lambda} \lambda\lbrack V_{g,\lambda} \rbrack
\end{displaymath}
This defines the isomorphism for degree $0$; the isomorphism for degree $-1$ is constructed by applying similar arguments to $S^1 \times X$. This suffices to establish the isomorphism for all degrees, as we know by Bott periodicity that
\begin{displaymath}
K_\Gamma^{-i}(X) \cong K_\Gamma^{-i-2}(X) \qquad \textrm{and} \qquad K^{-i}(X^g) \cong K^{-i-2}(X^g)
\end{displaymath}

Let
\begin{displaymath}
\pi K_\mathbb{C} \equiv K^\bullet(\textrm{pt}) \otimes \mathbb{C} = \mathbb{C}\lbrack\lbrack u,u^{-1} \rbrack\rbrack
\end{displaymath}
where $u$ is the Bott class of degree two, and define
\begin{equation}
H^{-i}_\Gamma(X;\mathbb{C}) \equiv \left( \bigoplus_{g \in \Gamma} H(X^g;\pi K_\mathbb{C})^{-i} \right)^\Gamma
\end{equation}
for all $i \geq 0$. (Warning: this is non-standard notation.) Here, $-i$ denotes the \emph{total} degree.  For instance, a class $c\,u^{-1}$, where $c$ is a collection of complex cohomology classes $c_g$ of degree two on the fixed-point sets $X^g$, is of degree zero.

To a $\mathbb{Z}/2\mathbb{Z}$-graded equivariant vector bundle $V \rightarrow X$ we associate its equivariant Chern character $\textrm{ch}_\Gamma^0(V) \in H^0_\Gamma(X;\mathbb{C})$, defined as follows. Fix $g \in \Gamma$ arbitrary, and decompose $V_g$ into a direct sum of $g$-eigenbundles $V_\lambda$ as before.  Then the equivariant Chern character is defined as
\begin{equation}
\textrm{ch}_g^0(V) \equiv \sum \lambda\,\textrm{ch}(V_\lambda) \in H(X^g;\pi K_\mathbb{C})^0
\end{equation}
\begin{equation}
\textrm{ch}_\Gamma^0(V) \equiv \sum\textrm{ch}_g^0(V) \in H^0_\Gamma(X;\mathbb{C})
\end{equation}
where $\textrm{ch}$ denotes the ordinary Chern character. Like the ordinary Chern character, the equivariant Chern character respects direct sums and tensor products, hence descends to a ring homomorphism on $K_\Gamma^0(X)$.  The Chern characters of nonzero degree are defined similarly.  In particular, if a class $x \in K_\Gamma^{-i}(X)$ is represented by an equivariant vector bundle $V \rightarrow S^i \times X$, then its equivariant Chern character is defined as
\begin{equation}
\textrm{ch}_g^{-i}(x) \equiv \int_{S^i \times X^g/X^g} \textrm{ch}_g^0(V) \in H(X^g;\pi K_\mathbb{C})^{-i}
\end{equation}
\begin{equation}
\textrm{ch}_\Gamma^{-i}(x) \equiv \sum \textrm{ch}_g^{-i}(x) \in H^{-i}_\Gamma(X;\mathbb{C})
\end{equation}
The equivariant Chern character extends linearly to a map
\begin{displaymath}
\textrm{ch}_\Gamma^{-i}: K_\Gamma^{-i}(X) \otimes \mathbb{C} \stackrel{\sim}{\longrightarrow} H^{-i}_\Gamma(X;\mathbb{C}) \qquad i \in \mathbb{N}
\end{displaymath}
which is an isomorphism by Theorem~\ref{thm:localization}.

The complex algebra $K_\Gamma^\bullet(X) \otimes \mathbb{C}$ has a canonical real structure given by conjugation on $\mathbb{C}$, and this induces a real structure on $H^\bullet_\Gamma(X;\mathbb{C})$.  This real structure may be seen explicitly as follows.  Notice first that the fixed-point set for $g$-action coincides with that for $g^{-1}$-action.  The ring $H^\bullet_\Gamma(X;\mathbb{C})$ is generated over $\mathbb{C}$ by the equivariant Chern characters, and the $g$-component of the Chern character of an equivariant $K$-theory class is the complex conjugate of the $g^{-1}$-component.  It follows that the induced real structure is given as $c \mapsto \bar{c}$, where $(\bar{c})_g = \overline{c_{g^{-1}}}$.  We denote the real subspace of $H^\bullet_\Gamma(X;\mathbb{C})$ thus determined by $H^\bullet_\Gamma(X;\mathbb{R})$.  Let
\begin{displaymath}
\Omega^\bullet_\Gamma(X) \subset \left( \bigoplus_{g \in \Gamma}  \Omega(X^g;\pi K_\mathbb{C})^\bullet \right)^\Gamma
\end{displaymath}
be the ring of forms $\omega$ satisfying $\omega_{g^{-1}} = \overline{\omega_g}$.  Then $H^\bullet_\Gamma(X;\mathbb{R})$ is the cohomology of the differential complex $\left( \Omega^\bullet_\Gamma(X) , d_\Gamma \right)$, where
\begin{displaymath}
(d_\Gamma\omega)_g = d\omega_g
\end{displaymath}

We may obtain differential form representatives for the equivariant Chern character via the Chern-Weil method.  Suppose, for instance, that $V \rightarrow X$ is a complex graded equivariant vector bundle with $\Gamma$-invariant connection $\nabla$.  Let $\Omega$ denote the curvature of this connection, conceived of as a $2$-form taking values in $\textrm{End}(V)$, and let $\Omega_g$ denote its restriction to $X^g$.  Define
\begin{equation}
\omega_g(V;\nabla) \equiv \textrm{tr}_\textrm{s}\left(g \cdot \textrm{exp}\left\{\frac{i}{2\pi} \Omega_g \, u^{-1}\right\}\right) \in \Omega(X^g;\pi K_\mathbb{C})^0
\end{equation}
\begin{equation}
\omega(V;\nabla) \equiv \sum \omega_g(V,\nabla) \in \Omega_\Gamma^0(X;\mathbb{C})
\end{equation}
This form is a representative for $\textrm{ch}_\Gamma^0(V)$.  If $\Gamma$-action and $\nabla$ are \emph{compatibly unitary}---that is, if there exists a $\Gamma$-invariant hermitian metric on $V$ that is compatible with $\nabla$---then $\omega(V;\nabla)$ is real under the real structure just defined.  If $\nabla,\nabla'$ are two such connections on $V$, then we define the Chern-Simons form
\begin{displaymath}
\textrm{CS}_\Gamma(V;\nabla,\nabla') \in \Omega^{-1}_\Gamma(X)
\end{displaymath}
as follows.  Let $\nabla_t$ be the smooth path in the affine space of connections for $V$ defined by
\begin{displaymath}
\nabla_t = \nabla + t(\nabla' - \nabla)
\end{displaymath}
and let $\overline{\nabla}$ be the connection for $I \times V \rightarrow I \times X$ (where $I = \lbrack0,1\rbrack$) given by
\begin{displaymath}
\overline{\nabla} = \nabla_t + dt\,\partial_t
\end{displaymath}
Define
\begin{equation}
\textrm{CS}_g(V;\nabla,\nabla') \equiv \int_{I \times X^g/X^g} \omega_g\left(I \times V;\overline{\nabla}\right)
\end{equation}
\begin{equation}
\textrm{CS}_\Gamma(V;\nabla,\nabla') \equiv \sum \textrm{CS}_g(V;\nabla,\nabla')
\end{equation}
Then
\begin{displaymath}
d_\Gamma \textrm{CS}_\Gamma(V;\nabla,\nabla') = \omega(V;\nabla') - \omega(V;\nabla)
\end{displaymath}

We shall sometimes omit the superscript $0$ in our notation for the degree zero equivariant Chern characters, when it seems that no confusion will thus result.

Let $\mathcal{H} = \mathcal{H}^0 \oplus \mathcal{H}^1$ be a complex separable infinite-dimensional $\mathbb{Z}/2\mathbb{Z}$-graded Hilbert space, with $\mathcal{H}^0 = \mathcal{H}^1$ as ungraded Hilbert spaces. Let $W_\Gamma$ be a finite-dimensional $\Gamma$-representation in whose decomposition each irreducible representation appears once, and give it a $\Gamma$-invariant inner product. Set
\begin{displaymath}
\mathcal{H}_\Gamma = \mathcal{H} \otimes W_\Gamma
\end{displaymath}
Thus, $\mathcal{H}_\Gamma$ is a graded Hilbert space and a $\Gamma$-representation in which each irreducible $\Gamma$-representation appears infinitely many times.

Next, let $\mathbb{C}\textrm{l}_i$ be the complex Clifford algebra generated over $\mathbb{C}$ by $\lbrace e_1,\dots,e_i\rbrace$ obeying the relations
\begin{displaymath}
e_je_k+e_ke_j = -2\delta_{j,k}
\end{displaymath}
Then $\mathcal{H}_\Gamma \otimes \mathbb{C}\textrm{l}_i$ is a graded $\mathbb{C}\textrm{l}_i$-module.  (Here and elsewhere, the tensor product is the \emph{graded} tensor product.) Let
\begin{displaymath}
\mathfrak{F}_\Gamma^{-i} \subset \textrm{End}(\mathcal{H}_\Gamma \otimes \mathbb{C}\textrm{l}_i)
\end{displaymath}
be the space of odd skew-adjoint Fredholm operators which commute (in the graded sense) with $\mathbb{C}\textrm{l}_i$-action.  Then $\mathfrak{F}_\Gamma^{-i}$ is a smooth Banach $\Gamma$-manifold, with $\Gamma$-action given by conjugation.

Let $\Omega \mathfrak{F}_\Gamma^{-i}$ denote the space of smooth maps from $I = \lbrack0,1\rbrack$ to $\mathfrak{F}_\Gamma^{-i}$ which map the boundary $\lbrace0,1\rbrace$ into $\mathfrak{F}_\Gamma^{-i}\cap\textrm{GL}(\mathcal{H}_\Gamma)$.  The space $\Omega \mathfrak{F}_\Gamma^{-i}$ carries the compact-open topology.  We construct a map
\begin{displaymath}
\phi_{i+1}: \mathfrak{F}_\Gamma^{-i-1} \longrightarrow \Omega \mathfrak{F}_\Gamma^{-i}
\end{displaymath}
in the following way.  First, fix a graded isomorphism
\begin{displaymath}
J: \mathcal{H}_\Gamma \otimes \mathbb{C}\textrm{l}_1 \longrightarrow \mathcal{H}_\Gamma
\end{displaymath}
and define the isomorphism
\begin{displaymath}
K_i: \mathbb{C}\textrm{l}_1^{\otimes i} \longrightarrow \mathbb{C}\textrm{l}_i \qquad \textrm{by} \qquad e_1^{j_i} \otimes \cdots \otimes e_1^{j_1} \mapsto e_i^{j_i} \cdots e_1^{j_1}
\end{displaymath}
Let the $\mathbb{C}\textrm{l}_i$-linear isomorphism
\begin{displaymath}
(\textrm{id} \otimes K_i) \circ (J \otimes \textrm{id}^{\otimes i}) \circ (\textrm{id} \otimes K_{i+1}^{-1}): \mathcal{H}_\Gamma \otimes \mathbb{C}\textrm{l}_{i+1} \longrightarrow \mathcal{H}_\Gamma \otimes \mathbb{C}\textrm{l}_i
\end{displaymath}
also be denoted by $J$. Associate to $T \in \mathfrak{F}_\Gamma^{-i-1}$ the map $I \rightarrow \mathfrak{F}_\Gamma^{-i}$ defined by
\begin{displaymath}
t \mapsto J \circ \left( \cos(\pi t)\,\gamma_{i+1} + \sin(\pi t)\,T \right) \circ J^{-1}
\end{displaymath}
where $\gamma_{i+1}$ denotes Clifford multiplication by $e_{i+1} \in \mathbb{C}\textrm{l}_{i+1}$.  This defines $\phi_{i+1}$.  We have (cf. \cite{AS3}, \cite{FHT}):

\begin{thm}\label{thm:classifying} If $i$ is odd, then $\mathfrak{F}_\Gamma^{-i}$ has three components.  Two are contractible and the third we denote by $\boldsymbol{\mathfrak{F}}_\Gamma^{-i}$.  If $i$ is even, then set $\boldsymbol{\mathfrak{F}}_\Gamma^{-i} = \mathfrak{F}_\Gamma^{-i}$.  For each $i \in \mathbb{N}$, $\phi_{i+1}$ is a homotopy equivalence $\boldsymbol{\mathfrak{F}}_\Gamma^{-i-1} \rightarrow \Omega \mathfrak{F}_\Gamma^{-i}$, and
\begin{displaymath}
K^{-i}_\Gamma(X) \cong \lbrack X,\boldsymbol{\mathfrak{F}}_\Gamma^{-i} \rbrack_\Gamma
\end{displaymath}
where $\lbrack X,\boldsymbol{\mathfrak{F}}_\Gamma^{-i} \rbrack_\Gamma$ is the space of $\Gamma$-homotopy equivalence classes of $\Gamma$-equivariant maps $X \rightarrow \boldsymbol{\mathfrak{F}}_\Gamma^{-i}$.\end{thm}

Recall that there exist universal equivariant Chern characters
\begin{displaymath}
\widetilde{\textrm{ch}}_\Gamma^{-i} \in H^{-i}_\Gamma\left(\boldsymbol{\mathfrak{F}}_\Gamma^{-i};\mathbb{R}\right)
\end{displaymath}
so that, if $f: X \rightarrow \boldsymbol{\mathfrak{F}}_\Gamma^{-i}$ is an equivariant map representing a class $x \in K_\Gamma^{-i}(X)$, then
\begin{displaymath}
f^*\widetilde{\textrm{ch}}_\Gamma^{-i} = \textrm{ch}_\Gamma^{-i}(x)
\end{displaymath}
Because we wish to compare $K$-theory classes with cohomology classes on the level of forms, we shall fix representatives for the universal Chern characters. It will simplify our work somewhat if we make our choices carefully.

To begin, fix a representative
\begin{displaymath}
\tilde{\omega}^0 \in \Omega^0_\Gamma(\boldsymbol{\mathfrak{F}}_\Gamma^0)
\end{displaymath}
for $\widetilde{\textrm{ch}}_\Gamma^0$. We shall use $\tilde{\omega}^0$ to construct representatives for the universal Chern characters of nonzero degree. First, use the family $\lbrace \phi_{k} \rbrace_{k \in \mathbb{N}}$ of smooth maps inductively to define
\begin{displaymath}
\psi_{i} : \left(I^i \times \boldsymbol{\mathfrak{F}}_\Gamma^{-i},\partial I^i \times \boldsymbol{\mathfrak{F}}_\Gamma^{-i}\right) \longrightarrow \left(\boldsymbol{\mathfrak{F}}_\Gamma^0,\boldsymbol{\mathfrak{F}}_\Gamma^0\cap\textrm{GL}(\mathcal{H}_\Gamma)\right)
\end{displaymath}
Namely, set
\begin{equation}
\psi_1(T,t) = (\phi_{1}(T))(t)
\end{equation}
\begin{equation}
\psi_{i+1}(t_1, \dots, t_{i+1},T) = \psi_i\left(t_1, \dots, t_i,(\phi_{i+1}(T))(t_{i+1})\right)
\end{equation}
It is not difficult to show that $\psi_i$ maps $\partial I^i$ into the space of operators that square to $-\textrm{id}$, hence into $\boldsymbol{\mathfrak{F}}_\Gamma^0\cap\textrm{GL}(\mathcal{H}_\Gamma)$. Then the form
\begin{displaymath}
\tilde{\omega}^{-i} \in \Omega^{-i}_\Gamma(\boldsymbol{\mathfrak{F}}_\Gamma^{-i})
\end{displaymath}
defined by
\begin{equation}
\tilde{\omega}^{-i}_g \equiv \int_{I^i \times (\boldsymbol{\mathfrak{F}}_\Gamma^{-i})^g/(\boldsymbol{\mathfrak{F}}_\Gamma^{-i})^g} \psi_i^* \tilde{\omega}^0_g
\end{equation}
is a representative for $\widetilde{\textrm{ch}}_\Gamma^{-i}$. The map
\begin{displaymath}
\textrm{Map}(X,\boldsymbol{\mathfrak{F}}_\Gamma^{-i})_\Gamma \longrightarrow \Omega^{-i}_\Gamma(X) \qquad \textrm{given by} \qquad f \mapsto f^*\tilde{\omega}^{-i}
\end{displaymath}
descends to the homomorphism
\begin{displaymath}
K^{-i}_\Gamma(X) \longrightarrow H^{-i}_\Gamma(X;\mathbb{R})
\end{displaymath}
given by Theorem~\ref{thm:localization}.

\subsection{Main definition}

Roughly speaking, in seeking to define differential equivariant $K$-theory, we are looking for a theory to put in the upper corner of the diagram
\begin{displaymath}
\xymatrix{
\textrm{\framebox[1.1\width]{ ? }} \ar[r] \ar[d] & \Omega^{-i}_\Gamma(X)_\textrm{cl} \ar[d]\\
K_\Gamma^{-i}(X) \ar[r] & H^{-i}_\Gamma(X;\mathbb{R}) }
\end{displaymath}
where $\Omega^{-i}_\Gamma(X)_\textrm{cl} \subset \Omega^{-i}_\Gamma(X)$ is the subring of closed forms.  As a first approximation, one might consider putting into this corner the ring
\begin{displaymath}
A^{-i}_\Gamma(X) \subset K^{-i}_\Gamma(X) \times \Omega^{-i}_\Gamma(X)_\textrm{cl}
\end{displaymath}
of pairs $(x,\omega)$ satisfying $\textrm{ch}_\Gamma^{-i}(x) = \lbrack \omega \rbrack$. However, we wish differential equivariant $K$-theory to be a pullback as a \emph{cohomology} theory.  In other words, we wish a class in $\check{K}_\Gamma^{-i}(X)$ to be a pair $(x,\omega)$ in $A^{-i}_\Gamma(X)$ together with an ``isomorphism'' of the Chern character of $x$ and the cohomology class of $\omega$ in $H^{-i}_\Gamma(X)$.

\begin{defn} \label{defn:bigdef} The differential equivariant $K$-theory of a $\Gamma$-manifold $X$ is defined as follows: $\check{K}_\Gamma^{-i}(X)$ is the set of equivalence classes whose representatives are triples
\begin{displaymath}
(f,\eta,\omega) \in \textrm{Map}(X,\boldsymbol{\mathfrak{F}}_\Gamma^{-i})_\Gamma \times \Omega^{-i-1}_\Gamma(X) \times \Omega^{-i}_\Gamma(X)_\textrm{cl}
\end{displaymath}
satisfying
\begin{displaymath}
d_\Gamma\eta = \omega - f^*\tilde{\omega}^{-i}
\end{displaymath}
Two triples $(f,\eta,\omega)$ and $(f',\eta',\omega')$ are in the same equivalence class if there exist
\begin{displaymath}
F \in \textrm{Map}(I \times X, \boldsymbol{\mathfrak{F}}_\Gamma^{-i})_\Gamma \qquad \textrm{and} \qquad \beta \in \Omega^{-i-2}_\Gamma(X)
\end{displaymath}
so that
\begin{displaymath}
F_0 = f \qquad F_1 = f' \qquad \omega = \omega'
\end{displaymath}
and
\begin{displaymath}
\eta' = \eta - \int_{I \times X/X} F^*\tilde{\omega}^{-i} + d_\Gamma\beta
\end{displaymath}
\end{defn}

\begin{rem} \label{rem:cs} Let
\begin{displaymath}
\Omega^\bullet_\Gamma(X)_K \subset \Omega^\bullet_\Gamma(X)_\textrm{cl}
\end{displaymath}
be the subgroup consisting of forms whose cohomology classes lie in the image of $K^\bullet_\Gamma(X)$ under the Chern character map. If $(f,\eta,\omega)$ and $(f,\eta',\omega)$ are two triples representing classes in $\check{K}_\Gamma^{-i}(X)$, then it follows from the definition that they are equivalent if and only if $\eta$ and $\eta'$ differ by an element of $\Omega_\Gamma^{-i-1}(X)_K$.

A particular application of this fact will be technically useful. Suppose that $\lbrace V_j : j \in \mathbb{Z}/k\mathbb{Z}\rbrace$ is a collection of complex equivariant vector bundles over $X$ with isomorphisms $\varphi_j: V_j \rightarrow V_{j+1}$. Suppose further that, for each $j \in \mathbb{Z}/k\mathbb{Z}$, $h_j$ is a smooth map from the interval $I$ into the space of compatibly unitary connections for $V$, such that $\varphi_j^*h_{j+1}(0) = h_j(1)$. Let $\nabla_j$ denote the connection for $I \times V_j \rightarrow I \times X$ defined as
\begin{displaymath}
h_j(t) + dt \, \partial_t
\end{displaymath}
Set
\begin{displaymath}
\kappa = \sum_{j=1}^k \int_{I \times X/X} \omega\left(I \times V_j;\nabla_j\right)
\end{displaymath}
Then $\kappa \in \Omega_\Gamma^{-1}(X)_K$. It follows that, if $(f,\eta,\omega)$ is a triple representing a class in $\check{K}_\Gamma^0(X)$, then $(f,\eta+\kappa,\omega)$ represents the same class.

Applying similar arguments to $S^i \times X$ obtains an analogous statement for $\check{K}_\Gamma^{-i}(X)$ \end{rem}

\begin{rem} \label{rem:bott} There is a homotopy equivalence $\boldsymbol{\mathfrak{F}}_\Gamma^{-i} \simeq \boldsymbol{\mathfrak{F}}_\Gamma^{-i-2}$, under which $\tilde{\omega}^{-i}$ corresponds to $\tilde{\omega}^{-i-2}$ up to an exact form.  It follows that Bott periodicity generalizes to differential equivariant $K$-theory, which is to say that
\begin{equation}\label{eq:bott}
\check{K}_\Gamma^{-i}(X) \cong \check{K}_\Gamma^{-i-2}(X)
\end{equation}
This allows one to define differential equivariant $K$-theory groups $\check{K}_\Gamma^{i}$ for $i > 0$. \end{rem}

The differential equivariant $K$-theory groups lie in the short exact sequence
\begin{equation} \label{eq:shortexact}
0 \longrightarrow \frac{H^{-i-1}_\Gamma(X;\mathbb{R})}{\textrm{ch}^{-i-1}_\Gamma K_\Gamma^{-i-1}(X)} \longrightarrow \check{K}^{-i}_\Gamma(X) \stackrel{c}{\longrightarrow} A^{-i}_\Gamma(X) \longrightarrow 0
\end{equation}
If $c(\check{x}) = (x,\omega)$, it is natural to call $x$ the \emph{characteristic class} of $\check{x}$ and $\omega$ the \emph{curvature}.  This short exact sequence may be rearranged as
\begin{equation} \label{eq:shortexact2}
0 \longrightarrow \frac{\Omega^{-i-1}_\Gamma(X)}{\Omega^{-i-1}_\Gamma(X)_K} \longrightarrow \check{K}^{-i}_\Gamma(X) \longrightarrow K^{-i}_\Gamma(X) \longrightarrow 0
\end{equation}
where the second map is the characteristic class. A third way to arrange (\ref{eq:shortexact}) and (\ref{eq:shortexact2}) is the short exact sequence
\begin{equation} \label{eq:shortexact3}
0 \longrightarrow K^{-i-1}_\Gamma(X;\mathbb{R}/\mathbb{Z}) \longrightarrow \check{K}^{-i}_\Gamma(X) \stackrel{\omega}{\longrightarrow} \Omega^{-i}_\Gamma(X)_K \longrightarrow 0
\end{equation}
The second map is the curvature.  Its kernel is the set of \emph{flat} differential equivariant $K$-theory classes, an abelian group whose identity component is the kernel of (\ref{eq:shortexact}) and whose group of components is the torsion subgroup of $K_\Gamma^{-i}(X)$ (cf. \cite{Lo}).

We may use these short exact sequences to compute the differential equivariant $K$-theory of a point:

\begin{prop}\label{prop:point} Let $\Gamma$ act trivially on the point.  Then
\begin{equation}
\check{K}_\Gamma^{-i}(\textrm{\emph{pt}}) \cong \left\{ \begin{array} {l l}
R(\Gamma) & \textrm{$i$ \emph{even}}\\
(R(\Gamma) \otimes \mathbb{R})/R(\Gamma) & \textrm{$i$ \emph{odd}}
\end{array}
\right.
\end{equation}
where $R(\Gamma)$ is the ring of virtual characters for $\Gamma$. \end{prop}

\begin{proof} If $i$ is even, then $\Omega^{-i-1}_\Gamma(\textrm{pt})$ is trivial and (\ref{eq:shortexact2}) implies that
\begin{displaymath}
\check{K}_\Gamma^{-i}(\textrm{pt}) \cong K_\Gamma^{-i}(\textrm{pt}) \cong R(\Gamma)
\end{displaymath}
On the other hand, if $i$ is odd, then (\ref{eq:shortexact}) implies that
\begin{displaymath}
\qquad \check{K}_\Gamma^{-i}(\textrm{pt}) \cong \frac{H^{-i-1}_\Gamma(\textrm{pt};\mathbb{R})}{K_\Gamma^{-i-1}(\textrm{pt})} \cong \frac{R(\Gamma) \otimes \mathbb{R}}{R(\Gamma)}
\end{displaymath}
since $A_\Gamma^{-i}(\textrm{pt})$ is trivial.  \end{proof}

A definition for degree zero differential equivariant $K$-theory has been proposed by Szabo and Valentino in \cite{SV}.  The main difference between our approach and that of \cite{SV} is that most of the statements of the latter are couched in the language of Bredon cohomology, whose definition we briefly recall.

Let $\textrm{Or}(\Gamma)$ be the orbit category for $\Gamma$, whose objects are homogeneous spaces $\Gamma/H$ and whose morphisms are the $\Gamma$-maps between them. Define the contravariant functor $\underline{R}(-): \textrm{Or}(\Gamma)^\textrm{op} \rightarrow \textrm{Ab}$, where $\textrm{Ab}$ is the category of abelian groups, by
\begin{displaymath}
\Gamma/H \mapsto R(H)
\end{displaymath}
Next, given a CW-complex $X$ with cellular action of $\Gamma$, define the contravariant functor $\underline{C}_n(X): \textrm{Or}(\Gamma)^\textrm{op} \rightarrow \textrm{Ab}$ by
\begin{displaymath}
\Gamma/H \mapsto C_n(X^H)
\end{displaymath}
where $X^H$ is the $H$-fixed point set.  Let $C_\Gamma^n(X,\underline{R})$ be the group of natural transformations from $\underline{C}_n(X)$ to $\underline{R}(-)$.  This group has a natural coboundary operator given by the boundary operator on $C_n$, and Bredon cohomology is defined as the cohomology $H_\Gamma^n(X;\underline{R}(-))$ of this cochain complex.

The Bredon cohomological equivariant Chern character, as constructed in \cite{Lu}, for instance, takes values in $H_\Gamma^\bullet(X;\mathbb{R} \otimes \underline{R}(-))$, and, after tensoring with $\mathbb{R}$, one has the isomorphism
\begin{equation}
K_\Gamma^\bullet(X) \otimes \mathbb{R} \cong H_\Gamma^\bullet(X;\mathbb{R} \otimes \underline{R}(-))
\end{equation}
(cf. \cite[Theorem 2.5]{SV}), whence it follows that
\begin{equation}
H_\Gamma^\textrm{even,odd}(X;\mathbb{R} \otimes \underline{R}(-)) \cong H_\Gamma^{0,-1}(X;\mathbb{R})
\end{equation}
In fact, the authors of \cite{SV} present a dimension-counting argument to deduce that
\begin{equation}\label{eq:bredon}
H_\Gamma^\bullet(X;\mathbb{R} \otimes \underline{R}(-)) \cong \bigoplus_{\lbrack g \rbrack \subset \Gamma} H^\bullet(X^g;\mathbb{R})^{Z_g}
\end{equation}
where the summation is over conjugacy classes in $\Gamma$ and $Z_g$ is the centralizer of $g$.  The latter cohomology is naturally identified with the real subspace of $H^{0,-1}_\Gamma(X;\mathbb{C})$ under the real structure induced by complex conjugation, rather than that induced by Theorem~\ref{thm:localization}, the Atiyah-Segal localization theorem.  The composition of the Bredon cohomological Chern character on $K_\Gamma^0(X)$ with (\ref{eq:bredon}) thus yields a class in $H^0_\Gamma(X;\mathbb{C})$ that is real under the former real structure.  A representative for a class in differential equivariant $K$-theory is composed of a map from $X$ into the space of Fredholm operators and a differential form in $\Omega^0_\Gamma(X;\mathbb{C})$ which is real under the real structure induced by conjugation, together with an ``isomorphism'' between the two, just as in our definition, except that the Bredon cohomological Chern character is used instead.  It is clear, however, that an application of the Five Lemma to, say, our short exact sequence (\ref{eq:shortexact}) and their version of the same, the sequence \cite[(5.12)]{SV}, suffices to prove that the groups we define are isomorphic.

The authors of \cite{SV} argue that the machinery developed in \cite{HS} cannot be immediately applied to the present context, and that their introduction of Bredon cohomology is required both by the equivariant Chern character isomorphism and by the explicit use of differential forms, neither of which can be accommodated by Borel equivariant cohomology, as constructed in \cite{GS}, for instance.  The authors remark, rightly, that to use the Borel theory would be to ignore important data.  It would seem, however, that for a finite group $\Gamma$, the complex $(\Omega^\bullet_\Gamma(X),d_\Gamma)$ together with the equivariant Chern character of Atiyah and Segal give all that is desired.  Furthermore, we shall see that a closer reliance on differential forms and the Weil homomorphism makes a construction employing this complex more conducive to the definition of a ring structure and an integration on $\check{K}_\Gamma^\bullet(X)$.

\subsection{Special cases}

The ordinary differential $K$-theory ring for a smooth manifold $X$ is defined as
\begin{equation}
\check{K}^\bullet(X) \equiv \check{K}_{\lbrace e \rbrace}^\bullet(X)
\end{equation}
where $\lbrace e \rbrace$ denotes the trivial group (cf. \cite{HS}, \cite{K}).  We know that there are natural isomorphisms relating equivariant $K$-theory to ordinary $K$-theory in two special cases, namely,
\begin{equation}
K^\bullet_\Gamma(X) \cong K^\bullet(X) \otimes R(\Gamma)
\end{equation}
if $\Gamma$ acts trivially, and
\begin{equation}
K^\bullet_\Gamma(X) \cong K^\bullet(X/\Gamma)
\end{equation}
if $\Gamma$ acts freely.  We now establish analogous isomorphisms in differential equivariant $K$-theory.

\begin{prop} \label{prop:free} Let $X$ be a $\Gamma$-manifold with free $\Gamma$-action.  Then
\begin{equation}
\check{K}_\Gamma^\bullet(X) \cong \check{K}^\bullet(X/\Gamma)
\end{equation}
\end{prop}

\begin{proof} Recall that we defined $\mathcal{H}_\Gamma$ as $\mathcal{H} \otimes W_\Gamma$, where $W_\Gamma$ is a $\Gamma$-representation in whose decomposition each irreducible representation appears once.  Let $\phi_i: \boldsymbol{\mathfrak{F}}_{\lbrace e \rbrace}^{-i} \hookrightarrow \boldsymbol{\mathfrak{F}}_\Gamma^{-i}$ be the inclusion map given by $T \mapsto T \otimes \textrm{id}$.  We may assume without loss of generality that the universal connections have been so chosen that the universal Chern character forms agree under this inclusion.

Notice that there exist natural isomorphisms
\begin{equation} \label{eq:isoms}
\Omega(X/\Gamma;\pi K_\mathbb{R})^\bullet \cong \Omega^\bullet_\Gamma(X) \qquad K^\bullet(X/\Gamma) \cong K_\Gamma^\bullet(X)
\end{equation}
Let $p: X \rightarrow X/\Gamma$ be the quotient map.  To a triple $(f,\eta,\omega)$ representing a class in $\check{K}^{-i}(X/\Gamma)$ associate the triple
\begin{displaymath}
(\phi_i \circ f \circ p, p^*\eta, p^*\omega)
\end{displaymath}
This descends to a homomorphism
\begin{displaymath}
\check{K}^{-i}(X/\Gamma) \longrightarrow \check{K}^{-i}_\Gamma(X)
\end{displaymath}
that agrees with the isomorphisms (\ref{eq:isoms}).  It thus fits into the commutative diagram
\begin{equation}
\xymatrix{
0 \ar[r] & \frac{H(X/\Gamma;\pi K_\mathbb{R})^{-i-1}}{\textrm{ch}^{-i-1}K^{-i}(X/\Gamma)} \ar[r] \ar[d] & \check{K}^{-i}(X/\Gamma) \ar[r] \ar[d] & A^{-i}(X/\Gamma) \ar[r] \ar[d] & 0\\
0 \ar[r] & \frac{H^{-i-1}_\Gamma(X;\mathbb{R})}{\textrm{ch}_\Gamma^{-i-1}K^{-i-1}_\Gamma(X)} \ar[r] & \check{K}_\Gamma^{-i}(X) \ar[r] & A^{-i}_\Gamma(X) \ar[r] & 0 }
\end{equation}
corresponding to the exact sequence (\ref{eq:shortexact}).  The first and third homomorphisms are isomorphisms, so an application of the Five Lemma suffices to prove that $\check{K}_\Gamma^\bullet(X) \cong \check{K}^\bullet(X/\Gamma)$. \end{proof}

\begin{prop} \label{prop:trivial} Let $X$ be a $\Gamma$-manifold with trivial $\Gamma$-action.  Then
\begin{displaymath}
\check{K}_\Gamma^\bullet(X) \cong \check{K}^\bullet(X) \otimes R(\Gamma)
\end{displaymath}
\end{prop}

\begin{proof} Let $(f,\eta,\omega)$ be a representative triple for a class in $\check{K}_\Gamma^{-i}(X)$.  The image of $X$ under $f$ lies in the subspace $(\boldsymbol{\mathfrak{F}}_\Gamma^{-i})^\Gamma \subset \boldsymbol{\mathfrak{F}}_\Gamma^{-i}$ of operators that commute with $\Gamma$-action on $\mathcal{H}_\Gamma$. Decompose $W_\Gamma$ into a direct sum of irreducible representations, writing
\begin{displaymath}
W_\Gamma = \bigoplus_\chi W_\chi,
\end{displaymath}
where $W_\chi$ is the representation with character $\chi$, and $\chi$ ranges over the characters of the irreducible representations. We thus have
\begin{equation}
\mathcal{H}_\Gamma = \bigoplus_\chi \mathcal{H} \otimes W_\chi
\end{equation}
and each operator $T \in (\boldsymbol{\mathfrak{F}}_\Gamma^{-i})^\Gamma$ is of the form
\begin{displaymath}
T = \sum_{\chi} T_\chi \otimes \textrm{id}_{W_\chi} \qquad T_\chi \in \boldsymbol{\mathfrak{F}}_{\lbrace e \rbrace}^{-i}
\end{displaymath}
We decompose $f$ accordingly, writing
\begin{displaymath}
f(x) = \sum_\chi f_\chi(x) \otimes \textrm{id}_{W_\chi} \qquad f_\chi: X \longrightarrow \boldsymbol{\mathfrak{F}}_{\lbrace e \rbrace}^{-i}
\end{displaymath}
Furthermore, there is a natural isomorphism
\begin{equation}
\Omega^\bullet_\Gamma(X) \cong \Omega(X;\pi K_\mathbb{R})^\bullet \otimes R(\Gamma)
\end{equation}
so we decompose $\omega$ as a sum of forms $\omega_\chi$, and $\eta$ as a sum of forms $\eta_\chi$. We thus have
\begin{equation}
\omega_g = \sum_\chi \omega_\chi \cdot \chi(g)
\end{equation}
and similarly for $\eta_g$.  It is easily checked that $(f_\chi,\eta_\chi,\omega_\chi)$ is a representative triple for a class in $\check{K}^{-i}(X)$.

The homomorphism
\begin{displaymath}
\check{K}_\Gamma^\bullet(X) \longrightarrow \check{K}^\bullet(X) \otimes R(\Gamma)
\end{displaymath}
is defined thus: to the class represented by the triple $(f,\eta,\omega)$, associate the class represented by
\begin{displaymath}
\sum_\chi (f_\chi,\eta_\chi,\omega_\chi) \otimes \chi
\end{displaymath}
An application of the Five Lemma to the exact sequence (\ref{eq:shortexact2}) proves that this is an isomorphism.  \end{proof}


\section{Alternative definitions}

In this section, we present two alternative definitions for the differential equivariant $K$-theory groups. We begin by recalling the required constructions and fixing notation. In Subsection 3.2, we give a slightly modified, but equivalent, definition of differential equivariant $K$-theory for compact spaces. This model is useful for defining a natural ring structure, as we shall see. We use the modified definition to establish the equivalence of the geometric description in Subsection 3.3.

\subsection{Finite-dimensional approximations}

As before, let $\mathcal{H}_\Gamma$ be a complex separable $\mathbb{Z}/2\mathbb{Z}$-graded infinite-dimensional Hilbert space which is a $\Gamma$-representation in whose decomposition each finite-dimensional representation appears infinitely many times. For each $T \in
\textrm{Fred}(\mathcal{H}_\Gamma^0)$, set
\begin{displaymath}
\hat{T} \equiv \left( \begin{array}{c c}
0 & -T^*\\
T & 0 \end{array} \right) \in \boldsymbol{\mathfrak{F}}_\Gamma^0
\end{displaymath}
Then the map given by $T \mapsto \hat{T}$ is a diffeomorphism from $\textrm{Fred}(\mathcal{H}_\Gamma^0)$ to $\boldsymbol{\mathfrak{F}}_\Gamma^0$. Recall that $\textrm{Fred}(\mathcal{H}_\Gamma^0)$ is of the homotopy type of $\mathbb{Z} \times \textrm{BU}$: the index yields a bijection from the set of connected components to $\mathbb{Z}$. For each $k \in \mathbb{Z}$, let
\begin{displaymath}
\boldsymbol{\mathfrak{F}}_{\Gamma,k}^0 \equiv \lbrace \hat{T} \in \boldsymbol{\mathfrak{F}}_\Gamma^0: \textrm{ind}(T) = k \rbrace
\end{displaymath}
There is a canonical open cover for $\boldsymbol{\mathfrak{F}}_\Gamma^0$ defined as follows. Let $\mathcal{G}_\Gamma$ denote the space of finite-dimensional $\Gamma$-invariant subspaces $W \subset \mathcal{H}_\Gamma^1$, and, for each $W$ in $\mathcal{G}_\Gamma$, define
\begin{displaymath}
\mathcal{O}_W \equiv \left\{ \hat{T} \in \boldsymbol{\mathfrak{F}}_\Gamma^0: T(\mathcal{H}_\Gamma^0) + W = \mathcal{H}_\Gamma^1 \right\} \subset \boldsymbol{\mathfrak{F}}_\Gamma^0
\end{displaymath}
Each $\mathcal{O}_W$ is an invariant open subset of $\boldsymbol{\mathfrak{F}}_\Gamma^0$, and $\lbrace \mathcal{O}_W \rbrace_{W \in \mathcal{G}_\Gamma}$ is an open cover for $\boldsymbol{\mathfrak{F}}_\Gamma^0$. Furthermore, for each $W \in \mathcal{G}_\Gamma$, there is a canonical graded equivariant vector bundle
\begin{displaymath}
V_W \longrightarrow \mathcal{O}_W
\end{displaymath}
whose fiber over $\hat{T} \in \mathcal{O}_W$ is given as
\begin{displaymath}
(V_W^0)_{\hat{T}} = T^{-1}(W) \qquad (V_W^1)_{\hat{T}} = W
\end{displaymath}
The restriction of $V_W^0$ to $\mathcal{O}_W \cap \boldsymbol{\mathfrak{F}}_{\Gamma,k}^0$ is of rank $(\textrm{dim}\,W + k)$.

We have seen that $\boldsymbol{\mathfrak{F}}_\Gamma^0$ is a classifying space for degree zero equivariant $K$-theory. The manner in which the two types of representatives for $K$-theory classes---namely, maps and graded vector bundles---relate to each other may be seen concretely as follows. First, if $X$ is a compact manifold and $f$ is an equivariant map $X \rightarrow \boldsymbol{\mathfrak{F}}_\Gamma^0$, then there exists $W \in \mathcal{G}_\Gamma$ such that $f(X) \subset \mathcal{O}_W$. Then $\lbrack f^*V_W \rbrack \in K_\Gamma^0(X)$ is the class represented by $f$. Now suppose, on the other hand, that $V = V^0 \oplus V^1$ is a graded equivariant vector bundle over $X$. Let $V' \rightarrow X$ be an ungraded bundle such that $V^1 \oplus V'$ is trivial. Choose a trivialization, and identify the fiber with some $W \in \mathcal{G}_\Gamma$. Let $T$ be a Fredholm operator with cokernel $W$ and trivial kernel. We know by Kuiper's Theorem that the Hilbert bundle $V^0 \oplus V' \oplus \mathcal{H}_\Gamma^0$ is trivializable; trivialize it, and choose an isomorphism from the fiber to $\mathcal{H}_\Gamma^0$. Let $\lbrace e_i \rbrace_{i \in \mathbb{N}}$ be an orthonormal basis for $\mathcal{H}_\Gamma^0$, and, for each $x \in X$, let $T_x \in \textrm{Fred}(\mathcal{H}_\Gamma^0)$ be the operator induced by the linear map on $V^0_x \oplus V'_x \oplus \mathcal{H}_\Gamma^0$ given by
\begin{displaymath}
(v,e_i) \mapsto \left\{ \begin{array}{l l}
0 & i = 1\\
(v,e_{i-1}) & i > 1 \end{array} \right. \qquad v \in V_x^0 \oplus V'_x
\end{displaymath}
Then the map $f: X \rightarrow \boldsymbol{\mathfrak{F}}_\Gamma^0$ defined by $x \mapsto \widehat{T \circ T_x}$ is a classifying map for $\lbrack V \rbrack$.

Recall that there exists a universal equivariant Chern character
\begin{displaymath}
\widetilde{\textrm{ch}}_\Gamma \in H^0_\Gamma\left(\boldsymbol{\mathfrak{F}}_\Gamma^0;\mathbb{R}\right)
\end{displaymath}
so that, if $f: X \rightarrow \boldsymbol{\mathfrak{F}}_\Gamma^0$ is an equivariant map representing a class $x \in K_\Gamma^0(X)$, then
\begin{displaymath}
f^*\widetilde{\textrm{ch}}_\Gamma = \textrm{ch}_\Gamma(x)
\end{displaymath}
In defining differential equivariant $K$-theory in the last section, we fixed a representative $\tilde{\omega}$ for this class without concerning ourselves with how it was chosen. However, some constructions also seem to have need of a ``Chern-Simons form'' relating the pullback of the universal Chern character representative under $f$ to a representative for $\textrm{ch}_\Gamma(x)$ constructed via the Chern-Weil method. While it may very well be possible to construct a global form $\tilde{\omega}$ from a universal connection via the Chern-Weil method, we have been unable to do so. It is therefore necessary for us to have recourse to finite-dimensional approximations. As we shall presently explain, there do exist universal connections on the sets in the canonical open cover, and, if $X$ is compact, then these connections may be compared with connections on $X$.

Let $\textrm{BU}(n)_\Gamma$ be the standard classifying space for equivariant principal $\textrm{U}(n)$-bundles, namely, the Grassmannian of $n$-dimensional subspaces of $\mathcal{H}_\Gamma^0$. Let $\textrm{EU}_n$ be the (contractible) space of isometric embeddings $\mathbb{C}^n \rightarrow \mathcal{H}^0_\Gamma$; then the universal $\textrm{U}(n)$-bundle is given by the fibration
\begin{displaymath}
\textrm{EU}(n)_\Gamma \longrightarrow \textrm{BU}(n)_\Gamma \qquad L \mapsto L(\mathbb{C}^n)
\end{displaymath}
Classifying spaces are unique up to homotopy equivalence, so it should be clear from the foregoing discussion that, for each $W \in \mathcal{G}_\Gamma$,
\begin{displaymath}
\mathcal{O}_W \simeq \bigsqcup_{n \in \mathbb{N}} \textrm{BU}(n)_\Gamma
\end{displaymath}
and, in particular, that
\begin{displaymath}
\mathcal{O}_W \cap \boldsymbol{\mathfrak{F}}_{\Gamma,k}^0 \simeq \textrm{BU}(\textrm{dim}\,W + k)_\Gamma
\end{displaymath}
The homotopy equivalence is given by the map
\begin{displaymath}
\varphi_W: \mathcal{O}_W \longrightarrow \bigsqcup_{n \in \mathbb{N}} \textrm{BU}(n)_\Gamma \qquad \hat{T} \mapsto T^{-1}(W)
\end{displaymath}
and the unitary frame bundle $\mathcal{P}_\textrm{U}(V^0_W) \rightarrow \mathcal{O}_W$ associated with $V^0_W$ is isomorphic to the pullback bundle
\begin{displaymath}
\varphi^*_W \left( \bigsqcup_{n \in \mathbb{N}} \textrm{EU}(n)_\Gamma \right) \longrightarrow \mathcal{O}_W
\end{displaymath}
There exists a canonical connection for the universal $\textrm{U}(n)$-bundle over $\textrm{BU}(n)_\Gamma$. Suppose $L \in \textrm{EU}(n)_\Gamma$. A convenient description of the tangent space $\textrm{T}_L\textrm{EU}(n)_\Gamma$ to $\textrm{EU}(n)_\Gamma$ at $L$ is given as follows.  A path in $\textrm{U}(\mathcal{L}_\Gamma)$ starting at the identity is mapped (via pointwise composition with $L$) to a path in $\textrm{EU}(n)_\Gamma$ starting at $L$.  This yields a linear map $\mathfrak{u}(\mathcal{H}^0_\Gamma) \rightarrow \textrm{T}_L\textrm{EU}(n)_\Gamma$.  This map is clearly surjective, and its kernel is $\mathfrak{u}(L(\mathbb{C}^n)^\perp)$.  Thus,
\begin{displaymath}
\textrm{T}_L\textrm{EU}(n)_\Gamma \cong \mathfrak{u}(\mathcal{H}^0_\Gamma)/\mathfrak{u}(L(\mathbb{C}^n)^\perp)
\end{displaymath}
We define an inner product on $\mathfrak{u}(\mathcal{H}^0_\Gamma)$ as follows: for each $A,B \in \mathfrak{u}(\mathcal{H}^0_\Gamma)$, let
\begin{displaymath}
\langle A,B \rangle = -\textrm{tr}_{L(\mathbb{C}^N)}(A^*B + B^*A)
\end{displaymath}
This inner product descends to a real, nondegenerate, positive-definite inner product on $\textrm{T}_L\textrm{EU}(n)_\Gamma$.  The metric on $\textrm{T}\textrm{EU}(n)_\Gamma$ thus obtained is invariant under $\textrm{U}(n)$-action, and the canonical connection $\Theta$ is given by orthogonal projection onto the fiber. Let
\begin{displaymath}
\hat{\omega} \in \Omega^0_\Gamma(\textrm{BU}(n)_\Gamma)
\end{displaymath}
be the representative for the Chern character constructed from this connection via the Chern-Weil method. Define
\begin{equation}
\tilde{\omega}_W \equiv \varphi_W^*\hat{\omega} - \textrm{dim}(W) \in \Omega^0_\Gamma(\mathcal{O}_W)
\end{equation}
Then $\tilde{\omega}_W$ is a representative for the restriction of the universal Chern character to $\mathcal{O}_W$, and, if $f: X \rightarrow \mathcal{O}_W$ an equivariant map representing a $K$-theory class $x$, then $f^*\tilde{\omega}_W$ is a form representing $\textrm{ch}_\Gamma(x)$.

We would like to relate these local representatives of the universal Chern character to one another on the overlaps of the canonical open sets. Suppose that $W,W' \in \mathcal{G}_\Gamma$ such that $\mathcal{O}_W \cap \mathcal{O}_{W'} \neq \varnothing$. We can use the universal connections to a construct a Chern-Simons form
\begin{displaymath}
\widetilde{\textrm{CS}}_{W,W'} \in \Omega^{-1}_\Gamma(\mathcal{O}_W \cap \mathcal{O}_{W'})
\end{displaymath}
satisfying
\begin{equation} \label{eq:glowfrog}
d_\Gamma\widetilde{\textrm{CS}}_{W,W'} = \tilde{\omega}_{W'} - \tilde{\omega}_W
\end{equation}
Suppose first that $W \subset W'$. Then $\mathcal{O}_W \subset \mathcal{O}_{W'}$. Let $\nabla_W$ be the covariant derivative for $V_W^0$ given by the universal connection $\varphi_W^*\Theta$, and $\nabla_{W'}$ the connection for $V_{W'}^0$ given by the universal connection $\varphi_{W'}^*\Theta$. Let $W''$ be the orthogonal complement to $W$ in $W'$, and notice that
\begin{displaymath}
V_{W'}^0 \simeq V_W^0 \oplus W''
\end{displaymath}
Let $\nabla_W' = \nabla_W \oplus d$, where $d$ is the trivial (flat) connection. Set
\begin{equation}
\widetilde{\textrm{CS}}_{W,W'} \equiv \textrm{CS}_\Gamma(V_{W'};\nabla_W',\nabla_{W'}) \in \Omega^{-1}_\Gamma(\mathcal{O}_W)
\end{equation}
Then this form satisfies (\ref{eq:glowfrog}). The general case is handled similarly, by considering $W$ and $W'$ as subspaces of their sum $W + W'$. Namely, we extend each of $\nabla_W$ and $\nabla_{W'}$ via flat connections to connections $\nabla_W'$ and $\nabla_{W'}'$ for $V_{W+W'}^0$, and set
\begin{equation}
\widetilde{\textrm{CS}}_{W,W'} \equiv \textrm{CS}_\Gamma(V_{W+W'};\nabla_W',\nabla_{W'}') \in \Omega^{-1}_\Gamma(\mathcal{O}_W \cap \mathcal{O}_{W'})
\end{equation}
It follows from the properties of Chern-Simons forms that, on $\mathcal{O}_W \cap \mathcal{O}_{W'}$,
\begin{equation}
\widetilde{\textrm{CS}}_{W,W'} = \widetilde{\textrm{CS}}_{W,W+W'} - \widetilde{\textrm{CS}}_{W',W+W'}
\end{equation}
up to an exact form.

Let us now generalize the foregoing constructions to nonzero degrees. Recall that we have chosen isomorphisms
\begin{displaymath}
J: \mathcal{H}_\Gamma \otimes \mathbb{C}\textrm{l}_i \longrightarrow \mathcal{H}_\Gamma \otimes \mathbb{C}\textrm{l}_{i-1}
\end{displaymath}
Let $\mathcal{G}_{\Gamma,i} \subset \mathcal{G}_\Gamma$ be the space of finite-dimensional $\Gamma$-invariant subspaces $W \subset \mathcal{H}_\Gamma^1$ which are of the form
\begin{displaymath}
J^i \left( (W^0 \otimes \mathbb{C}\textrm{l}_i^1) \oplus (W^1 \otimes \mathbb{C}\textrm{l}_i^0) \right)
\end{displaymath}
Notice that the subcover $\lbrace \mathcal{O}_W \rbrace_{W \in \mathcal{G}_{\Gamma,i}}$ of the canonical open cover for $\boldsymbol{\mathfrak{F}}_\Gamma^0$ is also an open cover, since each $W \in \mathcal{G}_\Gamma$ is a subspace of an element of $\mathcal{G}_{\Gamma,i}$. Define
\begin{equation}
\Omega^i\boldsymbol{\mathfrak{F}}_\Gamma^0 \equiv \textrm{Map}\left( (I^i,\partial I^i),(\boldsymbol{\mathfrak{F}}_\Gamma^0,\textrm{GL}(\mathcal{H}_\Gamma)) \right)
\end{equation}
and, for each $W \in \mathcal{G}_{\Gamma,i}$, define
\begin{equation}
\Omega^i\mathcal{O}_W \equiv \textrm{Map}\left( (I^i,\partial I^i),(\mathcal{O}_W,\mathcal{O}_W \cap \textrm{GL}(\mathcal{H}_\Gamma)) \right)
\end{equation}
The space $\Omega^i\boldsymbol{\mathfrak{F}}_\Gamma^0$ carries the compact open topology, and $\lbrace \Omega^i\mathcal{O}_W \rbrace_{W \in \mathcal{G}_{\Gamma,i}}$ is an open cover for $\Omega^i\boldsymbol{\mathfrak{F}}_\Gamma^0$. The map
\begin{displaymath}
\psi_i: I^i \times \boldsymbol{\mathfrak{F}}_\Gamma^{-i} \longrightarrow \boldsymbol{\mathfrak{F}}_\Gamma^0
\end{displaymath}
defined in Subection 2.1 induces a homotopy equivalence
\begin{displaymath}
\boldsymbol{\mathfrak{F}}_\Gamma^{-i} \stackrel{\sim}{\longrightarrow} \Omega^i\boldsymbol{\mathfrak{F}}_\Gamma^0
\end{displaymath}
Let $\mathcal{O}_W^{-i} \subset \boldsymbol{\mathfrak{F}}_\Gamma^{-i}$ be the preimage of $\Omega^i\mathcal{O}_W$ under this map. Then $\lbrace \mathcal{O}_W^{-i} \rbrace_{W \in \mathcal{G}_{\Gamma,i}}$ is an open cover for $\boldsymbol{\mathfrak{F}}_\Gamma^{-i}$. Notice also that
\begin{equation}
\mathcal{O}_W^{-i} \simeq \Omega^i\textrm{BU}(\textrm{dim}\,W)_\Gamma
\end{equation}

Set
\begin{equation}
\tilde{\omega}_W^{-i} \equiv \int_{I^i \times \mathcal{O}_W^{-i}/\mathcal{O}_W^{-i}} \psi_i^*\tilde{\omega}_W \in \Omega^{-i}_\Gamma(\mathcal{O}_W^{-i})
\end{equation}
Then $\tilde{\omega}_W^{-i}$ is a representative for the restriction of the universal Chern character to $\mathcal{O}_W^{-i}$, and, if $f: X \rightarrow \mathcal{O}_W^{-i}$ an equivariant map representing a $K$-theory class $x$, then $f^*\tilde{\omega}_W^{-i}$ is a form representing $\textrm{ch}_\Gamma^{-i}(x)$.

Now suppose that $W,W' \in \mathcal{G}_{\Gamma,i}$ such that $\mathcal{O}_W^{-i} \cap \mathcal{O}_{W'}^{-i} \neq \varnothing$. Define
\begin{equation}
\widetilde{\textrm{CS}}_{W,W'}^{-i} \equiv \int_{I^i \times (\mathcal{O}_W^{-i} \cap \mathcal{O}_{W'}^{-i})/(\mathcal{O}_W^{-i} \cap \mathcal{O}_{W'}^{-i})} \psi_i^*\widetilde{\textrm{CS}}_{W,W'},
\end{equation}
an element of $\Omega_\Gamma^{-i-1}\left(\mathcal{O}_W^{-i} \cap \mathcal{O}_{W'}^{-i}\right)$. Now, the image of $\partial I^i \times \mathcal{O}_W^{-i}$ under the map
\begin{displaymath}
\varphi_W \circ \psi_i: I^i \times \mathcal{O}_W^{-i} \longrightarrow \textrm{BU}(\textrm{dim}\,W)_\Gamma
\end{displaymath}
is a point. (This is because $W \in \mathcal{G}_{\Gamma,i}$, and so its image under Clifford multiplication by any nonzero linear combination $x = \sum x_j e_j$ of the generators of $\mathbb{C}\textrm{l}_i$ is independent of the choice of weights.) It follows that the restriction of the connection $\psi_i^*\nabla_W$ to the vector bundle
\begin{displaymath}
\partial\left(\psi_i^* V_W\right) \longrightarrow \partial I^i \times \mathcal{O}_W^{-i}
\end{displaymath} is trivial. The same is true for $\partial I^i \times \mathcal{O}_{W'}^{-i}$, so the Chern-Simons form $\psi_i^*\widetilde{\textrm{CS}}_{W,W'}$ vanishes when restricted to $\partial I^i \times (\mathcal{O}_W^{-i} \cap \mathcal{O}_{W'}^{-i})$. Thus,
\begin{equation}
d_\Gamma \widetilde{\textrm{CS}}_{W,W'}^{-i} = \tilde{\omega}_{W'}^{-i} - \tilde{\omega}_W^{-i}
\end{equation}

\subsection{A slightly modified definition}

The basic idea behind the modified definition is that one can define the differential equivariant $K$-theory of compact spaces by using the universal connections on the canonical open sets, rather than a global form representing the universal Chern character, provided that one includes additional data in the class representatives. The cost of doing so is, of course, a more complicated set of equivalence relations, but the construction is entirely canonical, and has the benefit of allowing one to define ``Chern-Simons forms'' as in the foregoing subsection. This is helpful for relating the main definition of differential $K$-theory to the geometric description, and it also proves useful in defining the ring structure and the pushforward map.

Fix $i \in \mathbb{N}$ for the time being. Let $X$ be a compact $\Gamma$-manifold. Consider quadruples
\begin{displaymath}
(W,f,\eta,\omega)
\end{displaymath}
where $W \in \mathcal{G}_{\Gamma,i}$ and
\begin{displaymath}
(f,\eta,\omega) \in \textrm{Map}(X,\mathcal{O}_W^{-i})_\Gamma \times \Omega^{-i-1}_\Gamma(X) \times \Omega^{-i}_\Gamma(X)_\textrm{cl}
\end{displaymath}
such that
\begin{displaymath}
d_\Gamma\eta = \omega - f^*\tilde{\omega}_W^{-i}
\end{displaymath}
Equivalence relations among quadruples are generated by the following:
\begin{enumerate}
\item If $f: X \rightarrow \mathcal{O}_W^{-i} \cap \mathcal{O}_{W'}^{-i}$, then
\begin{displaymath}
(W,f,\eta,\omega) \sim \left(W',f,\eta - f^*\widetilde{\textrm{CS}}^{-i}_{W,W'},\omega\right)
\end{displaymath}
\item If $F: I \times X \rightarrow \mathcal{O}_W^{-i}$ is a homotopy from $f$ to $f'$, then
\begin{displaymath}
(W,f,\eta,\omega) \sim \left(W,f',\eta - \int_{I \times X/X}F^*\tilde{\omega}_W^{-i},\omega\right)
\end{displaymath}
\item If $\beta \in \Omega_\Gamma^{-i-2}(X)$, then
\begin{displaymath}
(W,f,\eta,\omega) \sim (W,f,\eta + d_\Gamma\beta,\omega)
\end{displaymath} \end{enumerate}

\begin{defn} \label{defn:bigdef2} Let $M_\Gamma^{-i}(X)$ denote the set of equivalence classes of such quadruples. \end{defn}

Notice that $M_\Gamma^{-i}(X)$ fits into a short exact sequence
\begin{equation} \label{eq:shortexact4}
0 \longrightarrow \frac{H^{-i-1}_\Gamma(X;\mathbb{R})}{\textrm{ch}^{-i-1}_\Gamma K_\Gamma^{-i-1}(X)} \longrightarrow M_\Gamma^{-i}(X) \stackrel{c}{\longrightarrow} A^{-i}_\Gamma(X) \longrightarrow 0
\end{equation}
The injection is given by
\begin{displaymath}
\eta \mapsto (W,i,\eta,0)
\end{displaymath}
where $i$ is a constant map, and the surjection is given by
\begin{displaymath}
(W,f,\eta,\omega) \mapsto (x,\omega)
\end{displaymath}
where $x$ is the class obtained by pushing the class $\psi_i^*\lbrack V_W \rbrack$ forward to $\mathcal{O}_W^{-i}$ and then pulling it back to $X$ by $f$. $M_\Gamma^{-i}(X)$ also fits into the sequence
\begin{equation} \label{eq:shortexact5}
0 \longrightarrow \frac{\Omega^{-i-1}_\Gamma(X)}{\Omega_\Gamma^{-i-1}(X)_K} \longrightarrow M_\Gamma^{-i}(X) \stackrel{c}{\longrightarrow} K^{-i}_\Gamma(X) \longrightarrow 0
\end{equation}
where the injection is given by
\begin{displaymath}
\eta \mapsto (W,i,\eta,d_\Gamma \eta)
\end{displaymath}

\begin{prop} \label{prop:moddef} If $X$ is a compact $\Gamma$-manifold, then
\begin{displaymath}
\check{K}_\Gamma^{-i}(X) \cong M_\Gamma^{-i}(X)
\end{displaymath} \end{prop}

\begin{proof} Fix a representative $\tilde{\omega}^{-i} \in \Omega^{-i}_\Gamma(\boldsymbol{\mathfrak{F}}_\Gamma^{-i})$ for the universal Chern character. Since $\tilde{\omega}_W^{-i}$ represents the restriction of the universal Chern character to $\mathcal{O}_W^{-i}$ for each finite-dimensional invariant subspace $W \subset \mathcal{H}_\Gamma^1$, we may choose
\begin{displaymath}
\sigma_W \in \Omega^{-i-1}_\Gamma(\mathcal{O}_W^{-i})
\end{displaymath}
so that
\begin{displaymath}
d_\Gamma\sigma_W = \tilde{\omega}^{-i} - \tilde{\omega}_W^{-i}
\end{displaymath}
on $\mathcal{O}_W^{-i}$. Any two such choices differ by an exact form, since the degree $(-i-1)$ cohomology of $\mathcal{O}_W^{-i}$ is trivial. The isomorphism
\begin{displaymath}
M_\Gamma^{-i}(X) \longrightarrow \check{K}^{-i}_\Gamma(X)
\end{displaymath}
is then defined on the level of representatives by
\begin{displaymath}
(W,f,\eta,\omega) \mapsto (f,\eta - f^*\sigma_W,\omega)
\end{displaymath}
Clearly this map respects equivalence relations (2) and (3). Thus, in order to show that it is well-defined on the level of classes, we must show that, if $f$ maps $X$ into $\mathcal{O}_W^{-i} \cap \mathcal{O}_{W'}^{-i}$, then
\begin{displaymath}
\sigma_W - \sigma_{W'} = \widetilde{\textrm{CS}}_{W,W'}^{-i}
\end{displaymath}
up to an exact form. This is easy to see if $W \subset W'$, since then $\mathcal{O}_W^{-i} \subset \mathcal{O}_{W'}^{-i}$, and we know that the degree $(-i-1)$ cohomology of $\mathcal{O}_W^{-i}$ is trivial. But the result follows in general by simply considering $W,W' \subset W + W'$. We have seen that, on $\mathcal{O}_W^{-i} \cap \mathcal{O}_{W'}^{-i}$,
\begin{displaymath}
\widetilde{\textrm{CS}}_{W,W'}^{-i} = \widetilde{\textrm{CS}}_{W,W+W'}^{-i} - \widetilde{\textrm{CS}}_{W',W+W'}^{-i}
\end{displaymath}
up to an exact form; on the other hand,
\begin{displaymath}
\widetilde{\textrm{CS}}_{W,W+W'}^{-i} = \sigma_W - \sigma_{W+W'}
\end{displaymath}
and
\begin{displaymath}
\widetilde{\textrm{CS}}_{W',W+W'}^{-i} = \sigma_{W'} - \sigma_{W+W'}
\end{displaymath}
as we have just argued. So the map is well-defined. The proof follows from an application of the Five Lemma to the short exact sequences (\ref{eq:shortexact}) and (\ref{eq:shortexact4}). \end{proof}

\subsection{A geometric definition}

Our second alternative definition is more geometric in flavor. This description follows Klonoff's account of ordinary differential $K$-theory and is in the spirit of Lott's geometric description of $K$-theory with coefficients in $\mathbb{R}/\mathbb{Z}$ (cf. \cite{Lo}), whereas the definition we have already given is adapted from \cite{HS} and the sketch in \cite{F2}. We restrict our consideration to degree zero.

Let $X$ be a compact $\Gamma$-manifold.  Consider triples
\begin{displaymath}
(V,\nabla,\eta)
\end{displaymath}
where $V$ is a complex $\mathbb{Z}/2\mathbb{Z}$-graded $\Gamma$-equivariant vector bundle over $X$ with $\Gamma$-invariant connection $\nabla$ that is compatibly unitary, and
\begin{displaymath}
\eta \in \frac{\Omega^{-1}_\Gamma(X)}{d_\Gamma\,\Omega^{-2}_\Gamma(X)}
\end{displaymath}
Equivalence relations on the abelian group of formal sums and differences of such triples are generated by:
\begin{enumerate}
\item If $(V,\nabla)$ and $(V',\nabla')$ are vector bundles with unitary connection such that there exists an equivariant vector bundle isomorphism $\varphi: V \rightarrow V'$, then
\begin{displaymath}
(V,\nabla,\eta) \sim \left(V',\nabla',\eta - \textrm{CS}_\Gamma(V;\nabla,\varphi^*\nabla')\right)
\end{displaymath}
\item If $(V,\nabla,\eta)$ and $(V',\nabla',\eta)$ are two triples, then
\begin{displaymath}
(V,\nabla,\eta) + (V',\nabla',\eta') \sim (V \oplus V',\nabla \oplus \nabla',\eta + \eta')
\end{displaymath}
\item If $V$ is a graded vector bundle, let $\Pi V$ denote the bundle $V$ with the opposite grading. Then
\begin{displaymath}
(V,\nabla,\eta) + (\Pi V,\nabla,-\eta) \sim 0
\end{displaymath} \end{enumerate}

\begin{defn}\label{defn:bigdef3} Let $L_\Gamma(X)$ denote the set of equivalence classes of such triples. \end{defn}

\noindent The additive group $L_\Gamma(X)$ has a natural ring structure, with multiplication given by the graded tensor product of vector bundles and the wedge product of forms.  Namely, if $(V,\nabla,\eta)$ and $(V',\nabla',\eta')$ are representatives for classes in $L_\Gamma(X)$, then the product of the classes is represented by
\begin{equation}
\left(V \otimes V',\nabla \hat{\otimes} \nabla',\omega(V;\nabla) \wedge \eta' + \eta \wedge \omega(V';\nabla') + \eta \wedge d_\Gamma\eta'\right)
\end{equation}
where
\begin{equation}
\nabla \hat{\otimes} \nabla' \equiv \nabla \otimes \textrm{id} + \textrm{id} \otimes \nabla'
\end{equation}

\begin{prop} \label{prop:geomdesc} If $X$ is a compact $\Gamma$-manifold, then
\begin{displaymath}
\check{K}_\Gamma^0(X) \cong L_\Gamma(X)
\end{displaymath}
\end{prop}

\noindent We begin with a lemma.

\begin{lem} \label{lem:geomlem} $L_\Gamma(X)$ fits into a short exact sequence
\begin{equation} \label{eq:geomshort}
0 \longrightarrow \frac{H^{-1}_\Gamma(X;\mathbb{R})}{\textrm{\emph{ch}}^{-1}_\Gamma K_\Gamma^{-1}(X)} \longrightarrow L_\Gamma(X) \stackrel{c}{\longrightarrow} A^0_\Gamma(X) \longrightarrow 0
\end{equation} \end{lem}

\begin{proof} First, we define the (surjective) homomorphism
\begin{displaymath}
\alpha: L_\Gamma(X) \longrightarrow A_\Gamma^0(X)
\end{displaymath}
by
\begin{displaymath}
(V,\nabla,\eta) \mapsto \left(\lbrack V \rbrack, \omega(V;\nabla) + d_\Gamma\eta\right)
\end{displaymath}
We now need to show that the kernel of $\alpha$ is isomorphic to
\begin{displaymath}
\frac{H^{-1}_\Gamma(X;\mathbb{R})}{\textrm{ch}^{-1}_\Gamma K_\Gamma^{-1}(X)}
\end{displaymath}
Define
\begin{displaymath}
\beta: H^{-1}_\Gamma(X;\mathbb{R}) \longrightarrow L_\Gamma(X)
\end{displaymath}
by
\begin{displaymath}
\lbrack \eta \rbrack \mapsto \lbrack(0,0,\eta)\rbrack
\end{displaymath}
We must prove that
\begin{displaymath}
\textrm{im}(\beta) = \textrm{ker}(\alpha) \qquad \textrm{and} \qquad \textrm{ker}(\beta) = \textrm{ch}^{-1}_\Gamma K_\Gamma^{-1}(X)
\end{displaymath}

We begin with the former. It should be clear that $\textrm{im}(\beta) \subset \textrm{ker}(\alpha)$.  On the other hand, suppose that $\check{x} \in \textrm{ker}(\alpha)$.  We may represent $\check{x}$ by a triple of the form
\begin{displaymath}
(V^0 \oplus \Pi V^0,\nabla^0 \oplus \nabla^1,\eta)
\end{displaymath}
Set
\begin{displaymath}
\sigma = \textrm{CS}_\Gamma(V^0 \oplus \Pi V^0; \nabla^0 \oplus \nabla^1, \nabla^0 \oplus \nabla^0)
\end{displaymath}
Then
\begin{eqnarray}
(V^0 \oplus \Pi V^0,\nabla^0 \oplus \nabla^1,\eta) & \sim & (V^0 \oplus \Pi V^0,\nabla^0 \oplus \nabla^0,\eta + \sigma) {}\nonumber\\
& \sim & (0,0,\eta + \sigma) {}
\end{eqnarray}
Notice that $\eta + \sigma$ is closed.  It follows that $\check{x} \in \textrm{im}(\beta)$.  Thus, $\textrm{ker}(\alpha) \subset \textrm{im}(\beta)$.

Now we wish to show that $\textrm{ker}(\beta) = \textrm{ch}^{-1}_\Gamma K_\Gamma^{-1}(X)$.  Suppose first that $\lbrack \eta \rbrack \in \textrm{ch}^{-1}_\Gamma K_\Gamma^{-1}(X)$.  Then there exists a graded equivariant vector bundle $V \rightarrow S^1 \times X$ satisfying
\begin{displaymath}
\int_{S^1 \times X/X}\textrm{ch}_\Gamma(V) = \lbrack \eta \rbrack
\end{displaymath}
Let $h: I \times X \rightarrow S^1 \times X$ be the identification map given by identifying $(x,0)$ with $(x,1)$ for all $x \in X$.  Define two vector bundles
\begin{displaymath}
W_i = (h|_{\lbrace i \rbrace \times X})^*V \longrightarrow X \qquad i = 0,1
\end{displaymath}
Let $\nabla_0$ be a connection for $W_0$. Choose a trivialization for $h^*V$ in the $I$-direction. Then translation along $I$ induces a vector bundle isomorphism $\tau: W_1 \rightarrow W_0$. On the other hand, the identification map induces an isomorphism $\iota: W_0 \rightarrow W_1$, and
\begin{displaymath}
\lbrack\textrm{CS}_\Gamma(W_0;\nabla_0,\iota^*\tau^*\nabla_0)\rbrack = \lbrack \eta \rbrack
\end{displaymath}
Thus,
\begin{eqnarray}
(0,0,\eta) & \sim & (W_0 \oplus \Pi W_0,\nabla_0 \oplus \nabla_0,\eta) {}\nonumber\\
& \sim & (W_1 \oplus \Pi W_0,\tau^*\nabla_0 \oplus \nabla_0,0) {}\nonumber\\
& \sim & (0,0,0) {}
\end{eqnarray}
It follows that $\lbrack \eta \rbrack \in \textrm{ker}(\beta)$.  Similarly, if $\lbrack \eta \rbrack \in \textrm{ker}(\beta)$, then we may use this fact to construct a vector bundle $V$ over $S^1 \times X$ whose Chern character integrates to $\lbrack \eta \rbrack$.
\end{proof}

As we have stated, one motivation for developing a model using finite-dimensional approximations is that it allows us to compare the universal connections on the canonical open sets with connections on a compact space $X$. Let us see how this works before proceeding to the proof of Proposition~\ref{prop:geomdesc}. Suppose that $V \rightarrow X$ is a graded equivariant vector bundle with invariant, compatibly unitary connection $\nabla^V$, and let $f: X \rightarrow \mathcal{O}_W$ be a classifying map for $\lbrack V \rbrack \in K_\Gamma^0(X)$, for some $W \in \mathcal{G}_\Gamma$. We seek to construct a ``Chern-Simons form''
\begin{displaymath}
\sigma_W\left(V;f,\nabla^V\right) \in \Omega_\Gamma^{-1}(X)
\end{displaymath}
satisfying
\begin{equation}\label{eq:glowfrog2}
d_\Gamma\sigma_W\left(V;f,\nabla^V\right) = \omega\left(V;\nabla^V\right) - f^*\tilde{\omega}_W
\end{equation}
Let $V' \rightarrow X$ be a vector bundle complementary to $V^1$ (as ungraded bundles), so that $V^1 \oplus V'$ is trivial. Assume that $\textrm{dim}\,W = \textrm{rk}\,V^1 \oplus V'$ by adding trivial vector bundles if necessary. It follows from our constructions that $V^0 \oplus V'$ is isomorphic to $f^*V^0_W$. Set
\begin{equation} \label{eq:cs}
\sigma_W\left(V;f,\nabla^V\right) \equiv \textrm{CS}_\Gamma\left(f^*V^0_W;f^*\nabla_W,\nabla^V\right)
\end{equation}
Then $\sigma_W\left(V;f,\nabla^V\right)$ satisfies (\ref{eq:glowfrog2}). Notice that, if $f(X) \subset \mathcal{O}_W \cap \mathcal{O}_{W'}$, then
\begin{equation} \label{eq:lovely}
\sigma_W\left(V;f,\nabla^V\right) = f^*\widetilde{\textrm{CS}}_{W,W'} + \sigma_{W'}\left(V;f,\nabla^V\right)
\end{equation}
up to an element of $\Omega^{-1}_\Gamma(X)_K$.

\begin{proof}[Proof of Proposition~\ref{prop:geomdesc}] Rather than prove $L_\Gamma(X) \cong \check{K}_\Gamma^0(X)$ directly, we construct an isomorphism
\begin{equation} \label{eq:spring}
L_\Gamma(X) \stackrel{\sim}{\longrightarrow} M_\Gamma^0(X)
\end{equation}
and apply Proposition~\ref{prop:moddef}.

This map is defined on the level of representatives as follows. Take any triple $(V,\nabla,\eta)$ representing a class in $L_\Gamma(X)$. Let $f$ be an equivariant map $X \rightarrow \mathcal{B}_\Gamma$ representing $\lbrack V \rbrack$. Let $W \in \mathcal{G}_\Gamma$ be an invariant subspace so that $f(X) \subset \mathcal{O}_W$. To $(V,\nabla,\eta)$, associate the quadruple
\begin{displaymath}
\left( W, f, \sigma_W(V;f,\nabla) + \eta, \omega(V;\nabla) + d_\Gamma\eta \right)
\end{displaymath}
where $\sigma_W(V;f,\nabla)$ is defined as in equation (\ref{eq:cs}). That the correspondence is independent of the choice of $W$, up to equivalence, follows from (\ref{eq:lovely}). Similarly, if $F: I \times X \rightarrow \mathcal{O}_W$ is a homotopy from $f_1$ to $f_2$, then
\begin{displaymath}
\sigma_W(V;f_0,\nabla) - \sigma_W(V;f_1,\nabla)
\end{displaymath}
differs from
\begin{displaymath}
\int_{I \times X/X}F^*\tilde{\omega}_W
\end{displaymath}
by an exact form. It follows that the correspondence is independent of the choice of $f$ as well. It thus descends to a homomorphism $L_\Gamma(X) \rightarrow M_\Gamma^0(X)$. The short exact sequences (\ref{eq:geomshort}) and (\ref{eq:shortexact4}) fit into the commutative diagram
\begin{equation}
\xymatrix{
& & L_\Gamma(X) \ar[dr] \ar[d] & & \\
0 \ar[r] & \frac{H^{-1}_\Gamma(X;\mathbb{R})}{\textrm{ch}_\Gamma^{-1}K^{-1}_\Gamma(X)} \ar[r] \ar[ur] & M_\Gamma^0(X) \ar[r] & A_\Gamma^0(X) \ar[r] & 0 }
\end{equation}
An application of the Five Lemma therefore proves that (\ref{eq:spring}) is an isomorphism.  \end{proof}


\section{Ring structure}

In this section, we seek to define a ring structure for differential equivariant $K$-theory which agrees with the ring structures for ordinary equivariant $K$-theory (given by direct sum and tensor product of vector bundles) and differential forms (addition and wedge product of forms). Our discussion parallels ideas developed independently by Gomi using a slightly different model. We use here the modified definition for differential $K$-theory, Definition~\ref{defn:bigdef2}. In a previous version of this paper, we had claimed to construct the ring structure under the original definition, and we owe our gratitude to U. Bunke for pointing out an error in our construction. At present, we do not know of a way to handle the ring structure under the original definition without some sort of ``universal connections,'' but we developed Definition~\ref{defn:bigdef2} precisely to handle such deficits. It would also seem that using Definition~\ref{defn:bigdef2} allows one to sidestep the uniqueness issues raised by Bunke and Schick in \cite{BS2}.

\subsection{Addition}

We begin with a preliminary sketch. For simplicity's sake, fix $i = 0$ for the time being. We first recall how one constructs a pairing
\begin{displaymath}
\textrm{Map}(X,\boldsymbol{\mathfrak{F}}_\Gamma^0)_\Gamma \times \textrm{Map}(X,\boldsymbol{\mathfrak{F}}_\Gamma^0)_\Gamma \rightarrow \textrm{Map}(X,\boldsymbol{\mathfrak{F}}_\Gamma^0)_\Gamma
\end{displaymath}
which descends on homotopy classes of maps to addition in ordinary equivariant $K$-theory. Suppose that $x,x' \in K^0(X)$ are classes represented by the maps $f,f' \in \textrm{Map}(X,\boldsymbol{\mathfrak{F}}_\Gamma^0)_\Gamma$, respectively. Let $F$ be the map given by
\begin{displaymath}
p \mapsto \left( \begin{array}{c c}
f(p) & 0\\
0    & f'(p)
\end{array} \right)
\end{displaymath}
One would then like to say that $x + x'$ is represented by $F$. There is a technical issue here, however: $F$ is a map into the space of odd skew-adjoint Fredholm operators on $\mathcal{H}_\Gamma \oplus \mathcal{H}_\Gamma$, not $\mathcal{H}_\Gamma$. It is necessary, therefore, to choose an isomorphism of the Hilbert spaces; one then applies Kuiper's Theorem to conclude that the sum $x + x'$ thus obtained is independent of this choice.

We shall apply a similar argument in constructing an addition map for differential equivariant $K$-theory. Namely, if $\check{x},\check{x}' \in \check{K}^0(X)$ are classes represented by the quadruples $(W,f,\eta,\omega)$ and $(W',f',\eta',\omega')$, respectively, then their sum $\check{x} + \check{x}' \in \check{K}^0(X)$ should be represented by a quadruple roughly of the form
\begin{displaymath}
(W \oplus W',F, \eta + \eta' + \lambda_{W,W'}, \omega + \omega')
\end{displaymath}
The term $\lambda_{W,W'}$ is the pullback by $F$ of a differential form, a ``universal Chern-Simons form'' on $\boldsymbol{\mathfrak{F}}_\Gamma^0 \times \boldsymbol{\mathfrak{F}}_\Gamma^0$ which compensates for the difference between $F^*\tilde{\omega}_{W \oplus W'}$ and $f^*\tilde{\omega}_W + (f')^*\tilde{\omega}_{W'}$. Kuiper's Theorem then suffices to prove that the addition map is well-defined on the level of equivalence classes.

Let us make this sketch rigorous. First fix a $\Gamma$-linear isomorphism
\begin{equation}
A: \mathcal{H}_\Gamma \oplus \mathcal{H}_\Gamma \longrightarrow \mathcal{H}_\Gamma
\end{equation}
and let $A_i$  denote the isomorphism
\begin{displaymath}
(J^i)^{-1} \circ A \circ \left( J^i \oplus J^i \right): (\mathcal{H}_\Gamma \otimes \mathbb{C}\textrm{l}_i) \oplus (\mathcal{H}_\Gamma \otimes \mathbb{C}\textrm{l}_i) \longrightarrow \mathcal{H}_\Gamma \otimes \mathbb{C}\textrm{l}_i
\end{displaymath}
Let
\begin{equation}
s_i: \boldsymbol{\mathfrak{F}}_\Gamma^{-i} \times \boldsymbol{\mathfrak{F}}_\Gamma^{-i} \longrightarrow \boldsymbol{\mathfrak{F}}_\Gamma^{-i}
\end{equation}
be the map given by
\begin{equation}
(T_1,T_2) \mapsto A_i(T_1 \oplus T_2)A^{-1}_i
\end{equation}
Let $\pi_\varepsilon$ denote projection onto the $\varepsilon^\textrm{th}$ factor, $\varepsilon = 1,2$. Take $W,W' \in \mathcal{G}_{\Gamma,i}$. Then $s_i$ restricts to a map
\begin{displaymath}
\mathcal{O}_W^{-i} \times \mathcal{O}_{W'}^{-i} \longrightarrow \mathcal{O}_{A(W \oplus W')}^{-i}
\end{displaymath}
Notice that $A$ induces an isomorphism between the vector bundles
\begin{displaymath}
(\pi_1 \times \textrm{id})^* \psi_i^* V_W \oplus (\textrm{id} \times \pi_2)^* \psi_i^* V_W \longrightarrow I^i \times \mathcal{O}_W^{-i} \times \mathcal{O}_{W'}^{-i}
\end{displaymath}
and
\begin{displaymath}
(\textrm{id} \times s_i)^* \psi_i^* V_{A(W \oplus W')} \longrightarrow I^i \times \mathcal{O}_W^{-i} \times \mathcal{O}_{W'}^{-i}
\end{displaymath}
Let $\sigma_{W,W'}^{-i}$ be the Chern-Simons form from the connection
\begin{displaymath}
(\textrm{id} \times s_i)^* \psi_i^* \nabla_{A(W \oplus W')}
\end{displaymath}
to the connection
\begin{displaymath}
(\pi_1 \times \textrm{id})^* \psi_i^* \nabla_W \oplus (\textrm{id} \times \pi_2)^* \psi_i^* \nabla_{W'}
\end{displaymath}
Define
\begin{displaymath}
\lambda^{-i}_{W,W'} \equiv \int_{I^i \times \mathcal{O}_W^{-i} \times \mathcal{O}_{W'}^{-i} / \mathcal{O}_W^{-i} \times \mathcal{O}_{W'}^{-i}} \sigma_{W,W'}^{-i}
\end{displaymath}
Both connections are trivial on the boundary $\partial I^i \times \mathcal{O}_W^{-i} \times \mathcal{O}_{W'}^{-i}$, so
\begin{displaymath}
d_\Gamma \lambda_{W,W'}^{-i} = \pi_1^*\tilde{\omega}^{-i}_W + \pi_2^*\tilde{\omega}^{-i}_{W'} - s_i^*\tilde{\omega}^{-i}_{A(W \oplus W')}
\end{displaymath}
on $\mathcal{O}_W^{-i} \times \mathcal{O}_{W'}^{-i}$.

We now have what we need to define addition on the level of representatives for differential $K$-theory classes. (For convenience, we shall occasionally write representatives as column vectors.) If $(W,f,\eta,\omega)$ and $(W',f',\eta',\omega')$ are quadruples representing classes in $\check{K}_\Gamma^{-i}(X)$, we set
\begin{equation}
(W,f,\eta,\omega) + (W',f',\eta',\omega') \equiv
\left[ \begin{array}{c}
A(W \oplus W')\\
s_i\circ(f \times f')\\
\eta + \eta' + (f \times f')^*\lambda^{-i}_{W,W'}\\
\omega + \omega' \end{array} \right]
\end{equation}
We must check that this addition map is well-defined as a map on representative triples and descends to $\check{K}_\Gamma^{-i}(X)$, and that the addition on $\check{K}_\Gamma^{-i}(X)$ thus obtained is independent of the choices we have made, commutative, and associative.

That addition map is well-defined on the level of triples follows from our definition of $\lambda^{-i}_{W,W'}$. To show that it descends to equivalence classes, we must demonstrate that it respects the generating equivalence relations. Suppose, for instance, that $(W,f,\eta,\omega)$ and $(W',f',\eta',\omega')$ are two representatives, and that $W'' \in \mathcal{G}_{\Gamma,i}$ so that $f(X) \subset \mathcal{O}_{W''}^{-i}$. According to equivalence relation (1),
\begin{displaymath}
(W,f,\eta,\omega) \sim \left(W'',f,\eta - f^*\widetilde{\textrm{CS}}_{W,W''}^{-i},\omega\right)
\end{displaymath}
Now,
\begin{displaymath}
\left[ \begin{array}{c} W\\ f\\ \eta\\ \omega\end{array} \right]
+ \left[ \begin{array}{c} W'\\ f'\\ \eta'\\ \omega'\end{array} \right]
= \left[ \begin{array}{c}
A(W \oplus W')\\
s_i\circ(f \times f')\\
\eta + \eta' + (f \times f')^*\lambda^{-i}_{W,W'}\\
\omega + \omega' \end{array} \right]
\end{displaymath}
whereas
\begin{displaymath}
\left[ \begin{array}{c} W''\\ f\\ \eta - f^*\widetilde{\textrm{CS}}_{W,W''}^{-i}\\ \omega\end{array} \right]
+ \left[ \begin{array}{c} W'\\ f'\\ \eta'\\ \omega'\end{array} \right]
= \left[ \begin{array}{c}
A(W'' \oplus W')\\
s_i\circ(f \times f')\\
\eta + \eta' - f^*\widetilde{\textrm{CS}}_{W,W''}^{-i} + (f \times f')^*\lambda^{-i}_{W'',W'}\\
\omega + \omega' \end{array} \right]
\end{displaymath}
Furthermore,
\begin{displaymath}
(f \times f')^*\lambda^{-i}_{W,W'} - (f \times f')^*s_i^*\widetilde{\textrm{CS}}_{A(W \oplus W'),A(W'' \oplus W')}^{-i}
\end{displaymath}
differs from
\begin{displaymath}
(f \times f')^*\lambda^{-i}_{W'',W'} - f^*\widetilde{\textrm{CS}}_{W,W''}^{-i}
\end{displaymath}
by an element of $\Omega_\Gamma^{-i-1}(X)_K$ (see remark~\ref{rem:cs}), so we apply equivalence relation (1) once again to see that
\begin{equation}
\left[ \begin{array}{c} W\\ f\\ \eta\\ \omega\end{array} \right]
+ \left[ \begin{array}{c} W'\\ f'\\ \eta'\\ \omega'\end{array} \right]
\sim \left[ \begin{array}{c} W''\\ f\\ \eta - f^*\widetilde{\textrm{CS}}_{W,W''}^{-i}\\ \omega\end{array} \right]
+ \left[ \begin{array}{c} W'\\ f'\\ \eta'\\ \omega'\end{array} \right]
\end{equation}
Similar arguments show that addition respects relations (2) and (3), and thus that the addition map descends to $\check{K}_\Gamma^{-i}(X)$.

To show that the definition is independent of the choice of $A$, suppose that $\overline{A}$ is another such isomorphism, and let $\overline{s}_i$ be the corresponding map constructed as before.  By Kuiper's Theorem, the space of $\Gamma$-linear isomorphisms from $\mathcal{H}_\Gamma \oplus \mathcal{H}_\Gamma$ to $\mathcal{H}_\Gamma$ is contractible, so there exists a path of isomorphisms $\lbrace A_t \rbrace_{t \in I}$ from $A$ to $\overline{A}$. Let $A_{i,t}$ denote the isomorphism
\begin{displaymath}
(J^i)^{-1} \circ A_t \circ (J^i \oplus J^i)
\end{displaymath}
Define
\begin{equation}
S_i: I \times \boldsymbol{\mathfrak{F}}_\Gamma^{-i} \times \boldsymbol{\mathfrak{F}}_\Gamma^{-i} \longrightarrow \boldsymbol{\mathfrak{F}}_\Gamma^{-i}
\end{equation}
by
\begin{equation}
(t,T_1,T_2) \mapsto A_{i,t}(T_1 \oplus T_2)A_{i,t}^{-1}
\end{equation}
Then $S_i$ is a homotopy from $s_i$ to $\overline{s}_i$.

Now suppose that $(W,f,\eta,\omega)$ and $(W',f',\eta',\omega')$ are quadruples representing classes in $\check{K}_\Gamma^{-i}(X)$. Let
\begin{displaymath}
W'' = A(W \oplus W') + \overline{A}(W \oplus W') \in \mathcal{G}_{\Gamma,i}
\end{displaymath}
Then
\begin{equation}\label{eq:quadruple1}
\left[ \begin{array}{c}
A(W \oplus W')\\
s_i\circ(f \times f')\\
\eta + \eta' + (f \times f')^*\lambda^{-i}_{W,W'}\\
\omega + \omega' \end{array} \right]
\end{equation}
is equivalent to
\begin{displaymath}
\left[ \begin{array}{c}
W''\\
s_i\circ(f \times f')\\
\eta + \eta' + (f \times f')^*\left( \lambda^{-i}_{W,W'} - s_i^*\widetilde{\textrm{CS}}^{-i}_{A(W \oplus W'),W''} \right) \\
\omega + \omega' \end{array} \right]
\end{displaymath}
Similarly,
\begin{equation}\label{eq:quadruple2}
\left[ \begin{array}{c}
\overline{A}(W \oplus W')\\
\overline{s}_i\circ(f \times f')\\
\eta + \eta' + (f \times f')^*\overline{\lambda}^{-i}_{W,W'}\\
\omega + \omega' \end{array} \right]
\end{equation}
is equivalent to
\begin{displaymath}
\left[ \begin{array}{c}
W''\\
\overline{s}_i\circ(f \times f')\\
\eta + \eta' + (f \times f')^*\left( \overline{\lambda}^{-i}_{W,W'} - \overline{s}_i^*\widetilde{\textrm{CS}}^{-i}_{\overline{A}(W \oplus W'),W''} \right) \\
\omega + \omega' \end{array} \right]
\end{displaymath}
Now,
\begin{displaymath}
F \equiv S_i \circ (f \times f')
\end{displaymath}
is a homotopy from $s_i \circ (f \times f')$ to $\overline{s}_i \circ (f \times f')$, and
\begin{displaymath}
(f \times f')^*\left( \lambda^{-i}_{W,W'} - s_i^* \widetilde{\textrm{CS}}^{-i}_{A(W \oplus W'),W''} \right) - \int_{I \times X / X} F^*\tilde{\omega}^{-i}_{W''}
\end{displaymath}
differs from
\begin{displaymath}
(f \times f')^*\left( \overline{\lambda}^{-i}_{W,W'} - \overline{s}_i\widetilde{\textrm{CS}}^{-i}_{\overline{A}(W \oplus W'),W''} \right)
\end{displaymath}
by an element of $\Omega^{-i-1}_\Gamma(X)_K$. (See remark~\ref{rem:cs}. We assume without loss of generality that the image of $I \times X$ under $F$ lies in $\mathcal{O}_{W''}^{-i}$, adding an element of $\mathcal{G}_{\Gamma,i}$ to $W''$ if necessary.) It follows that (\ref{eq:quadruple1}) is equivalent to (\ref{eq:quadruple2}). This proves independence of the choice of $A$.

The proofs of commutativity and associativity proceed along similar lines.  First, let $B \in \textrm{GL}(\mathcal{H}_\Gamma \oplus \mathcal{H}_\Gamma)$ be the operator that reverses vector components, and set
\begin{displaymath}
B_i = (J^i \oplus J^i)^{-1} \circ B \circ (J^i \oplus J^i)
\end{displaymath}
Let $\lbrace B_t \rbrace_{t \in I}$ be a path from $A$ to $A \circ B$; use this to define a path $\lbrace B_{i,t} \rbrace_{t \in I}$ from $A_i$ to $A_i \circ B_i$, as before. Define
\begin{equation}
R_i: I \times \boldsymbol{\mathfrak{F}}_\Gamma^{-i} \times \boldsymbol{\mathfrak{F}}_\Gamma^{-i} \longrightarrow \boldsymbol{\mathfrak{F}}_\Gamma^{-i}
\end{equation}
by
\begin{equation}
(t,T_1,T_2) \mapsto B_{i,t}(T_1 \oplus T_2)B_{i,t}^{-1}
\end{equation}
Let $r_i$ be the map on $\boldsymbol{\mathfrak{F}}_\Gamma^{-i} \times \boldsymbol{\mathfrak{F}}_\Gamma^{-i}$ that exchanges factors.  Then $R_i$ is an equivariant homotopy from $s_i \circ r_i$ to $s_i$.  An argument analogous to the one just given suffices to prove commutativity.

Again, to prove associativity, define isomorphisms
\begin{equation}
C_i, \overline{C}_i: \mathcal{H}_\Gamma \oplus \mathcal{H}_\Gamma \oplus \mathcal{H}_\Gamma \longrightarrow \mathcal{H}_\Gamma
\end{equation}
by
\begin{equation}
C_i = A_i \circ (A_i \oplus \textrm{id}) \qquad \overline{C}_i = A_i \circ (\textrm{id} \oplus A_i)
\end{equation}
and let $\lbrace C_{i,t} \rbrace_{t \in I}$ be a path of isomorphisms from $C_i$ to $\overline{C}_i$.  Then the map given by
\begin{displaymath}
(t,T_1,T_2,T_3) \mapsto C_{i,t}(T_1 \oplus T_2 \oplus T_3)C_{i,t}^{-1}
\end{displaymath}
is an equivariant homotopy from $s_i \circ (s_i \times \textrm{id})$ to $s_i \circ (\textrm{id} \times s_i)$.  Proceed as before.

\begin{defn} Let $X$ be a compact $\Gamma$-manifold. If $\check{x},\check{x}' \in \check{K}^{-i}(X)$ are classes represented by the triples $(W,f,\eta,\omega)$ and $(W',f',\eta',\omega')$ respectively, then $\check{x} + \check{x}' \in \check{K}^{-i}(X)$ is the class represented by the triple
\begin{equation}
\left(A(W \oplus W'),s_i \circ (f \times f'), \eta + \eta' + (f \times f')^*\lambda^{-i}_{W,W'}, \omega + \omega'\right)
\end{equation}
\end{defn}

\noindent We have shown:

\begin{prop} This addition is well-defined, independent of choices, commutative, and associative.\end{prop}

\subsection{Multiplication}

The arguments here given are similar to the ones we applied in our definition of addition. First, fix a $\Gamma$-linear isomorphism
\begin{equation}
L: \mathcal{H}_\Gamma^{\otimes 2} \longrightarrow \mathcal{H}_\Gamma
\end{equation}
and let $L_{i,j}$ denote the isomorphism
\begin{displaymath}
(J^{i+j})^{-1} \circ L \circ \left( J^i \otimes J^j \right) : (\mathcal{H}_\Gamma \otimes \mathbb{C}\textrm{l}_i) \otimes (\mathcal{H}_\Gamma \otimes \mathbb{C}\textrm{l}_j) \longrightarrow \mathcal{H}_\Gamma \otimes \mathbb{C}\textrm{l}_{i+j}
\end{displaymath}
Given $T_1 \in \boldsymbol{\mathfrak{F}}_\Gamma^{-i}$ and $T_2 \in \boldsymbol{\mathfrak{F}}_\Gamma^{-j}$, set
\begin{equation}
T_1 \hat{\otimes} T_2 \equiv T_1 \otimes \textrm{id} + \textrm{id} \otimes T_2
\end{equation}
and define
\begin{equation}
m_{i,j}: \boldsymbol{\mathfrak{F}}_\Gamma^{-i} \times \boldsymbol{\mathfrak{F}}_\Gamma^{-j} \longrightarrow \boldsymbol{\mathfrak{F}}_\Gamma^{-i-j}
\end{equation}
by
\begin{equation}
(T_1,T_2) \mapsto L_{i,j}(T_1 \hat{\otimes} T_2)L^{-1}_{i,j}
\end{equation}

Take $W \in \mathcal{G}_{\Gamma,i}$ and $W' \in \mathcal{G}_{\Gamma,j}$. Then the forms
\begin{displaymath}
\pi_1^*\tilde{\omega}^{-i}_W \wedge \pi_2^*\tilde{\omega}^{-j}_{W'} \qquad \textrm{and} \qquad m_{i,j}^*\tilde{\omega}^{-i-j}_{L(W \otimes W')}
\end{displaymath}
are cohomologous on $\mathcal{O}_W^{-i} \times \mathcal{O}_{W'}^{-j}$. Let $\sigma_{W,W'}^{-i,-j}$ be the Chern-Simons form for the vector bundle
\begin{displaymath}
(\pi_1 \times \textrm{id})^* \psi_i^* V_W \otimes (\textrm{id} \times \pi_2)^* \psi_j^* V_{W'} \longrightarrow I^{i+j} \times \mathcal{O}_W^{-i} \times \mathcal{O}_{W'}^{-j}
\end{displaymath}
from the connection
\begin{displaymath}
(m_{i,j} \times \textrm{id})^* \psi_{i+j}^* \nabla_{L(W \otimes W')}
\end{displaymath}
to the connection
\begin{displaymath}
(\pi_1 \times \textrm{id})^* \psi_i^* \nabla_W \hat{\otimes} (\textrm{id} \times \pi_2)^* \psi_j^* \nabla_{W'}
\end{displaymath}
and define
\begin{displaymath}
\kappa^{-i,-j}_{W,W'} \equiv \int_{I^{i+j} \times \mathcal{O}_W^{-i} \times \mathcal{O}_{W'}^{-j}/\mathcal{O}_W^{-i} \times \mathcal{O}_{W'}^{-j}} \sigma_{W,W'}^{-i,-j} \in \Omega^{-i-j-1}_\Gamma(\mathcal{O}_W^{-i} \times \mathcal{O}_{W'}^{-j})
\end{displaymath}
Then
\begin{displaymath}
d_\Gamma\kappa^{-i,-j}_{W,W'} = \pi_1^*\tilde{\omega}^{-i}_W \wedge \pi_2^*\tilde{\omega}^{-j}_{W'} - m_{i,j}^*\tilde{\omega}^{-i-j}_{L(W \otimes W')}
\end{displaymath}

If $(W,f,\eta,\omega)$ represents a class in $\check{K}_\Gamma^{-i}(X)$ and $(W',f',\eta',\omega')$ a class in $\check{K}_\Gamma^{-j}(X)$, then we set
\begin{equation}
(f,\eta,\omega) \cdot (f',\eta',\omega') \equiv
\left[ \begin{array}{c}
L(W \otimes W')\\
m_{i,j}\circ(f \times f')\\
(-1)^i f^*\tilde{\omega}^{-i}_W \wedge \eta' + \eta \wedge \omega' + (f \times f')^*\kappa^{-i,-j}_{W,W'}\\
\omega \wedge \omega' \end{array} \right]
\end{equation}
We must check that this multiplication map is well-defined and descends to $\check{K}_\Gamma^\bullet(X)$, and that the multiplication on $\check{K}_\Gamma^\bullet(X)$ thus obtained is independent of choices, graded commutative, associative, and distributive with respect to addition.

One easily checks that multiplication descends to $\check{K}_\Gamma^\bullet(X)$; the proof is similar to that for addition. Independence of the choice of $L$ is proved by an application of Kuiper's Theorem, as before. The proof of graded commutativity is similar, but more complicated; we begin with some preliminary remarks.

The classifying space $\boldsymbol{\mathfrak{F}}_\Gamma^{-i}$ carries a natural action of the symmetric group $\textrm{Sym}(i)$ acting on $i$ letters. For each $\sigma \in \textrm{Sym}(i)$, let $H_\sigma$ be the linear map on $\mathbb{C}\textrm{l}_i$ given by
\begin{displaymath}
e_{j_1} \cdot e_{j_2} \cdots e_{j_k} \mapsto e_{\sigma(j_1)} \cdot e_{\sigma(j_2)} \cdots e_{\sigma(j_k)}
\end{displaymath}
where $1 \leq j_1 < j_2 < \cdots j_k \leq i$. Extend $H_\sigma$ linearly to $\mathcal{H}_\Gamma \otimes \mathbb{C}\textrm{l}_i$. This induces a transformation
\begin{displaymath}
h_\sigma: \boldsymbol{\mathfrak{F}}_\Gamma^{-i} \longrightarrow \boldsymbol{\mathfrak{F}}_\Gamma^{-i} \qquad T \mapsto H_\sigma^{-1} \circ T \circ H_\sigma
\end{displaymath}
Clearly,
\begin{displaymath}
h_\sigma\left(\mathcal{O}_W^{-i}\right) = \mathcal{O}_W^{-i}
\end{displaymath}
for $W \in \mathcal{G}_{\Gamma,i}$. Let $k_\sigma$ be the map on $I^i$ given by
\begin{displaymath}
(t_1,t_2,\dots,t_i) \mapsto (t_{\sigma(1)},t_{\sigma(2)},\dots,t_{\sigma(i)})
\end{displaymath}
Then
\begin{eqnarray}\label{eq:rosary}
h_\sigma^*\tilde{\omega}_W^{-i} & = & \int_{I^i \times \mathcal{O}_W^{-i}/\mathcal{O}_W^{-i}} (\textrm{id} \times h_\sigma)^*\psi_i^* \tilde{\omega}_W {}\nonumber\\
& = & \int_{I^i \times \mathcal{O}_W^{-i}/\mathcal{O}_W^{-i}} (k_\sigma \times \textrm{id})^*\psi_i^* \tilde{\omega}_W {}\nonumber\\
& = & (-1)^{|\sigma|}\int_{I^i \times \mathcal{O}_W^{-i}/\mathcal{O}_W^{-i}} \psi_i^* \tilde{\omega}_W {}\nonumber\\
& = & (-1)^{|\sigma|}\tilde{\omega}_W^{-i} {}
\end{eqnarray}

\begin{lem} \label{lem:comm2} Suppose $\check{x} \in \check{K}_\Gamma^{-i}(X)$, $i > 0$, is represented by the quadruple $(W,f,\eta,\omega)$.  Let $\sigma \in \textrm{\emph{Sym}}(i)$.  Then the quadruple
\begin{equation}
\left(W,h_\sigma \circ f, (-1)^{|\sigma|}\eta, (-1)^{|\sigma|}\omega\right)
\end{equation}
is a representative for $(-1)^{|\sigma|}\check{x}$.
\end{lem}

\begin{proof} We will show this explicitly in the case that $i = 2$ and $\sigma$ is the transposition $(1,2)$; the proof for the general case will then follow easily. Let $\mathfrak{G}_\Gamma$ denote the graph of $h_\sigma$ in $\boldsymbol{\mathfrak{F}}_\Gamma^{-2} \times \boldsymbol{\mathfrak{F}}_\Gamma^{-2}$. There is an equivariant homotopy from the restriction of $s_2$ to $\mathfrak{G}_\Gamma$ to a constant map. To see this, first consider the following general remarks.

Suppose $M$ is a graded complex $\mathbb{C}\textrm{l}_2$-module given by $\gamma: \mathbb{C}\textrm{l}_2 \rightarrow \textrm{End}(M)$. Let $M'$ be the module which is equal to $M$ as a vector space but on which $\mathbb{C}\textrm{l}_2$ acts under the algebra homomorphism given by $e_j \mapsto e_{\sigma(j)}$. We extend $M \oplus M'$ to a $\mathbb{C}\textrm{l}_3$-module in the following way: let
\begin{displaymath}
\tilde{\gamma}: \mathbb{C}\textrm{l}_3 \longrightarrow \textrm{End}(M \oplus M')
\end{displaymath}
be given by
\begin{displaymath}
\tilde{\gamma}(e_1) = \left( \begin{array}{c c}
\gamma(e_1) & 0\\
0 & \gamma(e_2) \end{array} \right)
\qquad \tilde{\gamma}(e_2) = \left( \begin{array}{c c}
\gamma(e_2) & 0\\
0 & \gamma(e_1) \end{array} \right)
\end{displaymath}
\begin{displaymath}
\tilde{\gamma}(e_3) = \frac{1}{\sqrt{2}}\left( \begin{array}{c c}
0 & \gamma(e_1) - \gamma(e_2) \\
\gamma(e_1) - \gamma(e_2) & 0 \end{array} \right)
\end{displaymath}
Then
\begin{displaymath}
\lbrace \tilde{\gamma}(e_1), \tilde{\gamma}(e_2), \tilde{\gamma}(e_3) \rbrace \subset \textrm{End}(M \oplus M')
\end{displaymath}
is a set of odd operators obeying the Clifford relations.

In just the same way, we extend
\begin{displaymath}
(\mathcal{H}_\Gamma \otimes \mathbb{C}\textrm{l}_2) \oplus (\mathcal{H}_\Gamma \otimes \mathbb{C}\textrm{l}_2)
\end{displaymath}
to a $\mathbb{C}\textrm{l}_3$-module. The operator
\begin{displaymath}
\tilde{\gamma}(e_3) \in \textrm{End}((\mathcal{H}_\Gamma \otimes \mathbb{C}\textrm{l}_2) \oplus (\mathcal{H}_\Gamma \otimes \mathbb{C}\textrm{l}_2))
\end{displaymath}
thus obtained commutes in the graded sense with every operator of the form
\begin{displaymath}
\left( \begin{array}{c c}
T & 0\\
0 & h_\sigma(T) \end{array} \right) \qquad \textrm{for } T \in \boldsymbol{\mathfrak{F}}_\Gamma^{-2}
\end{displaymath}
We thus obtain an operator $\Phi \in \boldsymbol{\mathfrak{F}}_\Gamma^{-2}$ which squares to $-\textrm{id}$ and commutes in the graded sense with every element in the image of $\mathfrak{G}_\Gamma$ under $s_2$. Define
\begin{displaymath}
F: I \times \mathfrak{G}_\Gamma\longrightarrow \boldsymbol{\mathfrak{F}}_\Gamma^{-2}
\end{displaymath}
by
\begin{equation}
\begin{array}{c}
(t,T,h_\sigma(T)) \mapsto \cos\left(\frac{\pi}{2}t\right) s_2(T,h_\sigma(T)) + \sin\left(\frac{\pi}{2}t\right) \Phi \end{array}
\end{equation}
Then $F$ is a homotopy from $s_2|_{\mathfrak{G}_\Gamma}$ to the constant map with image $\Phi$.

Now suppose that $\check{x} \in \check{K}_\Gamma^{-2}(X)$ is represented by the quadruple $(W,f,\eta,\omega)$.  Then
\begin{equation} \label{eq:charliebrown}
\left[ \begin{array}{c} W\\ f\\ \eta\\ \omega \end{array} \right]
+ \left[ \begin{array}{c} W\\ h_\sigma \circ f\\ -\eta\\ -\omega \end{array} \right]
= \left[ \begin{array}{c} A(W \oplus W)\\ s_2 \circ (f \times (h_\sigma \circ f))\\ (f \times (h_\sigma \circ f))^*\lambda^{-2}_{W,W}\\ 0 \end{array} \right]
\end{equation}
which is equivalent to
\begin{equation} \label{eq:charliebrown2}
\left[ \begin{array}{c} A(W \oplus W)\\ \Phi\\ (f \times (h_\sigma \circ f))^*\left( \lambda^{-2}_{W,W} - \int_{I \times \mathfrak{G}_\Gamma/\mathfrak{G}_\Gamma}F^*\tilde{\omega}_W^{-2} \right) \\ 0 \end{array} \right]
\end{equation}
Now,
\begin{displaymath}
d_\Gamma \lambda^{-2}_{W,W} = \pi_1^* \tilde{\omega}_W^{-2} + \pi_2^* \tilde{\omega}_W^{-2} - s_2^*\tilde{\omega}_{A(W \oplus W)}^{-2}
\end{displaymath}
 Equation (\ref{eq:rosary}) implies that $h_\sigma^* \tilde{\omega}_W^{-2} = -\tilde{\omega}_W^{-2}$. It follows that $\pi_1^* \tilde{\omega}_W^{-2} = -\pi_2^* \tilde{\omega}_W^{-2}$ on the open subset $(\mathcal{O}_W^{-2} \times \mathcal{O}_W^{-2}) \cap \mathfrak{G}_\Gamma$ of $\mathfrak{G}_\Gamma$, and thus that
\begin{equation}\label{eq:pencap}
\lambda^{-2}_{W,W} - \int_{I \times \mathfrak{G}_\Gamma/\mathfrak{G}_\Gamma}F^*\tilde{\omega}_W^{-2}
\end{equation}
is closed. The degree $-3$ cohomology of $(\mathcal{O}_W^{-2} \times \mathcal{O}_W^{-2}) \cap \mathfrak{G}_\Gamma \simeq \mathcal{O}_W^{-2}$ is trivial, so (\ref{eq:pencap}) is exact as well, and the sum (\ref{eq:charliebrown}) is zero.

An argument analogous to the foregoing proves the lemma for any transposition in $\textrm{Sym}(i)$. Since any $\sigma \in \textrm{Sym}(i)$ may be written as the product of transpositions, the general proof follows by an inductive argument. \end{proof}

Let $Q \in \textrm{GL}(\mathcal{H}_\Gamma^{\otimes2})$ be the isomorphism defined on homogeneous vectors by
\begin{displaymath}
v \otimes w \mapsto (-1)^{|v||w|}w \otimes v,
\end{displaymath}
and let $Q_{i,j}$ denote the isomorphism on $(\mathcal{H}_\Gamma \otimes \mathbb{C}\textrm{l}_i) \otimes (\mathcal{H}_\Gamma \otimes \mathbb{C}\textrm{l}_j)$ defined as
\begin{displaymath}
\left(J^i \otimes J^j\right)^{-1} \circ Q \circ (J^i \otimes J^j)
\end{displaymath}
There exists a path $\lbrace Q_t \rbrace_{t \in I}$ from $L$ to $L \circ Q$; use this to define a path $\lbrace Q_{i,j,t} \rbrace_{t \in I}$ from $L_{i,j}$ to $L_{i,j} \circ Q_{i,j}$, as before. Define
\begin{equation}
N_{i,j}: I \times \boldsymbol{\mathfrak{F}}_\Gamma^{-i} \times \boldsymbol{\mathfrak{F}}_\Gamma^{-j} \longrightarrow \boldsymbol{\mathfrak{F}}_\Gamma^{-i-j}
\end{equation}
by
\begin{equation}
(t,T_1,T_2) \mapsto Q_{i,j,t}(T_1 \hat{\otimes} T_2)Q_{i,j,t}^{-1}
\end{equation}
Then $N_{i,j}$ is a homotopy from $m_{i,j}$ to $h_\sigma \circ m_{j,i} \circ r_{i,j}$, where
\begin{displaymath}
r_{i,j}: \boldsymbol{\mathfrak{F}}^{-i} \times \boldsymbol{\mathfrak{F}}_\Gamma^{-j} \longrightarrow \boldsymbol{\mathfrak{F}}_\Gamma^{-j} \times \boldsymbol{\mathfrak{F}}_\Gamma^{-i}
\end{displaymath}
is the map that exchanges factors, and
\begin{equation}
\sigma = \left( \begin{array}{c c c c c c}
1 & \cdots & j & j + 1 & \cdots & i + j\\
i + 1 & \cdots & i + j & 1 & \cdots & i \end{array} \right) \in \textrm{Sym}(i + j)
\end{equation}
Note that the parity of $\sigma$ is $ij$.

Now suppose that $\check{x}_1 \in \check{K}_\Gamma^{-i}(X)$ and $\check{x}_2 \in \check{K}_\Gamma^{-j}(X)$ with representatives $(W_1,f_1,\eta_1,\omega_1)$ and $(W_2,f_2,\eta_2,\omega_2)$, respectively. We may assume without loss of generality that $W_1 = W_2 = W$. Then the product $\check{x}_2 \cdot \check{x}_1$ is represented by
\begin{displaymath}
\left[ \begin{array}{c}
L(W \otimes W)\\
m_{j,i} \circ (f_2 \times f_1)\\
(-1)^j f_2^*\tilde{\omega}^{-j}_W \wedge \eta_1 + \eta_2 \wedge \omega_1 + (f_2 \times f_1)^*\kappa^{-j,-i}_{W,W}\\
\omega_2 \wedge \omega_1 \end{array} \right]
\end{displaymath}
It therefore follows from Lemma~\ref{lem:comm2} that $(-1)^{ij}\check{x}_2 \cdot \check{x}_1$ is represented by
\begin{displaymath}
\left[ \begin{array}{c}
L(W \otimes W)\\
h_\sigma \circ m_{j,i} \circ (f_2 \times f_1)\\
(-1)^{ij+j} f_2^*\tilde{\omega}^{-j}_W \wedge \eta_1 + (-1)^{ij}\eta_2 \wedge \omega_1 + (-1)^{ij}(f_2 \times f_1)^*\kappa^{-j,-i}_{W,W}\\
(-1)^{ij}\omega_2 \wedge \omega_1 \end{array} \right]
\end{displaymath}
which is equal to
\begin{equation} \label{eq:zorba}
\left[ \begin{array}{c}
L(W \otimes W)\\
h_\sigma \circ m_{j,i} \circ r_{i,j} \circ (f_1 \times f_2)\\
\eta_1 \wedge f_2^*\tilde{\omega}^{-j}_W + (-1)^i \omega_1 \wedge \eta_2 + (-1)^{ij}(f_1 \times f_2)^*r_{i,j}^*\kappa^{-j,-i}_{W,W}\\
\omega_1 \wedge \omega_2 \end{array} \right]
\end{equation}
Here, we have just used the commutation relations for differential forms and the fact that $f_2 \times f_1 = r_{i,j} \circ (f_1 \times f_2)$. Next, notice that
\begin{displaymath}
\eta_1 \wedge f_2^*\tilde{\omega}^{-j}_W + (-1)^i \omega_1 \wedge \eta_2
\end{displaymath}
differs from
\begin{displaymath}
(-1)^i f_1^*\tilde{\omega}^{-i}_W \wedge \eta_2 + \eta_1 \wedge \omega_2
\end{displaymath}
by an exact form, so (\ref{eq:zorba}) is equivalent to
\begin{equation} \label{eq:zorba2}
\left[ \begin{array}{c}
L(W \otimes W)\\
h_\sigma \circ m_{j,i} \circ r_{i,j} \circ (f_1 \times f_2)\\
(-1)^i f_1^*\tilde{\omega}^{-i}_W \wedge \eta_2 + \eta_1 \wedge \omega_2 + (-1)^{ij}(f_1 \times f_2)^*r_{i,j}^*\kappa^{-j,-i}_{W,W}\\
\omega_1 \wedge \omega_2 \end{array} \right]
\end{equation}
Finally,
\begin{displaymath}
(f_1 \times f_2)^* \left( (-1)^{ij}r_{i,j}^*\kappa^{-j,-i}_{W,W} + \int_{I \times \boldsymbol{\mathfrak{F}}^{-i} \times \boldsymbol{\mathfrak{F}}_\Gamma^{-j}/\boldsymbol{\mathfrak{F}}^{-i} \times \boldsymbol{\mathfrak{F}}_\Gamma^{-j}} N_{i,j}^* \tilde{\omega}^{-i-j}_{L(W \otimes W)} \right)
\end{displaymath}
differs from
\begin{displaymath}
(f_1 \times f_2)^* \kappa^{-i,-j}_{W,W}
\end{displaymath}
by an element of $\Omega^{-i-j-1}_\Gamma(X)_K$, so (\ref{eq:zorba2}) is equivalent to
\begin{displaymath}
\left[ \begin{array}{c}
L(W \otimes W)\\
m_{i,j} \circ (f_1 \times f_2)\\
(-1)^i f_1^*\tilde{\omega}^{-i}_W \wedge \eta_2 + \eta_1 \wedge \omega_2 + (f_1 \times f_2)^*\kappa^{-i,-j}_{W,W}\\
\omega_1 \wedge \omega_2 \end{array} \right]
\end{displaymath}
which is a representative for $\check{x}_1 \cdot \check{x}_2$. It follows that
\begin{equation}
\check{x}_1 \cdot \check{x}_2 = (-1)^{ij}\,\check{x}_2 \cdot \check{x}_1
\end{equation}
This proves graded commutativity.

The arguments to show associativity and distributivity are similar to what has gone before.  First, define isomorphisms $R$ and $\overline{R}$ from $\mathcal{H}_\Gamma^{\otimes 3}$ to $\mathcal{H}_\Gamma$ by
\begin{equation}
R = L \circ (L \otimes \textrm{id}) \qquad \overline{R} = L \circ (\textrm{id} \otimes L)
\end{equation}
Let $\lbrace R_t \rbrace_{t \in I}$ be a path of isomorphisms from $R$ to $\overline{R}$. Use it to define a path of Clifford linear maps
\begin{displaymath}
R_{i,j,t}: (\mathcal{H}_\Gamma \otimes \mathbb{C}\textrm{l}_i) \otimes (\mathcal{H}_\Gamma \otimes \mathbb{C}\textrm{l}_j) \otimes (\mathcal{H}_\Gamma \otimes \mathbb{C}\textrm{l}_k) \longrightarrow \mathcal{H}_\Gamma \otimes \mathbb{C}\textrm{l}_{l}
\end{displaymath}
where $l = i+j+k$. Define
\begin{equation}
q_{i,j,k}: I \times \boldsymbol{\mathfrak{F}}_\Gamma^{-i} \times \boldsymbol{\mathfrak{F}}_\Gamma^{-j} \times \boldsymbol{\mathfrak{F}}_\Gamma^{-k} \longrightarrow \boldsymbol{\mathfrak{F}}_\Gamma^{-l}
\end{equation}
by
\begin{equation}
(t,T_1,T_2,T_3) \mapsto R_{i,j,t}(T_1 \hat{\otimes} T_2 \hat{\otimes} T_3)R_{i,j,t}^{-1}
\end{equation}
Then $q_{i,j,k}$ is a homotopy from $m_{i+j,k} \circ (m_{i,j} \times \textrm{id})$ to $m_{i,j+k} \circ (\textrm{id} \times m_{j,k})$.  Proceed as before.  This suffices to prove associativity.

Finally, to prove distributivity, let $\lbrace P_t \rbrace_{t \in I}$ be a path of isomorphisms from $\mathcal{H}_\Gamma \otimes (\mathcal{H}_\Gamma \oplus \mathcal{H}_\Gamma)$ to $\mathcal{H}_\Gamma$ with endpoints
\begin{equation}
P = A \circ (L \oplus L) \qquad \textrm{and} \qquad \overline{P} = L \circ (\textrm{id} \otimes A)
\end{equation}
and use this to construct the needed homotopy.

\begin{defn} If $\check{x} \in \check{K}_\Gamma^{-i}(X)$ and $\check{x}' \in \check{K}_\Gamma^{-j}(X)$ are classes represented by the quadruples $(W,f,\eta,\omega)$ and $(W',f',\eta',\omega')$, respectively, then $\check{x} \cdot \check{x}' \in \check{K}_\Gamma^{-i-j}(X)$ is the class represented by
\begin{displaymath}
\left[ \begin{array}{c}
L(W \otimes W')\\
m_{i,j}\circ(f \times f')\\
(-1)^i f^*\tilde{\omega}^{-i}_W \wedge \eta' + \eta \wedge \omega' + (f \times f')^*\kappa^{-i,-j}_{W,W'}\\
\omega \wedge \omega' \end{array} \right]
\end{displaymath}
\end{defn}

\noindent We have shown:

\begin{prop} This multiplication is well-defined, independent of choices, graded commutative, associative, and distributive with respect to addition. \end{prop}


\section{Pushforward}

The main goal of Section 5 is to construct the pushforward map. The first two subsections are devoted to preliminary material. We then begin by defining the pushforward from the differential equivariant $K$-theory of $Y \times W$, where $W$ is a spin $\Gamma$-representation, to that of $Y$. In the fourth subsection, we give the definition of the pushforward map
\begin{displaymath}
\check{K}^{-i}_\Gamma(X) \longrightarrow \check{K}^{-i-n}_\Gamma(Y)
\end{displaymath}
where $X \rightarrow Y$ is a compact equivariant fiber bundle with $n$-dimensional fibers such that the relative tangent bundle has a spin structure preserved by $\Gamma$-action. We establish some of the properties of the pushforward map in the fifth subsection, and, in particular, we there prove a version of Stokes' Theorem.

An analytic formula for the pushforward from the differential equivariant $K$-theory of an odd-dimensional spin $\Gamma$-manifold to the torus $\check{K}^\textrm{odd}_\Gamma(\textrm{pt})$ is conjectured in Section 6. This formula is analogous to the formula given by the Index Theorem for even-dimensional spaces. It has been established in certain special cases.

\subsection{Differential $K$-theory with compact support}

We define differential $K$-theory with compact support only for real vector bundles over compact manifolds. (Here, we regard a vector space as a vector bundle over a point.) We find it most convenient to give our definition in terms of rapidly-decreasing forms rather than compactly-supported forms, since the superconnection machinery developed by Mathai and Quillen in \cite{MQ} plays a major role in our formulation of the pushforward map.

Suppose that $E \rightarrow X$ is a real equivariant vector bundle over a compact $\Gamma$-manifold $X$.  Let
\begin{displaymath}
\Omega_{\Gamma,\textrm{rd}}^{-i}(E) \subset \Omega^{-i}_\Gamma(E)
\end{displaymath}
consist of those forms that are rapidly-decreasing in the fiber-wise direction (cf. \cite{MQ}).  One can relate rapidly-decreasing forms to forms with compact vertical support in the following way.  Fix a $\Gamma$-invariant metric on E, and let $B(E) \rightarrow X$ denote the open unit ball bundle.  Define
\begin{displaymath} \begin{array}{c}
\Psi: B(E) \longrightarrow E \qquad \textrm{by} \qquad e \mapsto \frac{e}{\sqrt{1-|e|^2}} \end{array}
\end{displaymath}
Then pullback via $\Psi$ defines an isomorphism
\begin{displaymath}
\Omega_{\Gamma,\textrm{rd}}^\bullet(E) \longrightarrow \Omega^\bullet_\Gamma(E,E\setminus B(E)) \hookrightarrow \Omega^\bullet_{\Gamma,\textrm{c}}(E)
\end{displaymath}
On the other hand, there is an obvious inclusion map
\begin{displaymath}
\Omega_{\Gamma,\textrm{c}}^\bullet(E) \hookrightarrow \Omega_{\Gamma,\textrm{rd}}^\bullet(E)
\end{displaymath}
The composition of these two maps in either order descends to the identity map on cohomology.

Next, let $D(E)$ be the closure of $B(E)$, and let
\begin{displaymath}
\textrm{Map}_0(D(E),\boldsymbol{\mathfrak{F}}_\Gamma^{-i})_\Gamma \subset \textrm{Map}(D(E),\boldsymbol{\mathfrak{F}}_\Gamma^{-i})_\Gamma
\end{displaymath}
be the space of maps such that for each $x \in X$, the restriction to $\partial D(E)_x$ is a map into the contractible space of invertible operators.  Define
\begin{displaymath}
\textrm{Map}_0(D(E),\boldsymbol{\mathfrak{F}}_\Gamma^{-i})_\Gamma \longrightarrow \textrm{Map}(E, \boldsymbol{\mathfrak{F}}_\Gamma^{-i})_\Gamma
\end{displaymath}
by
\begin{displaymath}
f \mapsto f \circ \Psi_g^{-1}
\end{displaymath}
The range of this map is independent of the metric; let it be denoted by $\textrm{Map}_\textrm{rd}(E,\boldsymbol{\mathfrak{F}}_\Gamma^{-i})_\Gamma$.

\begin{defn} \label{theorem:cvdef} Suppose $E \rightarrow X$ is a real equivariant vector bundle over a compact $\Gamma$-manifold $X$.  The differential equivariant $K$-theory of $E$ with compact support is defined as follows: $\check{K}_{\Gamma,\textrm{c}}^{-i}(E)$ is the set of equivalence classes whose representatives are triples
\begin{displaymath}
(f,\eta,\omega) \in \textrm{Map}_\textrm{rd}(E,\boldsymbol{\mathfrak{F}}_\Gamma^i)_\Gamma \times \Omega^{-i-1}_{\Gamma,\textrm{rd}}(E) \times \Omega^{-i}_{\Gamma,\textrm{rd}}(E)_\textrm{cl}
\end{displaymath}
satisfying
\begin{displaymath}
d_\Gamma\eta = \omega - f^*\tilde{\omega}^{-i}
\end{displaymath}
Equivalence relations are as in Definition~\ref{defn:bigdef}. Alternatively, we may define $\check{K}_{\Gamma,\textrm{c}}^0(E)$ as the set of equivalence classes $(W,f,\eta,\omega)$, where $W \in \mathcal{G}_\Gamma$ and
\begin{displaymath}
(f,\eta,\omega) \in \textrm{Map}_\textrm{rd}(E,\mathcal{O}_W)_\Gamma \times \Omega^{-1}_{\Gamma,\textrm{rd}}(E) \times \Omega^0_{\Gamma,\textrm{rd}}(E)_\textrm{cl}
\end{displaymath}
satisfying
\begin{displaymath}
d_\Gamma\eta = \omega - f^*\tilde{\omega}_W
\end{displaymath}
Equivalence relation are as in  Definition~\ref{defn:bigdef2}. \end{defn}

The differential equivariant $K$-theory groups with compact support lie in short exact sequences as before, i.e.,
\begin{equation} \label{eq:cshortexact}
0 \longrightarrow \frac{H^{-i-1}_{\Gamma,\textrm{c}}(E;\mathbb{R})}{\textrm{ch}_\Gamma^{-i-1}K_{\Gamma,\textrm{c}}^{-i-1}(E)} \longrightarrow \check{K}^{-i}_{\Gamma,\textrm{c}}(E) \stackrel{c}{\longrightarrow} A^{-i}_{\Gamma,\textrm{rd}}(E) \longrightarrow 0
\end{equation}
and
\begin{equation} \label{eq:cshortexact2}
0 \longrightarrow \frac{\Omega^{-i-1}_{\Gamma,\textrm{rd}}(E)}{\Omega^{-i-1}_{\Gamma,\textrm{rd}}(E)_K} \longrightarrow \check{K}^{-i}_{\Gamma,\textrm{c}}(E) \longrightarrow K^{-i}_{\Gamma,\textrm{c}}(E) \longrightarrow 0
\end{equation}
where
\begin{displaymath}
A^{-i}_{\Gamma,\textrm{rd}}(E) \subset K_{\Gamma,\textrm{c}}^{-i}(E) \times \Omega_{\Gamma,\textrm{rd}}^{-i}(E)
\end{displaymath}
is the set of pairs $(x,\omega)$ satisfying $\textrm{ch}^{-i}_\Gamma(x) = \lbrack \omega \rbrack$.

The ring structure for $\check{K}_{\Gamma,\textrm{c}}^\bullet(E)$ is analogous to the ring structure for $\check{K}_\Gamma^\bullet(X)$.  Notice that $\check{K}_{\Gamma,\textrm{c}}^\bullet(E)$ is a $\check{K}_\Gamma^\bullet(X)$-module, where a class $\check{x} \in \check{K}_\Gamma^\bullet(X)$ acts as $\check{y} \mapsto \pi^*\check{x} \cdot \check{y}$.

\subsection{Representatives for characteristic classes}

Let $E \rightarrow X$ be a real oriented equivariant vector bundle such that $\Gamma$ preserves the orientation. If we choose a $\Gamma$-invariant metric for $E$, then $\Gamma$ acts by orientation-preserving isometries, so $\Gamma$-action lifts to the oriented orthonormal frame bundle $\mathcal{P}_\textrm{SO}(E) \rightarrow X$.

\begin{defn} We say that an oriented real $\Gamma$-equivariant vector bundle $E \rightarrow X$ is a spin $\Gamma$-equivariant vector bundle if it is spin and if $\Gamma$ preserves the spin structure, i.e., there exists a lift of $\Gamma$-action to the principal bundle $\mathcal{P}_\textrm{Spin}(E) \rightarrow X$ such that the quotient map $\mathcal{P}_\textrm{Spin}(E) \rightarrow \mathcal{P}_\textrm{SO}(E)$ is $\Gamma$-equivariant. If $E$ is the tangent bundle to $X$, then we say simply that $X$ is a spin $\Gamma$-manifold.  The choice of a $\Gamma$-equivariant spin structure for $E$ amounts to the choice of spin structure together with a lift of the $\Gamma$-action to $\mathcal{P}_\textrm{Spin}(E)$. \end{defn}

Let $E \rightarrow X$ be a real spin $\Gamma$-equivariant vector bundle, and choose a $\Gamma$-equivariant spin structure. Fix a $\Gamma$-invariant metric for $E$ and a compatible $\Gamma$-invariant connection $\nabla^E$.

Fix $g \in \Gamma$ for the time being. This group element acts on the fibers of the restriction $E_g$ of $E$ to $X^g$, and the fixed point set $E^g \subset E$ of the action of $g$ on $E$ is the subbundle of $E_g$ whose fibers are the vector subspaces on which $g$ acts trivially.  On the other hand, the orthogonal complement $E^{g\perp}$ to $E^g$ decomposes into a direct sum of $2$-plane bundles on which $g$ acts by rotation through an angle, and a bundle on which $g$ acts as multiplication by $-1$.  Since $g$ acts by orientation-preserving isomorphisms, the latter bundle must be of even rank as well.  It follows that $E^{g\perp}$ is of even rank. Set $k_g = \frac{1}{2}\textrm{rk}\,E^{g\perp}$, conceived of as a locally constant integer-valued function on $X^g$.

The lift of $g$-action to $\mathcal{P}_\textrm{Spin}(E_g)$ yields an orientation for $E^{g\perp}$ (cf. the argument in \cite[$\S$6.4]{BGV}). If $E^{g\perp}$ already has a natural orientation (e.g., if $X^g$ is a point), then the two orientations may not agree; let $\varepsilon_g$ be the disparity, conceived of as a locally constant function on $X^g$ taking values in $\lbrace \pm1 \rbrace$.

We define the \emph{equivariant $\hat{\textrm{\emph{A}}}$-form} by
\begin{equation}
\hat{\alpha}_g(E;\nabla^E) \equiv \varepsilon_g\,i^{-k_g} \frac{\hat{\alpha}(E^g;\nabla^{E^g})}{\textrm{Pf}\left( \textrm{id} - g \cdot \textrm{exp}\left\{ \frac{i}{2\pi} \Omega^{E^{g\perp}}\,u^{-1} \right\} \right)}
\end{equation}
\begin{equation}
\hat{\alpha}(E;\nabla^E) \equiv \sum \hat{\alpha}_g(E;\nabla^E) \in \Omega^0_\Gamma(X)
\end{equation}
Here, $\nabla^{E^g}$ denotes the connection $\nabla$ restricted to $E^g$, and the form $\hat{\alpha}(E^g;\nabla^g)$ is the representative for the $\hat{\textrm{A}}$-polynomial for $E^g$ obtained via the Chern-Weil method. The curvature restricted to $E^{g\perp}$ is denoted by $\Omega^{E^{g\perp}}$, and $\textrm{Pf}$ denotes the Pfaffian.

The equivariant $\hat{\textrm{A}}$-form is a universal polynomial in characteristic forms with coefficients depending rationally on the eigenvalues of the group action. It is multiplicative in the sense that
\begin{equation}
\hat{\alpha}\left(E_1 \oplus E_2;\nabla^{E_1} \oplus \nabla^{E_2}\right) = \hat{\alpha}\left(E_1;\nabla^{E_1}\right) \wedge \hat{\alpha}\left(E_2;\nabla^{E_2}\right)
\end{equation}
It is independent of the metric and connection up to exactness; let
\begin{equation}
\hat{\textrm{\textbf{A}}}_\Gamma(E) \equiv \sum \hat{\textrm{\textbf{A}}}_g(E) \in H^0_\Gamma(X;\mathbb{R})
\end{equation}
be its cohomology class. We shall refer to this class as the \emph{equivariant $\hat{\textrm{\emph{A}}}$-polynomial}.

The superconnection machinery of \cite{MQ} gives a way of construction a rapidly-decreasing differential form representative for the cohomological Thom class, which we now review. Take $E$, $\nabla^E$ as before, and suppose that $E$ is of even rank $2N$. Form the associated spinor bundle
\begin{displaymath}
S(E) = S^+(E) \oplus S^-(E) \longrightarrow X
\end{displaymath}
Each of $S^\pm(E)$ is an equivariant vector bundle. Let $\pi$ be the projection map $E \rightarrow X$, and let
\begin{displaymath}
\mu: \pi^*S^\pm(E) \longrightarrow \pi^*S^\mp(E)
\end{displaymath}
be the Clifford multiplication map given by
\begin{displaymath}
(e,s) \mapsto (e,e \cdot s)
\end{displaymath}
Note that $\mu$ is an isomorphism on the complement of the zero section of $E$. It thus defines a class
\begin{displaymath}
s(E) \in K_{\Gamma,\textrm{c}}^{2N}(E) \cong K_{\Gamma,\textrm{c}}^0(E)
\end{displaymath}
This is the equivariant $K$-theoretic Thom class for $E \rightarrow X$; the Thom isomorphism
\begin{displaymath}
K^0_\Gamma(X) \stackrel{\sim}{\longrightarrow} K^{2N}_{\Gamma,\textrm{c}}(E)
\end{displaymath}
is given by
\begin{displaymath}
x \mapsto s(E) \cdot \pi^*x
\end{displaymath}
The connection $\nabla^E$ induces a connection $\nabla^{S(E)}$ for $S(E)$, hence a connection $\nabla^{\pi^*S(E)}$ for $\pi^*S(E)$. In the language of \cite{MQ}, the operator
\begin{equation}
\boldsymbol{\nabla}^{\pi^*S(E)} \equiv \pi^*\nabla^{S(E)} + \mu
\end{equation}
is a superconnection for $\pi^*S(E)$. We define the rapidly-decreasing differential form
\begin{equation}
\omega\left(\pi^*S(E);\boldsymbol{\nabla}^{\pi^*S(E)}\right) \in \Omega^0_{\Gamma,\textrm{rd}}(E)
\end{equation}
by
\begin{equation}\label{eq:chernthom}
\omega_g\left(\pi^*{S(E)};\boldsymbol{\nabla}^{\pi^*S(E)}\right) \equiv \textrm{tr}_\textrm{s}\left( g \cdot \textrm{exp}\left\{ \frac{i}{2\pi}\boldsymbol{\Omega}^{\pi^*S(E)}_g u^{-1}\right\} \right)
\end{equation}
This form is a representative for the equivariant Chern character
\begin{displaymath}
\textrm{ch}_\Gamma(s(E)) \in H^0_{\Gamma,\textrm{c}}(E;\mathbb{R})
\end{displaymath}
of $s(E)$. We claim that
\begin{equation}\label{eq:ahatthom}
\upsilon(E;\nabla^E) \equiv \omega\left(\pi^*S(E);\boldsymbol{\nabla}^{\pi^*S(E)}\right) \wedge \pi^* \hat{\alpha}\left(E;\nabla^E\right) \in \Omega^0_{\Gamma,\textrm{rd}}(E)
\end{equation}
is a form representing the cohomological Thom class for $E \rightarrow X$, which is to say that, for each $g \in \Gamma$, the term
\begin{displaymath}
\upsilon_g(E;\nabla^E) \in \Omega_\textrm{rd}(E^g;\pi K_\mathbb{C})^0
\end{displaymath}
is a rapidly-decreasing form representing the cohomological Thom class for the bundle $E^g \rightarrow X^g$. This is a local property in $X^g$, so, to prove it, we may assume that $E^g$ and $E^{g\perp}$ are spin. In that case, our connections for $E^g$, $E^{g\perp}$ induce connections $\nabla^{S(E^g)}$, $\nabla^{S(E^{g\perp})}$ for the respective spinor bundles $S(E^g)$, $S(E^{g\perp})$, and we pull these back to $E^g$, as before. We have
\begin{displaymath}
S(E)_g \simeq S(E^g) \otimes S(E^{g\perp})
\end{displaymath}
The restriction of $\boldsymbol{\nabla}^{\pi^*S(E)}$ to $E^g$ may be written as
\begin{displaymath}
(\nabla^{\pi^*S(E^g)} + \mu) \otimes \textrm{id} + \textrm{id} \otimes \nabla^{\pi^*S(E^{g\perp})}
\end{displaymath}
and (\ref{eq:chernthom}) is of the form
\begin{equation}
\omega\left(\pi^*S(E^g);\nabla^{\pi^*S(E^g)}+\mu\right) \wedge \pi^*\omega_g\left(S(E^{g\perp});\nabla^{S(E^{g\perp})} + \mu\right)
\end{equation}
As is shown in \cite[$\S6.4$]{BGV},
\begin{equation}
\omega_g\left(S(E^{g\perp});\nabla^{S(E^{g\perp})}\right) = \varepsilon_g \, i^{k_g} \, \textrm{Pf}\left( \textrm{id} - g \cdot \textrm{exp}\left\{ \frac{i}{2\pi} \Omega^{E^{g\perp}}\,u^{-1} \right\} \right)
\end{equation}
while it follows from \cite{MQ} that
\begin{equation}
\omega\left(\pi^*S(E^g);\nabla^{\pi^*S(E^g)} + \mu\right) = \upsilon(E^g;\nabla^{E^g}) \wedge \pi^* \hat{\alpha}\left(E^g;\nabla^{E^g}\right)^{-1}
\end{equation}
where $\upsilon(E^g;\nabla^{E^g})$ is a rapidly-decreasing Thom form for $E^g \rightarrow X^g$. This proves the claim. We thus have
\begin{equation}
\int_{E^g/X^g}\omega_g\left(\pi^*S(E);\boldsymbol{\nabla}^{\pi^*S(E)}\right) = \hat{\alpha}_g(E;\nabla^E)^{-1} \cdot u^{k_g-N}
\end{equation}
and, on the level of cohomology classes,
\begin{equation}
\int_{E^g/X^g}\textrm{ch}_g(s(E)) = \hat{\textrm{\textbf{A}}}_g(E)^{-1} \cdot u^{k_g-N}
\end{equation}

\subsection{Integration along a $\Gamma$-representation}

Given an equivariant fiber bundle $X \rightarrow Y$, the topological pushforward of an equivariant $K$-theory class on $X$ is obtained by embedding $X$ in $Y \times W$, where $W$ is $\Gamma$-representation, taking the product of the integrand with the $K$-theoretic Thom class, and integrating along $W$ to the $K$-theory of $Y$ via the Thom isomorphism. We therefore begin our discussion of the pushforward map in differential $K$-theory by defining integration along a $\Gamma$-representation.

\begin{defn} Let $W$ be an orthogonal $\Gamma$-representation.  We say that $W$ is a spin $\Gamma$-representation if $\Gamma$ acts by orientation-preserving isometries and the homomorphism $\Gamma \rightarrow \textrm{SO}(W)$ induced by the representation lifts to a homomorphism $\Gamma \rightarrow \textrm{Spin}(W)$.\end{defn}

\noindent Fix some spin $\Gamma$-representation $W$ of dimension $m$, and let $\psi$ be projection from $W$ to a point.  Then the integration map
\begin{displaymath}
\psi_!: K_{\Gamma,\textrm{c}}^0(W) \longrightarrow K_\Gamma^{-m}(\textrm{pt})
\end{displaymath}
is given by the Thom isomorphism
\begin{displaymath}
K_{\Gamma,\textrm{c}}^0(W) \cong K_\Gamma^0(W^+) \cong K_\Gamma^{-m}(\textrm{pt})
\end{displaymath}
where $W^+$ is the one-point compactification of $W$. Let us recall how this is related to cohomological integration.  If $m$ is odd, then $K^{-m}_\Gamma(\textrm{pt})$ and $K_{\Gamma,\textrm{c}}^0(W)$ are both trivial, and the Thom isomorphism is just the zero map.  Suppose, then, that $m$ is even.  Let
\begin{displaymath}
S(W) = S^+(W) \oplus S^-(W)
\end{displaymath}
be the complex graded spinors, conceived of as a bundle over the point. Notice that Clifford multiplication
\begin{displaymath}
\mu: \psi^*S^\pm(W) \longrightarrow \psi^*S^\mp(W) \qquad (v,s) \mapsto (v,v \cdot s)
\end{displaymath}
gives an isomorphism between $\psi^*S^+(W)$ and $\psi^*S^-(W)$ on the complement of the origin.  It thus determines a class
\begin{displaymath}
s(W) \in K_{\Gamma,\textrm{c}}^0(W)
\end{displaymath}
and this class generates $K_{\Gamma,\textrm{c}}^0(W)$ over $K_\Gamma^0(\textrm{pt})$.  Composing the Thom isomorphism with the isomorphism
\begin{displaymath}
K_\Gamma^{-m}(\textrm{pt}) \cong K_\Gamma^0(\textrm{pt})
\end{displaymath}
given by Bott periodicity yields the isomorphism
\begin{displaymath}
K_{\Gamma,\textrm{c}}^0(W) \stackrel{\sim}{\longrightarrow}  K_\Gamma^0(\textrm{pt}) \qquad s(W) \cdot \psi^*x \mapsto x
\end{displaymath}
Define
\begin{equation} \label{eq:aeven}
a_g(W) \equiv \hat{\alpha}_g(W;d) = \textrm{tr}_\textrm{s}\left( g|_{S(W^{g\perp})} \right) \in \mathbb{C}^\times
\end{equation}
Then, if $x \in K_\Gamma^0(\textrm{pt})$,
\begin{equation}
\int_{W^g}\textrm{ch}_g(s(W) \cdot \psi^*x) = a_g(W)^{-1} \rho_x(g) \cdot u^{-\frac{1}{2}\textrm{dim}\,W^g}
\end{equation}
where $\rho_x$ is the class function given by $x$ under the isomorphism
\begin{displaymath}
K_\Gamma^0(\textrm{pt}) \cong R(\Gamma)
\end{displaymath}

We apply these observations in our definition of an integration map
\begin{displaymath}
\psi_!: \Omega^0_{\Gamma,\textrm{rd}}(W) \longrightarrow \Omega^{-m}_\Gamma(\textrm{pt})
\end{displaymath}
Once again, we may assume that $m$ is even; otherwise this integration is just the zero map.  If $\omega \in \Omega_{\Gamma,\textrm{rd}}^0(W)$, then define its integral $\psi_!\omega$ by
\begin{equation}
(\psi_!\omega)_g = a_g(W) \int_{W^g} \omega_g \cdot u^{-\frac{1}{2}\textrm{dim}\,W^{g\perp}} \qquad g \in \Gamma
\end{equation}
This map descends to cohomology and agrees with integration in equivariant $K$-theory in the sense that, if $(x,\omega) \in A_{\Gamma,\textrm{rd}}^0(X)$, then $(\psi_!x,\psi_!\omega) \in A_\Gamma^{-m}(\textrm{pt})$.  It follows that this integration couples with the Thom isomorphism to give an integration map
\begin{displaymath}
A^0_{\Gamma,\textrm{rd}}(W) \longrightarrow A^{-m}_\Gamma(\textrm{pt})
\end{displaymath}

Finally, we define integration
\begin{displaymath}
\psi_!: \Omega^{-1}_{\Gamma,\textrm{rd}}(W) \longrightarrow \Omega^{-m-1}_\Gamma(\textrm{pt})
\end{displaymath}
as follows. We assume now that $m$ is odd, as the integral will be zero otherwise. Let $W'$ be the $\Gamma$-representation given by adding a trivial $1$-dimensional representation to $W$.  Define
\begin{equation}\label{eq:aodd}
a_g(W) \equiv a_g(W')
\end{equation}
If $\gamma \in \Omega_{\Gamma,\textrm{rd}}^{-1}(W)$ is a differential form representing the equivariant Chern character of $s(W') \in K^{-1}_{\Gamma,\textrm{c}}(W)$, then
\begin{displaymath}
\int_{W^g} \gamma_g = a_g(W)^{-1} u^{-\frac{1}{2}(\textrm{dim}\,W^g+1)}
\end{displaymath}
Thus, given $\eta \in \Omega_{\Gamma,\textrm{rd}}^{-1}(W)$, we define its integral $\psi_!\eta$ by
\begin{equation}
(\psi_!\eta)_g = a_g(W) \int_{W^g} \eta_g \cdot u^{-\frac{1}{2}\textrm{dim}\,W^{g\perp}} \qquad g \in \Gamma
\end{equation}
We have defined this map so as to be compatible with the equivariant Chern character, and it descends to
\begin{displaymath}
\frac{\Omega_{\Gamma,\textrm{c}}^{-1}(W)}{\Omega_{\Gamma,\textrm{c}}^{-1}(W)_K} \longrightarrow \frac{\Omega^{-m-1}_\Gamma(\textrm{pt})}{\Omega^{-m-1}_\Gamma(\textrm{pt})_K}
\end{displaymath}

We now seek to define integration
\begin{displaymath}
\psi_!: \check{K}_{\Gamma,\textrm{c}}^0(W) \longrightarrow \check{K}_\Gamma^{-m}(\textrm{pt})
\end{displaymath}
The integration maps just introduced apply to the short exact sequences (\ref{eq:shortexact})/(\ref{eq:cshortexact}) and (\ref{eq:shortexact2})/(\ref{eq:cshortexact2}) to yield commutative diagrams; for instance, corresponding to the former, we have
\begin{equation} \label{eq:diagram}
\xymatrix{
\frac{H_{\Gamma,\textrm{c}}^{-1}(W;\mathbb{R})}{\textrm{ch}_\Gamma^{-1}K_{\Gamma,\textrm{c}}^{-1}(W)} \ar[r] \ar[d] & \check{K}_{\Gamma,\textrm{c}}^0(W) \ar[r] & A^0_{\Gamma,\textrm{rd}}(W) \ar[d]\\
\frac{H^{-m-1}_\Gamma(\textrm{pt};\mathbb{R})}{\textrm{ch}_\Gamma^{-m-1}K^{-m-1}_\Gamma(\textrm{pt})} \ar[r] & \check{K}_\Gamma^{-m}(\textrm{pt}) \ar[r] & A^{-m}_\Gamma(\textrm{pt}) }
\end{equation}
The inclusion of integration in differential $K$-theory should agree with the given integration maps so as to preserve the commutativity of these diagrams.  This condition suffices to determine its definition.  We consider the even and odd cases separately.  If $m = \textrm{dim}\,W$ is even, then (\ref{eq:diagram}) reduces to
\begin{equation}
\xymatrix{
\check{K}_{\Gamma,\textrm{c}}^0(W) \ar[r]^\sim & A^0_{\Gamma,\textrm{rd}}(W) \ar[d]\\
\check{K}_\Gamma^{-m}(\textrm{pt}) \ar[r]^\sim & A^{-m}_\Gamma(\textrm{pt})}
\end{equation}
and this specifies the integration map.  Now suppose that $m$ is odd.  Then the diagram corresponding to (\ref{eq:shortexact2}), (\ref{eq:cshortexact2}) reduces to
\begin{equation}
\xymatrix{
\frac{\Omega_{\Gamma,\textrm{rd}}^{-1}(W)}{\Omega_{\Gamma,\textrm{rd}}^{-1}(X)_K} \ar[r]^\sim \ar[d] & \check{K}_{\Gamma,\textrm{c}}^0(W) \\
\frac{\Omega^{-m-1}_\Gamma(\textrm{pt})}{\Omega^{-m-1}_\Gamma(\textrm{pt})_K} \ar[r]^\sim & \check{K}_\Gamma^{-m}(\textrm{pt}) }
\end{equation}
and the map is determined by integration of forms.

In order to define integration
\begin{displaymath}
\check{K}_{\Gamma,\textrm{c}}^{-1}(W) \longrightarrow \check{K}_\Gamma^{-1-m}(\textrm{pt})
\end{displaymath}
apply the Thom isomorphism
\begin{displaymath}
K_{\Gamma,\textrm{c}}^{-1}(W) \cong K_{\Gamma,\textrm{c}}^0(W \oplus \mathbb{R})
\end{displaymath}
and integrate as before.  Integration on $\check{K}_{\Gamma,\textrm{c}}^{-i}(W)$ for $i$ arbitrary is then given by Bott periodicity (\ref{eq:bott}).

This definition generalizes to the definition of an integration map
\begin{displaymath}
\psi_!: \check{K}_{\Gamma,\textrm{c}}^{-i}(Y \times W) \longrightarrow \check{K}_\Gamma^{-i-m}(Y)
\end{displaymath}
for any smooth manifold $Y$. The desuspension of a classifying map
\begin{displaymath}
f \in \textrm{Map}_\textrm{rd}(Y \times W,\boldsymbol{\mathfrak{F}}_\Gamma^{-i})_\Gamma
\end{displaymath}
is obtained by regarding $f$ as a map $Y \rightarrow \Omega^m\boldsymbol{\mathfrak{F}}_\Gamma^{-i}$. Recall that
\begin{displaymath}
\Omega^m\boldsymbol{\mathfrak{F}}_\Gamma^{-i} \simeq \boldsymbol{\mathfrak{F}}_\Gamma^{-i-m}
\end{displaymath}
We may therefore assume without loss of generality on the level of differential $K$-theory that $f(Y)$ lies in the image of $\boldsymbol{\mathfrak{F}}_\Gamma^{-i-m}$ under the map
\begin{displaymath}
\boldsymbol{\mathfrak{F}}_\Gamma^{-i-m} \longrightarrow \Omega^m\boldsymbol{\mathfrak{F}}_\Gamma^{-i}
\end{displaymath}
and that the map
\begin{displaymath}
f_W: Y \longrightarrow \boldsymbol{\mathfrak{F}}_\Gamma^{-i-m}
\end{displaymath}
thus obtained is smooth. This yields a map which descends to the Thom isomorphism
\begin{displaymath}
K_{\Gamma,\textrm{c}}^{-i}(Y \times W) \stackrel{\sim}{\longrightarrow} K_\Gamma^{-i-m}(Y)
\end{displaymath}
Furthermore, we have
\begin{equation}
f_W^*\tilde{\omega}^{-i-m} = \psi_!\,f^*\tilde{\omega}^{-i} \in \Omega_\Gamma^{-i-m}(Y)
\end{equation}
Integration over $W$ in differential equivariant $K$-theory is thus defined as follows: if $\check{x} \in \check{K}_{\Gamma,\textrm{c}}^{-i}(Y \times W)$ is represented by the triple $(f,\eta,\omega)$; then $\psi_!\,\check{x} \in \check{K}_\Gamma^{-i-m}(Y)$ is the class represented by the triple
\begin{displaymath}
\left( f_W, \psi_!\,\eta, \psi_!\,\omega \right)
\end{displaymath}

\subsection{The pushforward map}

We begin by recalling the topological pushforward map.

Suppose that $p: X \rightarrow Y$ is a $\Gamma$-equivariant fiber bundle with fiber of dimension $n$, where $X$ is compact, such that the relative tangent bundle $\textrm{T}(X/Y)$ is a spin $\Gamma$-equivariant vector bundle. Choose an equivariant spin structure for $T(X/Y)$. Let $\varphi: X \hookrightarrow W$ be an equivariant embedding of $X$ into a real spin orthogonal $\Gamma$-representation $W$ of dimension $2N+n$. Couple this with $p$ to obtain the embedding
\begin{displaymath}
p \times \varphi: X \hookrightarrow Y \times W
\end{displaymath}
The normal bundle
\begin{displaymath}
\nu \stackrel{\pi}{\longrightarrow} X
\end{displaymath}
to $X$ in $Y \times W$ is a spin equivariant vector bundle over $X$ of rank $2N$, with equivariant spin structure induced by the spin structures of $T(X/Y)$ and $W$. Let
\begin{equation}
U \equiv s(\nu)\in K_{\Gamma,\textrm{c}}^{2N}(\nu) \cong K_{\Gamma,\textrm{c}}^0(\nu)
\end{equation}
be the equivariant $K$-theoretic Thom class. Then the Thom isomorphism
\begin{displaymath}
K^{-i}_\Gamma(X) \stackrel{\sim}{\longrightarrow} K^{-i+2N}_{\Gamma,\textrm{c}}(\nu)
\end{displaymath}
is given by
\begin{displaymath}
x \mapsto U \cdot \pi^*x
\end{displaymath}
Next, identify $\nu$ with a $\Gamma$-invariant tubular neighborhood of $X$ in $Y \times W$ (cf. \cite{Kz}); then the exact sequence of the pair $(Y \times W,\nu)$ induces a homomorphism
\begin{displaymath}
K^{-i+2N}_{\Gamma,\textrm{c}}(\nu) \longrightarrow K^{-i+2N}_{\Gamma,\textrm{c}}(Y \times W)
\end{displaymath}
Composition with the Thom isomorphism yields the pushforward map
\begin{displaymath}
p_!: K^{-i}_\Gamma(X) \longrightarrow K^{-i-n}_\Gamma(Y)
\end{displaymath}
This pushforward is independent of the choices we have made.

To make this pushforward ``differential,'' we need a geometric refinement of $U$, a differential $K$-theoretic Thom class $\check{U} \in \check{K}_{\Gamma,\textrm{c}}^{2N}(\nu)$ which projects to $U$ under the exact sequence (\ref{eq:shortexact2}). We begin with a definition.

\begin{defn} A $\Gamma$-equivariant Riemannian fiber bundle is an equivariant fiber bundle $p: X \rightarrow Y$ together with an invariant metric on the relative tangent bundle $\textrm{T}(X/Y) \rightarrow X$ and an invariant projection
\begin{displaymath}
\textrm{T}X \longrightarrow \textrm{T}(X/Y)
\end{displaymath}
The choice of a $\Gamma$-equivariant Riemannian structure amounts to the choice of metric and projection.
\end{defn}

\noindent Fix a $\Gamma$-equivariant Riemannian structure for $p$. Then the metric and projection determine a connection $\nabla^{X/Y}$ for $\textrm{T}(X/Y)$. Choose a $\Gamma$-invariant, compatibly orthogonal connection $\nabla^\nu$ for $\nu$, and let $\nabla^{S(\nu)}$ be the induced connection for $S(\nu)$. Define
\begin{equation}\begin{array}{c}
\omega = \frac{1}{a(W)}\, \pi^*\hat{\alpha}\left(\textrm{T}(X/Y);\nabla^{X/Y}\right) \wedge \upsilon\left(\nu;\nabla^\nu\right) \end{array}
\end{equation}
where $a(W) \equiv a_g(W)$ is the class function defined in (\ref{eq:aeven}). Next, choose a classifying map
\begin{displaymath}
f \in \textrm{Map}_\textrm{rd}(\nu,\mathcal{B}_\Gamma)_\Gamma
\end{displaymath}
for $U$, and let $V \in \mathcal{G}_\Gamma$ so that $f(\nu) \subset \mathcal{O}_V$. Notice that
\begin{displaymath}
\nu \oplus \textrm{T}(X/Y) \simeq \overline{W},
\end{displaymath}
where $\overline{W}$ is the trivial bundle over $X$ with fiber $W$, and define
\begin{equation} \begin{array}{c}
\eta = \sigma_V\left(\pi^*S(\nu); f, \boldsymbol{\nabla}^{\pi^*S(\nu)}\right) \qquad \qquad \qquad \qquad \qquad \qquad \qquad \qquad\\
\qquad - \frac{1}{a(W)}\, \omega\left(\pi^*S(\nu);\boldsymbol{\nabla}^{\pi^*S(\nu)} \right) \wedge \pi^* \textrm{CS}_{\hat{\textrm{A}},\Gamma}\left(\overline{W};\nabla^\nu \oplus \nabla^{X/Y},d \right) \end{array}
\end{equation}
where the first term is defined as in equation (\ref{eq:cs}). The latter factor of the second term is a Chern-Simons form for the equivariant $\hat{\textrm{A}}$-polynomial; it satisfies
\begin{displaymath}
d_\Gamma\textrm{CS}_{\hat{\textrm{A}},\Gamma}\left(\overline{W};\nabla^\nu \oplus \nabla^{X/Y},d \right) = a(W) - \hat{\alpha}\left(\nu;\nabla^\nu\right) \wedge \hat{\alpha}\left(\textrm{T}(X/Y);\nabla^{X/Y}\right)
\end{displaymath}
The quadruple $(A,f,\eta,\omega)$ is a representative for a choice of a differential equivariant $K$-theory Thom class
\begin{displaymath}
\check{U} \in K_{\Gamma,\textrm{c}}^{2N}(\nu) \cong K_{\Gamma,\textrm{c}}^0(\nu)
\end{displaymath}
It is independent of the choice of $f$ and $V$, and depends only on the choice of $\nabla^\nu$ and the Riemannian structure for $p$.

The choices we have made together amount to a choice of differential equivariant $K$-theory orientation (cf. \cite{HS}):

\begin{defn} Let $p: X \rightarrow Y$ be an equivariant fiber bundle with fiber of dimension $n$, where $X$ is compact. A \emph{differential equivariant $K$-theory orientation} $\Phi$ of $p$ consists of the following data:
\begin{enumerate}
\item a topological equivariant $K$-theory orientation of $p$, namely, an equivariant embedding $\varphi$ of $X$ into a real spin $\Gamma$-representation $W$ of dimension $2N + n$, an identification of the normal bundle $\nu$ to $X \subset Y \times W$ with a tubular neighborhood of $X$, and a spin structure for $\textrm{T}(X/Y)$; and
\item a differential equivariant $K$-theoretic Thom class $\check{U} \in \check{K}_{\Gamma,\textrm{c}}^{2N}(\nu)$, namely, a refinement of the $K$-theoretic Thom class $U$ of $\nu \rightarrow X$ given by a choice of equivariant Riemannian structure for $p$ and connection $\nabla^\nu$ for $\nu$.
\end{enumerate}
\end{defn}

\noindent We shall often refer to such orientations merely as \emph{differential orientations}. We are now ready to define the pushforward map:

\begin{defn} Let $p: X \rightarrow Y$ be an equivariant fiber bundle with fiber of dimension $n$, where $X$ is compact.  Choose a differential equivariant $K$-theory orientation $\Phi = (W,\varphi,\nu,\check{U})$ of $p$.  Then the pushforward map $p_!$ is defined as the composition
\begin{displaymath}
\check{K}_\Gamma^{-i}(X) \longrightarrow \check{K}^{-i+2N}_{\Gamma,\textrm{c}}(\nu) \longrightarrow \check{K}^{-i+2N}_{\Gamma,{c}}(Y \times W) \longrightarrow \check{K}_\Gamma^{-i-n}(Y)
\end{displaymath}
where the first map is given by $\check{x} \mapsto  \check{U} \cdot \pi^*\check{x}$, the second map is given by the exact sequence of the pair $(W,\nu)$, and the last map is integration over a $\Gamma$-representation.
\end{defn}

\subsection{Properties of the pushforward map}

The pushforward map depends on the differential orientation only in the choice of Riemannian structure. This is a consequence of the following analogue of Stokes' Theorem:

\begin{prop} \label{prop:stokes} Let $p: X \rightarrow Y$ be an equivariant fiber bundle with fiber of dimension $n$, and suppose that $p$ extends to an equivariant fiber bundle $p: C \rightarrow Y$ with $\partial C = X$. Suppose that $\textrm{\emph{T}}(C/Y)$ is spin. Choose a differential equivariant $K$-theory orientation of $p: X \rightarrow Y$ and a compatible equivariant Riemannian structure for $p: C \rightarrow Y$. If $\check{x} \in \check{K}^{-i}_\Gamma(X)$ extends to a class $\check{c} \in \check{K}^{-i}_\Gamma(C)$, then $p_!(\check{x})$ is the image of a differential form under the map
\begin{displaymath}
\Omega^{-i-n-1}_\Gamma(Y) \longrightarrow \check{K}^{-i-n}_\Gamma(Y)
\end{displaymath}
given by the short exact sequence (\ref{eq:shortexact2}), and the $g$-component of this form is
\begin{equation}
\int_{C^g/Y^g} \hat{\alpha}_g\left(\textrm{\emph{T}}(C/Y);\nabla^{C/Y}\right) \wedge \omega_g(\check{c}) \cdot u^{-l_g}
\end{equation}
where
\begin{equation} \begin{array}{c}
l_g = \frac{1}{2}\left(\textrm{\emph{rk}}\,\textrm{\emph{T}}(C/Y) - \textrm{\emph{rk}}\,\textrm{\emph{T}}(C^g/Y^g)\right) \in \mathbb{Z} \end{array}
\end{equation}
and $\omega(\check{c})$ is the curvature of $\check{c}$. \end{prop}

\begin{proof} Let $\Phi = (W,\varphi,\nu,\check{U})$ be a differential equivariant $K$-theory orientation of $p: X \rightarrow Y$.  Suppose for now that there exists an embedding $\varphi: C \hookrightarrow W \times I$ which extends $\varphi$ in the following way: the preimage of $W \times \partial I$ in $C$ is $X$, and $\varphi$, restricted to $X$ and composed with projection to $W$, agrees with $\varphi$.  Let $\pi: \nu' \rightarrow C$ be the normal bundle to $C$ in $Y \times W \times I$, identified with a tubular neighborhood of $C$ in such a way that $\nu = \nu' \cap (Y \times W \times \partial I)$.

Extend $\check{U}$ to a differential Thom class $\check{U}' \in \check{K}_{\Gamma,\textrm{c}}^{2N}(\nu')$, and let $(f,\eta,\omega)$ be a triple representing $\check{U}' \cdot \pi^*\check{c}$. For $j = 0,1$, let $\check{c}_j$ denote the restriction of $\check{c}$ to $X_j$, where
\begin{displaymath}
X_j = \varphi^{-1}(W \times \lbrace j \rbrace) \subset X
\end{displaymath}
Write
\begin{displaymath}
\check{y}_j = (p|_{X_j})_!\check{c}_j \in \check{K}^{-i-n}_\Gamma(Y)
\end{displaymath}
and let $(f_j,\eta_j,\omega_j)$ be a representative triple for $\check{y}_j$. Then the ``desuspension''
\begin{displaymath}
f_W \in \textrm{Map}(Y \times I,\boldsymbol{\mathfrak{F}}_\Gamma^{-i-n})_\Gamma
\end{displaymath}
of $f$, obtained as in the previous subsection, is a homotopy from $f_0$ to $f_1$, and
\begin{displaymath}
\left(f_0,\eta_1 + \int_{Y \times I/Y} f_W^*\tilde{\omega},\omega_1\right)
\end{displaymath}
is a representative for $\check{y}_1$. Let $\iota_j$ be the inclusion of $Y \times W \times \lbrace j \rbrace$ in $Y \times W \times I$ for $j = 0,1$. Then
\begin{eqnarray}
d \int_{Y^g \times W^g \times I/Y^g}\eta_g & = & \int_{Y^g \times W^g \times I/Y^g}d\eta_g {}\nonumber\\
&& \qquad - \int_{Y^g \times W^g/Y^g} (\iota_1^*\eta_g - \iota_0^*\eta_g) {}
\end{eqnarray}
Now,
\begin{displaymath}
d\eta_g = \omega_g - f^*\tilde{\omega}_g
\end{displaymath}
and
\begin{eqnarray}
\int_{Y^g \times W^g \times I/Y^g} \omega_g & = & \int_{\nu^g/Y^g} \omega_g(\check{U}') \wedge \pi^* \omega_g(\check{c}) {}\\
& = & \frac{u^{-m_g}}{a_g(W)}\int_{C^g/Y^g} \hat{\alpha}_g\left(\textrm{T}(C/Y);\nabla^{C/Y}\right) \wedge \omega_g(\check{c}) {}\nonumber
\end{eqnarray}
while
\begin{equation}
\int_{Y^g \times W^g \times I/Y^g} f^*\tilde{\omega}_g = \frac{1}{a_g(W)}\int_{Y^g \times I/Y^g} f_W^*\tilde{\omega}_g
\end{equation}
On the other hand,
\begin{equation}
\int_{Y^g \times W^g/Y^g} (\iota_1^*\eta_g - \iota_0^*\eta_g) = \frac{1}{a_g(W)}(\eta_{1g} - \eta_{0g})
\end{equation}
It follows that
\begin{displaymath}
\eta_{1g} + \int_{Y^g \times I/Y^g}f_W^*\tilde{\omega}_g
\end{displaymath}
differs from
\begin{displaymath}
\eta_{0g} + \int_{C^g/Y^g} \hat{\alpha}_g\left(\textrm{T}(C/Y);\nabla^{C/Y}\right) \wedge \omega_g(\check{c}) \cdot u^{-l_g}
\end{displaymath}
by an exact form. Thus,
\begin{displaymath}
\left( c, \sum_{g \in \Gamma} \int_{C^g/Y^g} \hat{\alpha}_g\left(\textrm{T}(C/Y);\nabla^{C/Y}\right) \wedge \omega_g(\check{c}) \cdot u^{-l_g}, \omega_1 - \omega_0 \right),
\end{displaymath}
where $c$ is a constant map, is a representative for
\begin{displaymath}
\check{y}_1 - \check{y}_0 = p_!\check{x} \in \check{K}^{-i-n}_\Gamma(Y)
\end{displaymath}
This proves the proposition in the case that the given differential equivariant $K$-theory orientation extends to $C$ in the way we have described.

Now suppose that $W_0$ is a real orthogonal spin $\Gamma$-representation of dimension $2N_0$, and let $W_1 = W \oplus W_0$.  We may use the canonical embedding $W \hookrightarrow W_1$ to define a new differential orientation $\Phi_1$ in such a way as to leave the pushforward unchanged. First, retain the Riemannian structure for $p$. Let $\varphi_1$ be the composition of $\varphi$ with $W \hookrightarrow W_1$.  The normal bundle $\nu_0$ to the origin in $W_0$ is isomorphic to $W_0$; we identify it with the unit ball centered at the origin. Set $\nu_1 = \nu \times \nu_0$, identified with the normal bundle to $X$ in $Y \times W_1$.

Construct $\check{U}_1$ as follows.  First, let $\pi_0: \nu_0 \rightarrow \textrm{pt}$ be projection to the origin, and let $U_0$ be the equivariant $K$-theory Thom class given by $S(\nu_0)$.  Let $\mu_0$ denote the Clifford multiplication map from $\pi_0^*S^+(\nu_0)$ to $\pi_0^*S^-(\nu_0)$. Let $\check{U}_0 \in \check{K}_{\Gamma,\textrm{c}}^{2N_0}(\nu_0) \cong \check{K}_{\Gamma,\textrm{c}}^0(\nu_0)$ be the class represented by the quadruple
\begin{displaymath}
(V_0,f_0,\eta_0,\omega_0)
\end{displaymath}
where
\begin{displaymath}
f_0 \in \textrm{Map}_\textrm{rd}(\nu_0,\mathcal{O}_{V_0})_\Gamma,
\end{displaymath}
is a map representing the Thom class $U_0$,
\begin{equation}
\omega_0 = \omega\left( \pi_0^*S(\nu_0);d + \mu_0 \right),
\end{equation}
and
\begin{equation}
\eta_0 = \sigma_{V_0}\left( \pi_0^*S(\nu_0); f_0, d + \mu_0 \right)
\end{equation}
is defined as in (\ref{eq:cs}). Here, $d$ denotes the trivial connection (the exterior derivative).  Now let $\phi$ denote projection of $Y \times W_1$ onto $Y \times W$ and $\phi_0$ projection onto $W_0$, and set
\begin{equation}
\check{U}_1 = \phi^*\check{U} \cdot \phi_0^*\check{U}_0 \in \check{K}_{\Gamma,\textrm{c}}^{2(N+N_0)}(\nu_1)
\end{equation}
This class is the differential Thom class given by the choices we have made, and these data comprise a differential orientation $\Phi_1$ of $p$.  It is an immediate consequence of our choices that the pushforward given by $\Phi_1$ is the same as that given by $\Phi$.

Now, $C$ can be equivariantly embedded in some even-dimensional spin $\Gamma$-representation $V$, hence in $Y \times (W \oplus V)$.  The restriction of the embedding
\begin{displaymath}
C \hookrightarrow Y \times (W \oplus V)
\end{displaymath}
to $X$ is equivariantly homotopic to the embedding defined by composing $\varphi$ with the canonical inclusion $Y \times W \hookrightarrow Y \times (W \oplus V)$. If $s \in \mathbb{N}$ is large enough, the two induced embeddings of $X$ into $Y \times (W \oplus V)^s$ are thus equivariantly \emph{isotopic}, by the equivariant Whitney embedding theorem (cf. \cite[Proposition 1.10]{W}).  Let
\begin{displaymath}
W_0 = W^{s-1} \oplus V^s
\end{displaymath}
We take $s$ to be odd, so that $W_0$ is even-dimensional.  Construct the differential orientation $\Phi_1 = (W_1,\varphi_1,\nu_1,\check{U}_1)$ as before.  Let the embedding $X \hookrightarrow W_1$ induced by $C \hookrightarrow V$ be denoted by $\varphi_2$.  Use an equivariant isotopy from $\varphi_1$ to $\varphi_2$ to embed $X \times I$ in $W_1 \times I$, and identify $\nu_1 \times I$ with a tubular neighborhood of $X \times I$ in $Y \times W_1 \times I$ in such a way that its intersection with $Y \times W_1 \times \lbrace 0 \rbrace$ is precisely $\nu_1$.  Use the product structure to pull $\check{U}_1$ back to $\nu_1 \times I$.  Let $\Phi_2$ be the differential equivariant $K$-theory orientation for $X$ obtained by restricting these data to $Y \times W_1 \times \lbrace 1 \rbrace$.  Then $\Phi_2$ yields the same pushforward as $\Phi_1$, hence $\Phi$.

Now let $h: Y \rightarrow I$ be a defining function for $X$, so that
\begin{displaymath}
X = h^{-1}\lbrace0,1\rbrace
\end{displaymath}
Use this to construct an embedding $C \hookrightarrow W_1 \times I$.  Apply the first part of the proof to the new differential orientation $\Phi_2$.  \end{proof}

\begin{cor} The pushforward map depends on the differential orientation for $p: X \rightarrow Y$ only in the choice of equivariant Riemannian structure. \end{cor}

\begin{proof} Suppose $\Phi_j = (W_j,\varphi_j,\nu_j,\check{U}_j)$, $j = 0,1$, are two differential orientations of $p: X \rightarrow Y$ with the same Riemannian structure; denote the respective pushforward maps by $p_{0!},p_{1!}$.  Let
\begin{displaymath}
W = W_0 \oplus W_1 \oplus \mathbb{R}
\end{displaymath}
where $\Gamma$ acts on $\mathbb{R}$ trivially.  Let $(X_0,Y_0)$, $(X_1,Y_1)$ be two copies of $(X,Y)$, but imbue $\textrm{T}(X_1/Y_1)$ with the given orientation and $\textrm{T}(X_0/Y_0)$ with its opposite. Let
\begin{displaymath}
p \sqcup p: X_0 \sqcup X_1 \longrightarrow Y_0 \sqcup Y_1
\end{displaymath}
be the map defined by
\begin{equation}
(p \sqcup p)(x) = \left\{ \begin{array}{l l}
p(x) \in Y_0 & \textrm{if $x \in X_0$}\\
p(x) \in Y_1 & \textrm{if $x \in X_1$} \end{array} \right.
\end{equation}
We construct a differential orientation $\Phi$ of $p \sqcup p$ as follows.  The embedding $\varphi: X_0 \sqcup X_1 \rightarrow W$ is given by
\begin{equation}
\varphi(x) = \left\{
\begin{array}{l l}
\left(\varphi_0(x),0,-1\right) & \textrm{if $q \in X_0$}\\
\left(0,\varphi_1(x),1\right) & \textrm{if $q \in X_1$} \end{array}\right.
\end{equation}
The normal bundle $\nu$ to $X_0 \sqcup X_1$ in
\begin{displaymath}
(Y_0 \sqcup Y_1) \times W = \left((Y_0 \times W_0) \times W_1 \times \mathbb{R}\right)\, \bigsqcup\, \left(W_0 \times (Y_1 \times W_1) \times \mathbb{R}\right)
\end{displaymath}
is identified with a tubular neighborhood as
\begin{equation}
\nu = \left( \nu_0 \times B(W_1) \times (-2,0) \right) \bigsqcup \left( B(W_0) \times \nu_1 \times (0,2) \right)
\end{equation}
where $B(W_j)$ is the unit ball in $W_j$ centered at the origin.  Form the Thom class $\check{U}$ by extending $\check{U}_1$ and $-\check{U}_0$ as in the previous proof.

Let $p_!$ denote the pushforward map given by $\Phi$.  Take $\check{x} \in \check{K}_\Gamma^{-i}(X)$.  Let $\check{x}_1$ be the corresponding class in $\check{K}_\Gamma^{-i}(X_1)$ and $\check{x}_0$ the class in $\check{K}_\Gamma^{-i}(X_0)$.  We compute
\begin{equation}
p_!(\check{x}_0 + \check{x}_1) = p_{1!}(\check{x}) - p_{0!}(\check{x})
\end{equation}
Now consider the fibration
\begin{displaymath}
X \times I \longrightarrow Y
\end{displaymath}
Give $\textrm{T}I$ the standard metric with the trivial connection, and give $\textrm{T}(X \times I/Y)$ the product metric. Let $\nabla$ be the corresponding connection. By Proposition~\ref{prop:stokes}, $p_!(\check{x}_0 + \check{x}_1)$ is the image of a differential form under the short exact sequence (\ref{eq:shortexact2}), and the $g$-component of this form is
\begin{equation} \label{eq:toady}
\pm\int_{X^g \times I} \hat{\alpha}_g \left(\textrm{T}(X \times I/Y);\nabla\right) \wedge \omega_g(\pi_I^*\check{x})
\end{equation}
However,
\begin{equation}
\nabla = \pi_I^*\nabla^{X/Y} \oplus d
\end{equation}
where $\pi_I$ denotes projection onto $X$, so the form (\ref{eq:toady}) is
\begin{equation}
\pm \int_{X^g \times I} \pi_I^* \left(\hat{\alpha}_g \left(\textrm{T}(X/Y);\nabla^{X/Y}\right) \wedge \omega_g(\check{x})\right) = 0
\end{equation}
So $p_{0!}$ agrees with $p_{1!}$. \end{proof}

Next let us consider how the pushforward behaves under composition.

\begin{prop} \label{prop:composition} Suppose that $p: X \rightarrow Y$ and $q: Y \rightarrow Z$ are two compact equivariant fiber bundles, with fibers of dimension $m$ and $n$, respectively, such that the respective relative tangent bundles are spin. Fix a Riemannian structure for each of $p$ and $q$. Then $q \circ p$ has a natural Riemannian structure, and
\begin{displaymath}
q_! \circ p_! = (q \circ p)_!
\end{displaymath}
as maps
\begin{displaymath}
\check{K}_\Gamma^{-i}(X) \longrightarrow \check{K}_\Gamma^{-i-m-n}(Z)
\end{displaymath}
\end{prop}

\begin{proof} We begin by constructing the natural Riemannian structure for $q \circ p$. Let
\begin{displaymath}
P_1: \textrm{T}(X) \longrightarrow \textrm{T}(X/Y) \qquad \textrm{and} \qquad P_2: \textrm{T}(Y) \longrightarrow \textrm{T}(Y/Z)
\end{displaymath}
be the projection maps given by the respective Riemannian structures for $p$ and $q$. Notice that
\begin{displaymath}
\textrm{T}(X/Z) \cong \textrm{T}(X/Y) \oplus p^*\textrm{T}(Y/Z)
\end{displaymath}
Identify $p^*\textrm{T}Y$ with $\textrm{ker}(P_1)$; this identifies the pullback of $\textrm{T}(Y/Z)$ with a subspace of $\textrm{T}(X/Z)$. The metric on $\textrm{T}(X/Z)$ is then given by the metrics on $\textrm{T}(X/Y)$ and $\textrm{T}(Y/Z)$, and the projection
\begin{displaymath}
\textrm{T}X \longrightarrow \textrm{T}(X/Z)
\end{displaymath}
is given by its kernel, the kernel of the map $P_2 \circ dp$ restricted to $\textrm{ker}(P_1)$. This is the natural Riemannian structure for $q \circ p$.

Next, let
\begin{displaymath}
\Phi_1 = (W_1,\varphi_1,\nu_1,\check{U}_1) \qquad \textrm{and} \qquad \Phi_2 = (W_2,\varphi_2,\nu_2,\check{U}_2)
\end{displaymath}
be differential orientations for $p$ and $q$, respectively, with the given Riemannian structures. We use these data, together with the Riemannian structure for $q \circ p$, to construct a differential orientation for $q \circ p$. Set
\begin{equation}
W = W_1 \oplus W_2
\end{equation}
Then
\begin{equation}
\varphi = \varphi_1 \times (\varphi_2 \circ p)
\end{equation}
is an embedding of $X$ into $W$. The tubular neighborhood $\nu_2$ of $Y$ in $Z \times W_2$ gives a tubular neighborhood of $Y \times W_1$ in $Z \times W$. Putting this together with the neighborhood $\nu_1$ of $X$ in $Y \times W_1$, we obtain a tubular neighborhood $\nu$ of $X$ in $Z \times W$. Let $\phi_j$ be the projection of $\nu \cong \nu_1 \oplus \nu_2$ onto $\nu_j$, $j = 0,1$. The corresponding differential Thom class is
\begin{equation}
\check{U} = \phi_1^* \check{U}_1 \cdot \phi_2^* \check{U}_2
\end{equation}
These data comprise a differential orientation $\Phi = (W,\varphi,\nu,\check{U})$ for $q \circ p$.

Let $\pi_1$ be the projection map from $\nu_1$ onto $X$ and $\psi_1$ projection from $Y \times W_1$ onto $Y$; similarly, let $\pi_2$ be the projection map from $\nu_2$ onto $Y$ and $\psi_2$ projection from $Z \times W_2$ onto $Z$. We have the following commutative diagram of maps between manifolds:
\begin{equation}
\xymatrix{
X \ar[dr]_p & \nu_1 \ar[l]_{\pi_1} \ar[d]^{\psi_1} & \nu \ar[l]_{\phi_1} \ar[d]^{\phi_2} \\
& Y \ar[dr]_q & \nu_2 \ar[l]_{\pi_2} \ar[d]^{\psi_2} \\
& & Z }
\end{equation}
The data from $\Phi_1$ and $\Phi_2$ give rise to the diagram
\begin{equation}
\xymatrix{
\check{K}^{-i}_\Gamma(X) \ar[r]^\alpha & \check{K}^{-i+2M}_\Gamma(\nu_1) \ar[r]^\beta \ar[d]^\gamma & \check{K}^{-i+2(M+N)}_\Gamma(\nu) \ar[d]^\delta \\
& \check{K}^{-i-m}_\Gamma(Y) \ar[r]^\epsilon & \check{K}^{-i-m+2N}_\Gamma(\nu_2) \ar[d]^\zeta \\
& & \check{K}^{-i-m-n}_\Gamma(Z) }
\end{equation}
On the other hand, we use $\Phi$ to define
\begin{equation}
\xymatrix{
\qquad \check{K}^{-i}_\Gamma(X) \ar[r]^\theta & \check{K}^{-i+2(M+N)}_\Gamma(\nu) \ar[d]^\kappa \\
& \check{K}^{-i-m-n}_\Gamma(Z) } \qquad
\end{equation}
It is clear from our constructions that
\begin{equation}
\beta \circ \alpha = \theta
\end{equation}
Furthermore,
\begin{equation}
\zeta \circ \delta = \kappa
\end{equation}
follows from the definition of integration over a representation. It therefore remains to show that
\begin{equation}
\epsilon \circ \gamma = \delta \circ \beta
\end{equation}
This amounts to the observation that
\begin{equation}
\phi_{2!}(\phi_2^*\check{U}_2 \cdot \phi_1^*\check{x}) = \check{U}_2 \cdot \pi_2^*\,\psi_{1!}\check{x}
\end{equation}
We thus have
\begin{equation}
q_! \circ p_! = \zeta \circ \epsilon \circ \gamma \circ \alpha = \zeta \circ \delta \circ \beta \circ \alpha = \kappa \circ \theta = (q \circ p)_!
\end{equation}
which proves the proposition.\end{proof}

We have seen that the pushforward depends on the differential orientation only in the choice of Riemannian structure for $p: X \rightarrow Y$. Let us consider this dependence more closely, in the case that $Y$ is a point. Suppose, then, that $X$ is a Riemannian spin $\Gamma$-manifold; then pushforward gives a map from the differential $K$-theory of $X$ to that of a point. If $n = \textrm{dim}(X)$ is even, then the pushforward map
\begin{displaymath}
\check{K}^0_\Gamma(X) \longrightarrow \check{K}^{-n}(\textrm{pt})
\end{displaymath}
is independent of the metric. According to Proposition \ref{prop:point},
\begin{displaymath}
\check{K}^{-n}(\textrm{pt}) \cong K^{-n}(\textrm{pt}),
\end{displaymath}
so this pushforward is merely the topological pushforward given by the characteristic class. In the context of the geometric description of differential $K$-theory, it is just the index of a twisted Dirac operator. In general, however, the pushforward does depend on the metric. For instance, still considering the case $X \rightarrow \textrm{pt}$, if $X$ is odd-dimensional, then the pushforward map on $\check{K}^0_\Gamma(X)$ takes values in the torus $(R(\Gamma) \otimes \mathbb{R})/R(\Gamma)$, and depends on the metric.

Nevertheless, the pushforward is invariant under conformal variations of the metric:

\begin{prop} Suppose that $X$ is a Riemannian spin $\Gamma$-manifold of dimension $n$. Suppose that two $\Gamma$-invariant metrics $\gamma$, $\gamma'$ are conformally related, and let $p_!$, $p_!'$ denote the corresponding pushforward maps
\begin{displaymath}
\check{K}^{-i}_\Gamma(X) \longrightarrow \check{K}^{-i-n}_\Gamma(\textrm{\emph{pt}})
\end{displaymath}
Then $p_! = p_!'$. \end{prop}

\begin{proof}  It was proven by Chern and Simons in \cite{CS} that the representatives for the Pontrjagin classes of the tangent bundle constructed via the Chern-Weil method are invariant under conformal variations. The same arguments may be used to show that the representatives that we have constructed in this paper---universal polynomials in Pontrjagin forms for the relative tangent bundle with coefficients depending rationally on the eigenvalues of the group action---are conformally invariant as well.  In particular, if $\nabla$ and $\nabla'$ denote the Levi-Civita connections for $\gamma$ and $\gamma'$, respectively, then
\begin{equation}
\hat{\alpha}\left(\textrm{T}X;\nabla'\right) = \hat{\alpha}\left(\textrm{T}X;\nabla\right)
\end{equation}
Similarly,
\begin{equation}
\textrm{CS}_{\hat{\textrm{A}},\Gamma}(\textrm{T}X;\nabla,\nabla') = 0
\end{equation}
It follows that the construction of the differential Thom class is invariant under conformal variations of the metric; but the pushforward map depends on the metric only in the construction of the Thom class.  \end{proof}


\section{The reduced eta invariant}

\subsection{A conjectured analytic formula}

Suppose that $X$ is a compact Riemannian $\Gamma$-manifold with $\Gamma$-invariant metric.  Let $\textrm{Cl}_X$ be the bundle of Clifford algebras associated to the cotangent bundle, and suppose that $M$ is a bundle of $\Gamma$-equivariant $\textrm{Cl}_X$-modules with compatible metric and covariant derivative.  The Dirac operator $A$ is defined as the composition
\begin{displaymath}
\Gamma(M) \stackrel{\nabla}{\longrightarrow} \Gamma(\textrm{T}^*X \otimes M) \stackrel{\textrm{Cliff}}{\longrightarrow} \Gamma(M)
\end{displaymath}
where the second map is Clifford multiplication.

Let $H$ be the Hilbert space of $\mathcal{L}^2$-sections of $M$ over $X$.  The spectrum of $A$ is real, discrete, and has no accumulation points, and the eigenspaces of $A$ are finite dimensional subspaces of $H$.  Let $Z \subset \mathbb{R}$ denote the spectrum of $A$, with each eigenvalue listed only once.  For each $\lambda \in Z$, let $H\lbrack \lambda \rbrack$ denote the corresponding $A$-eigenspace in $H$.  Notice that $H\lbrack\lambda\rbrack$ is a finite-dimensional $\Gamma$-representation; denote its character by $\chi_\lambda$.

Given $\alpha \in \mathbb{R} \setminus Z$, we would like to associate to $A$ the element of $R(\Gamma) \otimes \mathbb{R}$ which is half the zeta function regularization of the ``character'' of the virtual $\Gamma$-representation
\begin{displaymath}
\bigoplus_{\lambda > \alpha} H\lbrack\lambda\rbrack \ominus \bigoplus_{\lambda < \alpha} H\lbrack\lambda\rbrack
\end{displaymath}
For each $g \in \Gamma$, then, define
\begin{equation}
\xi_g(\alpha)\lbrack s \rbrack \equiv \frac{1}{2}\left(\sum_{\lambda \in Z\setminus\lbrace0\rbrace} \textrm{sign}(\lambda-\alpha)|\lambda|^{-s}
\chi_\lambda(g) - \textrm{sign}(\alpha)\chi_0(g)\right)
\end{equation}
The function $s \mapsto \xi(\alpha)\lbrack s \rbrack$ is the linear combination of finitely many zeta functions, weighted by the eigenvalues of $g$, and hence has a meromorphic continuation to $0$.  Let $\xi_g(\alpha)$ denote the value of this continuation, and let $\xi_\Gamma(\alpha)$ denote the class function given by $g \mapsto \xi_g(\alpha)$.  This is the regularization we seek.

Notice that $\xi_\Gamma(\alpha)$ is independent of $\alpha$ up to characters of a finite-dimensional representations: if $\alpha < \beta$, then
\begin{equation}
\xi_\Gamma(\alpha) - \xi_\Gamma(\beta) = \sum_{\lambda \in Z \cap (\alpha,\beta)} \chi_\lambda
\end{equation}
We may thus associate to $A$ a canonically defined element $\xi_\Gamma(A)$ of the torus $(R(\Gamma) \otimes \mathbb{R})/R(\Gamma)$.  Furthermore, as $A$ varies continuously, the $\xi(\alpha)$ are piecewise continuous with jumps in $R(\Gamma)$, so $\xi_\Gamma(A)$ may be conceived of as a continuous function on a family of Dirac operators $\lbrace A_t \rbrace_{t \in T}$ (cf. \cite{FM}).

Now suppose that $X$ is an odd-dimensional spin Riemannian $\Gamma$-manifold.  Let $S_X$ be the associated complex spinor bundle with induced connection.  Suppose $V$ is a $\Gamma$-equivariant vector bundle with invariant connection $\nabla^V$ that is compatibly unitary.  Form the twisted spinor bundle $S_X \otimes V \rightarrow X$ and let $D_V$ denote the twisted Dirac operator. We have the reduced eta invariant
\begin{displaymath}
\xi_\Gamma(D_V) \in (R(\Gamma) \otimes \mathbb{R})/R(\Gamma)
\end{displaymath}
On the other hand, the triple $(V,\nabla^V,0)$ represents a class $\check{x} \in \check{K}_\Gamma^0(X)$ under the geometric description of differential equivariant $K$-theory. Let $p: X \rightarrow \textrm{pt}$. Then the pushforward map
\begin{equation}
\wp_\Gamma \equiv p_!: \check{K}^0_\Gamma(X) \longrightarrow \check{K}^{-\textrm{dim}(X)}_\Gamma(\textrm{pt})
\end{equation}
also yields an element
\begin{displaymath}
\wp_\Gamma(\check{x}) \in (R(\Gamma) \otimes \mathbb{R})/R(\Gamma)
\end{displaymath}
We conjecture:

\begin{conj} \label{conj:conjecture} Suppose $X$ is a compact odd-dimensional Riemannian spin $\Gamma$-manifold; suppose that $V$ is a $\Gamma$-equivariant vector bundle with invariant connection $\nabla^V$ that is compatibly unitary, and let $\check{x}$ be the differential equivariant $K$-theory class represented by $(V,\nabla^V,0)$.  Then
\begin{equation}
\wp_\Gamma\left(\check{x}\right) = \xi_\Gamma(D_V)
\end{equation}
as class functions modulo $R(\Gamma)$. \end{conj}

We can prove this conjecture in the case that $(X,V)$ is a boundary in an appropriate sense; then it is a consequence of Proposition~\ref{prop:stokes} and the formulae given in \cite{APS1} and \cite{D2}.  On the other hand, Klonoff has proven the conjecture in complete generality in \emph{ordinary} differential $K$-theory---the case in which $\Gamma$ is trivial---by a reduction to the boundary case. The extension of this proof to differential equivariant $K$-theory in the case of a free action is relatively straightforward. However, a reduction to the boundary case for differential equivariant $K$-theory in general remains to be established.

We proceed to these partial proofs.

\subsection{Proof for the boundary case}

Let $X$, $V$, $\nabla^V$ be as in the foregoing.  Assume for now that $X$ is the boundary of a Riemannian spin $\Gamma$-manifold $C$ and that, near the boundary, the metric on $C$ is a product and $\Gamma$-action amounts to the action on $X \times I$.  Assume further that $V$ and its connection are the restriction to $\partial C$ of a $\Gamma$-equivariant vector bundle $W$ over $Y$ with compatibly unitary, invariant connection $\nabla^W$.  Form the graded spinor bundle $S_C = S^+_C \oplus S^-_C \rightarrow C$ with compatible metric and connection, such that $S^\pm_C$ restrict to $S_X$ on $\partial C$ under the isomorphism
\begin{displaymath}
S_C^\pm|_{\partial C} \stackrel{\sim}{\longrightarrow} S_C^\mp|_{\partial C}
\end{displaymath}
given by Clifford multiplication by the outward-pointing unit normal vector $\partial_u$.  Let $D_W$ denote the twisted Dirac operator:
\begin{displaymath}
D^\pm_W: \Gamma(S^\pm_C \otimes W) \longrightarrow \Gamma(S^\mp \otimes W)
\end{displaymath}
We assume that $D_W$ decomposes near the boundary as
\begin{displaymath}
D_W^\pm = \partial_u + D_V
\end{displaymath}
Let $P \in \textrm{End}(\Gamma(S_X \otimes V))$ be projection onto the sum of eigenspaces of $D_V$ with positive eigenvalue; let $\Gamma(S_C^+ \otimes W;P)$ be the space of sections $s$ satisfying the boundary condition $P(s|_X) = 0$, and let $\Gamma(S_C^- \otimes W;1-P)$ be the space of sections satisfying the adjoint boundary condition.  Consider the restricted operators
\begin{displaymath}
\bar{D}_W^+: \Gamma(S_C^+ \otimes W;P) \longrightarrow \Gamma(S_C^- \otimes W)
\end{displaymath}
\begin{displaymath}
\bar{D}_W^-: \Gamma(S_C^- \otimes W;1-P) \longrightarrow \Gamma(S_C^+ \otimes W)
\end{displaymath}
and let $\textrm{ind}_\Gamma(D_W^+)$ be the character of the corresponding virtual representation
\begin{displaymath}
\textrm{ker}\left(\bar{D}_W^+\right) \ominus \textrm{ker}\left(\bar{D}_W^-\right)
\end{displaymath}

\begin{lem} \label{lem:integrand} We have
\begin{equation}
\textrm{\emph{ind}}_g\left(D^+_W\right) = \int_{C^g}\hat{\alpha}_g\left(\textrm{\emph{T}}C;\nabla^\textrm{\emph{lc}}\right) \wedge \omega_g\left(W;\nabla\right) - \xi_g(D_V)
\end{equation}
where $\nabla^\textrm{\emph{lc}}$ is the Levi-Civita connection for $\textrm{\emph{T}}C$. \end{lem}

\begin{proof} Let $k_t$ be the smoothing kernel for the operator $\textrm{exp}(-tD_W^2)$ on the space $\Gamma(S_C \otimes W)$, that is, a section of the exterior tensor product
\begin{displaymath}
(S_C \otimes W)^* \boxtimes (S_C \otimes W) \longrightarrow C \times C
\end{displaymath}
so that, at each $c \in C$,
\begin{displaymath}
\left(\textrm{exp}\left(-tD_W^2\right)s\right)(y) = \int_C k_t(x,y)s(x)\cdot\textrm{vol}_C(x)
\end{displaymath}
for each smooth section $s$ of $S_C \otimes W$.  The kernel for $g \cdot \textrm{exp}(-tD_W^2)$ is given by
\begin{displaymath}
\left(g \cdot \textrm{exp}\left(-tD_W^2\right)s\right)(c) = \int_C g^*(k_t(x,g \cdot c)s(x))\cdot\textrm{vol}_C(x)
\end{displaymath}
Donnelly in \cite{D2} gives the following formula for the asymptotic expansion:
\begin{equation}
\int_C \textrm{tr}_\textrm{s}\lbrace g^*k_t(x,g \cdot x) \rbrace \,\textrm{vol}_C(x) \sim \frac{1}{(4\pi t)^{n/2}}\sum_{j=0}^\infty t^j\int_{C^g}b_{g,j}(x)\cdot\textrm{vol}_{C^g}(x)
\end{equation}
The main theorem of \cite{D2} yields
\begin{equation} \label{equation:john}
\textrm{ind}_g\left(D_W^+\right) =  \left(\frac{1}{(4\pi)^{n/2}}\int_{C^g} b_{g,n/2}(x)\cdot\textrm{vol}_{C^g}(x)\right) - \xi_g(D_V)
\end{equation}
In order to derive an expression for this integral it is necessary to identify the integrand as a polynomial of characteristic forms.  The papers \cite{D1}, \cite{D2}, and \cite{DP} of Donnelly and Patodi handle this problem in some detail, but in the case of the signature operator. However, the computation of the integrand is a local problem, so the formula for the $\Gamma$-index derived in \cite[$\S6.4$]{BGV} for the case $\partial C = \varnothing$ applies to the general case as well. The integrand there computed coincides with our
\begin{displaymath}
\hat{\alpha}_g\left(\textrm{T}C;\nabla^\textrm{lc}\right) \wedge \omega_g\left(W;\nabla\right)
\end{displaymath}
The proof follows.  \end{proof}

\begin{prop} \label{prop:boundary} If a compact odd-dimensional Riemannian spin $\Gamma$-manifold $X$ is the boundary of a spin Riemannian $\Gamma$-manifold $C$ such that, near the boundary, the metric on $C$ is a product, and $\Gamma$-action amounts to the action on $X \times I$, then Conjecture~\ref{conj:conjecture} holds.
\end{prop}

\begin{proof} If $\check{c} \in \check{K}_\Gamma^0(C)$ restricts to $\check{x}$ on $X$, then, by Proposition~\ref{prop:stokes} and Lemma~\ref{lem:integrand},
\begin{equation}
\xi_g(D_V) + \textrm{ind}_g(D_W^+) = \int_{C^g} \hat{\alpha}_g\left(\textrm{T}C;\nabla^\textrm{lc}\right) \wedge \omega_g(W;\nabla^W) = \wp_g(\check{x})
\end{equation}
It follows that $\xi_\Gamma(D_V)$ differs from $\wp_\Gamma(\check{x})$ by a character, namely, $\textrm{ind}_\Gamma(D_V)$, so they are equal in $(R(\Gamma) \otimes \mathbb{R})/R(\Gamma)$.  \end{proof}

\subsection{A Fubini theorem}

Before proceeding to Klonoff's proof of Conjecture \ref{conj:conjecture} for ordinary differential $K$-theory, we prove a ``Fubini theorem'' for pushforward to the differential equivariant $K$-theory of a point.

\begin{prop} \label{prop:fubini} Suppose that $X_1$, $X_2$, are two compact spin $\Gamma$-manifolds with $\Gamma$-invariant Riemannian structure. If $\check{x}_j \in \check{K}_\Gamma^0(X_j)$, $j = 1,2$, then
\begin{equation}
\wp_\Gamma(\phi_1^*\check{x}_1 \cdot \phi_2^* \check{x}_2) = \wp_\Gamma(\check{x}_1) \cdot \wp_\Gamma(\check{x}_2)
\end{equation}
where $\phi_j$ is the projection map from $X_1 \times X_2$ onto $X_j$. \end{prop}

\begin{proof}  If $X_1$ and $X_2$ are both even-dimensional, then the pushforward is given by the topological pushforward of the characteristic classes, and the proof follows from the Fubini theorem for ordinary equivariant $K$-theory. We assume, therefore, that $X_1$ is odd-dimensional  and $X_2$ even-dimensional.

Now suppose $\textrm{dim}(X)$ is odd and $\textrm{dim}(Y)$ is even.  Choose differential equivariant $K$-theory orientations $\Phi_1 = (W_1,\varphi_1,\nu_1,\check{U}_1)$ and $\Phi_2 = (W_2,\varphi_2,\nu_2,\check{U}_2)$ for $X_1$ and $X_2$, respectively.  Let $(V_1,f_1,\eta_1,\omega_1)$ be a quadruple representing $\pi_1^*\check{x}_1 \cdot \check{U}_1$, where $f_1$ is the constant map; then
\begin{equation}
\wp_g(\check{x}_1) = a_g(W_1) \int_{W_1^g}\eta_{1g} \cdot u^{m_{1g}}
\end{equation}
where $a_g(W_1)$ is defined as in (\ref{eq:aodd}). On the other hand, if $(V_2,f_2,\eta_2,\omega_2)$ is a quadruple representing $\pi_2^*\check{x}_2 \cdot \check{U}_2$, then
\begin{equation}
\wp_g(\check{x}_2) = a_g(W_2) \int_{W_2^g}\omega_{2g} \cdot u^{m_{2g}}
\end{equation}
where $a_g(W_2)$ is defined as in (\ref{eq:aeven}). Let $W = W_1 \oplus W_2$.  Then the product $\check{U}$ of $\check{U}_1$ and $\check{U}_2$, pulled back to $W$ via the natural projection maps, is a differential equivariant Thom form for the induced embedding $X \times Y \hookrightarrow W$.  We use this differential orientation to compute the pushforward of $\phi_1^*\check{x}_1 \cdot \phi_2^* \check{x}_2$.  The class  $\pi^*(\phi_1^*\check{x}_1 \cdot \phi_2^* \check{x}_2) \cdot \check{U}$ is represented by the quadruple
\begin{displaymath}
\left[ \begin{array}{c}
L(V_1 \otimes V_2)\\
m \circ (f_1 \times f_2)\\
f_1^*\tilde{\omega}_{V_1} \wedge \eta_2 + \eta_1 \wedge \omega_2 + (f_1 \times f_2)^*\kappa\\
\omega_1 \wedge \omega_2 \end{array} \right]
\end{displaymath}
where $L$, $m \equiv m_{0,0}$, and $\kappa \equiv \kappa^{0,0}_{V_1,V_2}$ are defined as in the discussion of ring structure.  Let $F$ be a homotopy from $m \circ (f_1 \times f_2)$ to the constant map $i$, constructed from a homotopy $F_2$ from $f_2$ to $i$.  Since $\textrm{dim}(W)$ is odd,
\begin{equation}
\wp_g(\phi_1^*\check{x}_1 \cdot \phi_2^* \check{x}_2) = a_g(W) \int_{W^g} \left( \eta_g - \int_{I \times W^g/W^g} (F^*\tilde{\omega}_{L(V_1 \otimes V_2)})_g \right) \cdot u^{m_{1g}+m_{2g}}
\end{equation}
It may easily be checked that $a_g(W) = a_g(W_1) \cdot a_g(W_2)$.  Now, $f_1$ is a constant map, so $f_1^*\tilde{\omega}_{V_1} = 0$ and the first term of $\eta_g$ is zero.  On the other hand,
\begin{equation}
0 = f_1^*\tilde{\omega}_{V_1} \wedge F_2^*\tilde{\omega}_{V_2} = F^*\tilde{\omega}_{L(V_1 \otimes V_2)} + d_\Gamma (f_1 \times F_2)^*\kappa
\end{equation}
It follows that
\begin{equation}
\int_{I \times W^g/W^g} F^*(\tilde{\omega}_{L(V_1 \otimes V_2)})_g = -\int_{I \times W^g/W^g} d(f_1 \times F_2)^*\kappa_g = (f_1 \times f_2)^*\kappa_g
\end{equation}
Thus,
\begin{equation}
\wp_g(\phi_1^*\check{x} \cdot \phi_2^* \check{y}) = a_g(W) \int_{W^g} \eta_{1g} \wedge \omega_{2g} \cdot u^{m_{1g}+m_{2g}}
\end{equation}
The proof follows from an application of the Fubini Theorem for integration of forms.

Finally, suppose that $\textrm{dim}(X)$ and $\textrm{dim}(Y)$ are both odd.  Choose $\Phi_1$, $\Phi_2$ as before.  Since the $W_j$ are odd-dimensional and the $\omega(\check{U}_j)$ are even forms, it follows from the Fubini Theorem for differential forms that
\begin{equation}
\int_{W_1 \oplus W_2} \omega(\check{U}_1 \cdot \check{U}_2) = \int_{W_1} \omega(\check{U}_1) \cdot \int_{W_2} \omega(\check{U}_2) = 0
\end{equation}
This completes the proof.  \end{proof}

\subsection{Reduction to the boundary case for $\Gamma$ trivial}

Recall that the ordinary differential $K$-theory ring for a smooth manifold $X$ is defined as $\check{K}^\bullet(X) \equiv \check{K}_{\lbrace e \rbrace}^\bullet(X)$, where ${\lbrace e \rbrace}$ denotes the trivial group. If $X$ is spin, then we set $\wp \equiv \wp_{\lbrace e \rbrace}$.

\begin{prop} \label{prop:ordinary} Conjecture~\ref{conj:conjecture} holds for ordinary differential $K$-theory. \end{prop}

\noindent The line of argumentation we follow in the proof of this is due to K. Klonoff; it employs as a lemma the following boundary theorem of M. Hopkins:

\begin{lem} \label{lem:ks} Suppose $V$ is a vector bundle over an odd-dimensional spin manifold $X$.  Let $H \rightarrow \mathbb{CP}^1$ denote the Hopf hyperplane bundle.  Then there is a natural number $s$ so that the vector bundle
\begin{displaymath}
V \boxtimes H^s \longrightarrow X \times (\mathbb{CP}^1)^s,
\end{displaymath}
where $H^s$ is the exterior tensor product of $s$ copies of $H \rightarrow \mathbb{CP}^1$, is a boundary, in the sense that there exists a spin manifold $Y$ with $\partial Y = X \times (\mathbb{CP}^1)^s$ and a vector bundle $W \rightarrow Y$ such that
\begin{displaymath}
W|_{\partial Y} \simeq V \boxtimes H^s
\end{displaymath}
\end{lem}

\begin{proof} See \cite{K}.
\end{proof}

\begin{proof}[Proof of Proposition \ref{prop:ordinary}.] Suppose that $X$ is a compact odd dimensional spin manifold, and $V$ a vector bundle over $X$ with unitary connection $\nabla^V$.  These data determine a class $\check{x} \in \check{K}^0(X)$.  Let $H \rightarrow \mathbb{CP}^1$ be the Hopf hyperplane bundle, and let $s \in \mathbb{N}$ so that $V \boxtimes H^s$ is a boundary in the sense of Lemma~\ref{lem:ks}.  We claim that
\begin{equation}\label{eq:equiv}
\xi(D_{V \boxtimes H^s}) = \xi(D_V)
\end{equation}
in $\mathbb{R}/\mathbb{Z}$. To see this, notice that
\begin{equation}
\textrm{ker}(D_{V \boxtimes H^s}) = \textrm{ker}(D_{V \boxtimes H^{s-1}}) \otimes \left(\textrm{ker}(D_H^+) \oplus \textrm{ker}(D_H^-)\right)
\end{equation}
Also,
\begin{equation}
\textrm{dim}\,\textrm{ker}(D_H) = 1 + 2\,\textrm{dim}\,\textrm{ker}(D_H^-)
\end{equation}
since $\textrm{ind}(D_H) = 1$, so
\begin{equation} \begin{array}{c}
\frac{1}{2}(\textrm{dim}\,\textrm{ker}(D_{V \boxtimes H^s}) - \textrm{dim}\,\textrm{ker}((D_{V \boxtimes H^{s-1}})) = \textrm{dim}\,\textrm{ker}(D_H^-) \in \mathbb{Z} \end{array}
\end{equation}
Furthermore, the eta invariants satisfy
\begin{equation}
\eta(D_{V \boxtimes H^s}) = \textrm{ind}(D_H)\,\eta(D_{V \boxtimes H^{s-1}}) = \eta(D_{V \boxtimes H^{s-1}})
\end{equation}
(cf. \cite{APS3}).  It follows that
\begin{equation}
\xi(D_{V \boxtimes H^s}) - \xi(D_{V \boxtimes H^{s-1}}) \in \mathbb{Z}
\end{equation}
Continuing in this manner, we deduce (\ref{eq:equiv}).

On the other hand, the connection on $V \boxtimes H^s$ induced by the connections on $V$ and $H$ determines a class
\begin{displaymath}
\check{x} \boxtimes \check{\tau}^s \in \check{K}^0(X \times \mathbb{CP}^1)
\end{displaymath}
It should be clear that
\begin{equation}
\check{x} \boxtimes \check{\tau}^s = \pi^*\check{x} \cdot \prod_{i=1}^s \pi_i^*\check{\tau}
\end{equation}
where $\pi$ and $\pi_i$ denote projection onto $X$ and the $i^\textrm{th}$ factor of $(\mathbb{CP}^1)^s$, respectively.  It follows from Proposition~\ref{prop:fubini} that
\begin{equation}
\wp(\check{x} \boxtimes \check{\tau}^s) = \textrm{ind}(D_H)^s\, \wp(\check{x}) = \wp(\check{x})
\end{equation}
Thus, by Proposition~\ref{prop:boundary},
\begin{equation}
\wp(\check{x}) = \wp\left(\check{x} \boxtimes \check{\tau}^s\right) = \xi(D_{V \boxtimes H^s}) = \xi(D_V)
\end{equation}
We conclude that Conjecture~\ref{conj:conjecture} holds in ordinary differential $K$-theory. \end{proof}

\subsection{Proof for the free case}

It is now relatively straightforward to prove Conjecture~\ref{conj:conjecture} for the case of a free action by reducing to ordinary differential $K$-theory.

\begin{prop} \label{prop:free2} Conjecture~\ref{conj:conjecture} holds if $\Gamma$ acts freely. \end{prop}

\noindent We begin with a lemma.

\begin{lem} \label{lem:freelemma} Suppose $X$ is a compact Riemannian spin $\Gamma$-manifold with free $\Gamma$-action. Let $Y = X/\Gamma$, where $\Gamma$ acts trivially on $Y$, and let $p: X \rightarrow Y$ be the quotient map. If $\check{y} \in \check{K}_\Gamma^0(Y)$, then
\begin{equation}
p_!(p^*\check{y}) = \chi_\Gamma \cdot \check{y}
\end{equation}
where $\chi_\Gamma$ is the character of the regular representation. \end{lem}

\begin{proof} Let $\varphi: X \hookrightarrow W$ be an embedding of $X$ in a real spin $\Gamma$-representation of dimension $2N$. Let $B_w(r) \subset W$ be the ball of radius $r$ centered at $w$. Fix $r$ small enough that
\begin{displaymath}
B_{\varphi(x)}(r) \cap B_{\varphi(x')}(r) = \varnothing
\end{displaymath}
for all $x,x' \in X$ such that $p(x) = p(x')$. Then
\begin{equation}
\nu = \lbrace (p(x),w) \in Y \times W: x \in X, w \in B_{\varphi(x)}(r) \rbrace
\end{equation}
is a tubular neighborhood for $X$ in $Y \times W$, and $\nu_Y = Y \times B_0(r)$ is a tubular neighborhood for $Y$ in $Y \times W$. Let $\check{U}_Y \in \check{K}_\Gamma^{2N}(\nu_Y)$ be a differential equivariant $K$-theoretic Thom class for $\pi_Y: \nu_Y \rightarrow Y$; then the pullback $\check{U}$ of $\check{U}_Y$ to $\nu$ under the obvious quotient map is a Thom class for $\pi: \nu \rightarrow X$.

Let $(f_Y,\eta_Y,\omega_Y)$ be a representative for $\check{U}_Y \cdot \pi_Y^*\check{y}$. Then the pullback $(f,\eta,\omega)$ of this triple to $\nu$ is a representative for $\check{U} \cdot \pi^*p^*\check{y}$. Clearly,
\begin{equation}
\psi_!\omega = \psi_!\omega_Y \cdot \chi_\Gamma \qquad \psi_!\eta = \psi_!\eta_Y \cdot \chi_\Gamma
\end{equation}
where $\psi$ is the projection map $Y \times W \rightarrow W$. Furthermore, notice that, for each $y \in Y$, $f$ restricts to an equivariant map from $D^{2N} \times \Gamma$ into $\boldsymbol{\mathfrak{F}}_\Gamma^0$ which takes each copy of $\partial D^{2N}$ into the space of invertible operators and agrees with $f_Y$ on each copy of $D^{2N}$. Thus,
\begin{equation}
\psi_!(\check{U} \cdot \pi^*p^*\check{y}) = \psi_!(\check{U}_Y \cdot \pi_Y^*\check{y}) \cdot \chi_\Gamma
\end{equation}
Furthermore,
\begin{equation}
\psi_!(\check{U}_Y \cdot \pi_Y^*\check{y}) = \check{y}
\end{equation}
The proof follows. \end{proof}

\begin{proof}[Proof of Proposition~\ref{prop:free2}] Suppose that $X$ is a compact connected odd dimensional Riemannian spin $\Gamma$-manifold with free $\Gamma$-action, and let $V$, $\nabla$, $\check{x}$ be as in the statement of the conjecture.  Let $Y = X/\Gamma$.  Recall that each unitary representation of
\begin{displaymath}
\pi_1(Y)/\pi_1(X) \cong \Gamma
\end{displaymath}
determines a flat vector bundle over the quotient space $Y$.  Given a character $\chi$, let $V_\chi \rightarrow Y$ be the flat bundle determined by the $\Gamma$-representation $W_\chi$ corresponding to $\chi$.  It satisfies
\begin{displaymath}
p^*V_\chi \simeq X \times W_\chi
\end{displaymath}
where $p$ is the quotient map.  Give $V_\chi$ the flat unitary connection $\nabla^\chi$, and let $D_{V/\Gamma \otimes V_\chi}$ be the Dirac operator $D_{V/\Gamma}$ twisted by $(V_\chi,\nabla^\chi)$.  Summing over the characters of the irreducible representations, we decompose each $\lambda$-eigenspace $H\lbrack \lambda \rbrack$ for $D_V$ as
\begin{eqnarray}
H\lbrack \lambda \rbrack & \cong & \bigoplus_\textrm{$\chi$ irred.} \textrm{Hom}_\Gamma\left( H \lbrack \lambda \rbrack, W_\chi \right) \otimes W_\chi {}\nonumber\\
& \cong & \bigoplus_\textrm{$\chi$ irred.} \left( H \lbrack \lambda \rbrack \otimes W_\chi^* \right)^\Gamma \otimes W_\chi {}\nonumber\\
& \cong & \bigoplus_\textrm{$\chi$ irred.} \left( H \lbrack \lambda \rbrack \otimes W_\chi \right)^\Gamma \otimes W_{\bar{\chi}} {}
\end{eqnarray}
where we have used the fact that $W_\chi^* \cong W_{\bar{\chi}}$.  Notice that
\begin{displaymath}
\left( H \lbrack \lambda \rbrack \otimes W_\chi \right)^\Gamma
\end{displaymath}
is precisely the $\lambda$-eigenspace of $D_{V/\Gamma \otimes V_\chi}$.  It follows that
\begin{equation} \label{eq:freeeta}
\xi_\Gamma \left(D_V\right) = \sum_\textrm{$\chi$ irred.} \xi \left(D_{V/\Gamma \otimes V_\chi} \right) \cdot \bar{\chi}
\end{equation}
in
\begin{displaymath}
(R(\Gamma) \otimes \mathbb{R})/R(\Gamma) \cong R(\Gamma) \otimes \mathbb{R}/\mathbb{Z}
\end{displaymath}
(These arguments are drawn largely from \cite{APS2}.)

Next, let $\langle\,,\rangle$ denote the inner product on the class functions on $\Gamma$ defined by
\begin{displaymath}
\langle\alpha,\beta\rangle = \frac{1}{|\Gamma|} \sum_{g \in \Gamma} \alpha(g)\bar{\beta}(g)
\end{displaymath}
Under this inner product, the characters of the irreducible representations constitute an orthonormal basis for $R(\Gamma) \otimes \mathbb{R}$. We claim that the diagram
\begin{equation}\label{eq:freediag}
\xymatrix{
\check{K}_\Gamma^0(X) \ar[r]^\sim \ar[d]^{\wp_\Gamma} & \check{K}^0(Y) \ar[d]^{\wp}\\
\frac{R(\Gamma) \otimes \mathbb{R}}{R(\Gamma)} \ar[r] & \mathbb{R}/\mathbb{Z} }
\end{equation}
commutes, where the upper map is given by the isomorphism of Proposition~\ref{prop:free}, and the lower map is orthogonal projection
\begin{displaymath}
\rho \mapsto \langle \rho,\chi_1 \rangle
\end{displaymath}
onto the subspace spanned by the one-dimensional trivial representation $\chi_1$. Recall first that
\begin{displaymath}
\check{K}_\Gamma^0(Y) \cong \check{K}^0(Y) \otimes R(\Gamma)
\end{displaymath}
by Proposition~\ref{prop:trivial}; if $\check{y} \cdot \chi \in \check{K}^0(Y) \otimes R(\Gamma)$, then
\begin{equation}
\wp_\Gamma(\check{y} \cdot \chi) = \wp(\check{y}) \cdot \chi
\end{equation}
in $\check{K}_\Gamma^{-\textrm{dim}(X)}(\textrm{pt})$. On the other hand, the isomorphism
\begin{displaymath}
\check{K}^0(Y) \stackrel{\sim}{\longrightarrow} \check{K}_\Gamma^0(X)
\end{displaymath}
is given by
\begin{displaymath}
\check{y} \mapsto p^*(\check{y} \cdot \chi_1)
\end{displaymath}
Thus, given $\check{x} = p^*(\check{y} \cdot \chi_1) \in \check{K}_\Gamma^0(X)$, we have
\begin{eqnarray}
\wp_\Gamma(\check{x}) & = & \wp_\Gamma(p_!\check{x}) {}\nonumber\\
& = & \wp_\Gamma(p_!\,p^*(\check{y} \cdot \chi_1)) {}\nonumber\\
& = & \wp_\Gamma(\check{y} \cdot \chi_\Gamma) {}\nonumber\\
& = & \wp(\check{y}) \cdot \chi_\Gamma {}
\end{eqnarray}
where the first equality follows from Proposition~\ref{prop:composition} and the third from Lemma~\ref{lem:freelemma}. The commutativity of (\ref{eq:freediag}) follows.

The data $(V_\chi,\nabla^\chi)$ determine a differential $K$-theory class
\begin{displaymath}
\check{v}_\chi \in \check{K}^0(Y) \cong \check{K}^0_\Gamma(X),
\end{displaymath}
and it follows from Proposition~\ref{prop:fubini} that
\begin{equation}
\wp_\Gamma(\check{v}_\chi \cdot \check{x}) = \wp_\Gamma(\check{x}) \cdot \chi
\end{equation}
Thus,
\begin{equation}
\langle \wp_\Gamma(\check{x}),\bar{\chi} \rangle = \langle \wp_\Gamma(\check{v}_\chi \cdot \check{x}), \chi_1 \rangle = \wp(\check{v}_\chi \cdot \check{x})
\end{equation}
We deduce that
\begin{equation}\label{eq:freepush}
\wp_\Gamma(\check{x}) = \sum_\textrm{$\chi$ irred.} \wp(\check{v}_\chi \cdot \check{x}) \cdot \bar{\chi}
\end{equation}
By Proposition~\ref{prop:ordinary},
\begin{equation}
\wp(\check{v}_\chi \cdot \check{x}) = \xi(D_{V/\Gamma \otimes V_\chi})
\end{equation}
Combining this with equations (\ref{eq:freeeta}) and (\ref{eq:freepush}), we have
\begin{equation}
\wp_\Gamma(\check{x}) = \xi_\Gamma(D_V)
\end{equation}
So Conjecture~\ref{conj:conjecture} holds.
\end{proof}


\end{document}